\documentclass[a4paper,11pt]{article}

\usepackage[sep]{optional}
\usepackage{ucs}
\usepackage[utf8x]{inputenc}
\usepackage{type1ec}
\usepackage[utf8x]{inputenc}

\usepackage{hyperref}
\usepackage[frenchb]{babel}
\usepackage{fullpage}
\usepackage{imakeidx}

\usepackage{bbm}
\usepackage[bbgreekl]{mathbbol}
\usepackage{nccrules}
\usepackage{tocbibind}
\usepackage[intlimits,leqno]{amsmath}
\usepackage{amsthm}
\usepackage{amssymb}
\usepackage{mathrsfs}
\usepackage{stmaryrd}
\usepackage{xspace}
\usepackage[all]{xy}
\newdir{ >}{{}*!/-10pt/\dir{>}}

\newcommand{\Z}{\ensuremath{\mathbb{Z}}}
\newcommand{\Q}{\ensuremath{\mathbb{Q}}}
\newcommand{\R}{\ensuremath{\mathbb{R}}}
\newcommand{\C}{\ensuremath{\mathbb{C}}}

\newcommand{\A}{\ensuremath{\mathbb{A}}}

\newcommand{\Resprod}{\ensuremath{{\prod}'}}

\newcommand{\dd}{\ensuremath{\,\mathrm{d}}}

\newcommand{\angles}[1]{\ensuremath{\langle #1 \rangle}}
\newcommand{\mes}{\ensuremath{\mathrm{mes}}}

\newcommand{\identity}{\ensuremath{\mathrm{id}}}

\newcommand{\Hom}{\ensuremath{\mathrm{Hom}}}

\newcommand{\rightiso}{\ensuremath{\stackrel{\sim}{\rightarrow}}}

\newcommand{\Ker}{\ensuremath{\mathrm{Ker}\,}}

\newcommand{\Lie}{\ensuremath{\mathrm{Lie}\,}}
\newcommand{\Ad}{\ensuremath{\mathrm{Ad}\,}}

\newcommand{\Spec}{\ensuremath{\mathrm{Spec}\,}}

\newcommand{\Gm}{\ensuremath{\mathbb{G}_\mathrm{m}}}
\newcommand{\Ga}{\ensuremath{\mathbb{G}_\mathrm{a}}}

\newcommand{\Supp}{\ensuremath{\mathrm{Supp}}}

\newcommand{\GL}{\ensuremath{\mathrm{GL}}}

\theoremstyle{plain}
\newtheorem{proposition}{Proposition}[subsection]
\newtheorem{lemma}[proposition]{Lemme}
\newtheorem{theorem}[proposition]{Théorème}
\newtheorem{corollary}[proposition]{Corollaire}

\theoremstyle{definition}
\newtheorem{definition}[proposition]{Définition}
\newtheorem{definition-theorem}[proposition]{Définition-Théorème}
\newtheorem{definition-proposition}[proposition]{Définition-Proposition}
\newtheorem{hypothesis}[proposition]{Hypothèse}
\newtheorem{example}[proposition]{Exemple}

\newtheorem{remark}[proposition]{Remarque}

\newcommand{\bmu}{\ensuremath{\bbmu}}
\newcommand{\bomega}{{\ensuremath{\boldsymbol{\omega}}}}
\newcommand{\noyau}{\ensuremath{\boldsymbol{\varepsilon}}} 
\newcommand{\rev}{\ensuremath{\mathbf{p}}} 
\newcommand{\asp}{\ensuremath{\dashrule[.7ex]{2 2 2 2}{.4}}} 

\title{La formule des traces pour les revêtements de groupes réductifs connexes. I. \\ Le développement géométrique fin}
\author{Wen-Wei Li}
\date{}

\makeindex[name=iFT1,title=Index,columns=3]
\begin{document}

\maketitle

\begin{abstract}
  On étudie la partie spécifique de la formule des traces d'Arthur-Selberg pour certains revêtements des groupes réductifs connexes. Comme un premier pas vers la formules des traces invariante, on exprime le côté géométrique en termes des intégrales orbitales pondérées. Les résultats s'appliquent, en particulier, aux revêtements construits par Brylinski et Deligne.
\end{abstract}

\begin{flushleft}
  \small MSC classification (2010): \textbf{11F72}, 11F70.
\end{flushleft}

\tableofcontents

\opt{these}{\chapter{Le développement géométrique fin}}
\section{Introduction}
\paragraph{Motivation}
La théorie des représentations automorphes des groupes réductifs connexes a depuis longtemps été l'objet de travaux intensifs, et la formule des traces d'Arthur-Selberg s'est avérée un outil indispensable. Or certaines questions arithmétiques nous obligent à considérer non seulement les groupes réductifs connexes, mais aussi leurs revêtements finis qui ne sont pas algébriques. Cet article fait partie d'un projet consistant à étendre les travaux d'Arthur aux revêtements.

Historiquement, Flicker et Kazhdan \cite{F80,KF86} ont déjà utilisé une forme simple de la formule des traces sur les revêtements métaplectiques de $\GL(n)$. Mezo \cite{MZ02,MZ01} reprenait leur travail à l'aide de la formule des traces invariante d'Arthur. Malheureusement ils ne considèrent pas les autres revêtements. De plus, vu la profondeur des travaux d'Arthur sur la formule des traces invariante \cite{Ar88-TF1,Ar88-TF2}, les justifications données dans \cite{MZ02} ne sont peut-être pas suffisantes.

Passons en revue le cas des groupes réductifs connexes. Soient $F$ un corps de nombres et $G$ un $F$-groupe réductif connexe. Désignons $\A = \A_F$ l'anneau des adèles de $F$. On sait définir le sous-groupe $G(\A)^1 \subset G(\A)$  (voir \S\ref{sec:app-HC-global}). Fixons un sous-groupe de Lévi minimal $M_0$ et un sous-groupe compact maximal $K$ de $G(\A)$ en bonne position relativement à $M_0$. \textit{Grosso modo}, la formule des traces ``grossière'' d'Arthur \cite{Ar78,Ar80} est une égalité des fonctionnelles linéaires sur $C_c^\infty(G(\A)^1)$ (que l'on appelle aussi ``distributions'', par abus de terminologie)
$$ J := \sum_{\mathfrak{o}} J_\mathfrak{o} = \sum_{\chi} J_\chi , $$
où $\mathfrak{o}$ (resp. $\chi$) indexe des données géométriques (resp. spectrales). Les données $\mathfrak{o}$ sont faciles à décrire: elles correspondent aux classes de conjugaison semi-simples dans $G(F)$. Le terme correspondant à la classe $\{1\}$ est noté $J_\text{unip}$ et s'appelle le terme unipotent. Les données $\chi$ correspondent, en gros, aux représentations automorphes cuspidales sur les sous-groupes de Lévi semi-standards modulo l'action du groupe de Weyl de $G$. Cette formule grossière est relativement facile à adapter aux revêtements (voir \S\ref{sec:formule-grossiere-revetement}).

Avant d'obtenir la formule des traces invariante, il faut d'abord développer les distributions $J_\mathfrak{o}$ (resp. $J_\chi$) en termes des intégrales orbitales pondérées (resp. caractères pondérés), et le résultat s'appelle le développement géométrique (resp. spectral) fin. En sommant les développements fins pour chaque $J_\mathfrak{o}$, le développement géométrique fin  prend la forme (cf. \cite{Ar86})
$$ J(f) = \sum_M |W_0^M| |W_0^G|^{-1} \sum_{\gamma \in (M(F))_{M,S}} a^M(S, \dot{\gamma}) J_M(\dot{\gamma}, f), \quad f \in C_c^\infty(G(\A)^1), $$
où
\begin{itemize}
  \item $M$ parcourt les sous-groupes de Lévi semi-standards de $G$;
  \item $W_0^M$ est le groupe de Weyl de $M$;
  \item $S$ est un ensemble fini de places de $F$ suffisamment grand relativement à $f$;
  \item $(M(F))_{M,S}$ est l'ensemble des classes de $(M,S)$-équivalence (voir \ref{def:(M,S)-equiv});
  \item le symbole pointé $\dot{\gamma}$ signifie que l'on choisit une mesure invariante sur la classe de conjugaison de $\gamma$ dans $M(F_S)$;
  \item $a^M(S,\dot{\gamma})$ s'appelle le coefficient de ce développement en $\dot{\gamma}$, qui est un objet global;
  \item $J_M(\dot{\gamma}, f)$ est l'intégrale orbitale pondérée de $f$ en $\dot{\gamma}$, qui est un objet local.
\end{itemize}

C'est ce qui est problématique pour les revêtements. Le principal but de cet article est un développement géométrique fin pour les revêtements. Quoique le résultat \ref{prop:developpement-raffine-geometrique} a l'air très similaire, sa formulation ainsi que sa démonstration nécessitent des modifications inattendues. Précisons.

\paragraph{Revêtements}
Avant d'entamer ce projet, il faut bien sûr signaler une classe convenable de revêtements. Soit $F$ un corps local ou global, toujours supposé de caractéristique nulle dans cet article. Soit $G$ un $F$-groupe réductif connexe. Notons $A := F$ si $F$ est local et $A := \A$ si $F$ est global, alors $G(A)$ est muni d'une topologie déduite de celle de $A$. En premier lieu, on considère des extensions centrales finies topologiques de $G(A)$, à savoir
$$ 1 \to N \to \tilde{G} \stackrel{\rev}{\to} G(A) \to 1 $$
où $N$ est un groupe abélien fini. Les représentations de $\tilde{G}$ se décomposent selon les caractères de $N$. On fixe un tel caractère $\xi: N \to \bmu_m$, où $m \in \Z_{\geq 1}$ et $\bmu_m := \{\noyau \in \C^\times : \noyau^m=1 \}$. On pousse l'extension centrale en avant via $\xi$. On s'est ainsi ramené aux revêtements avec $N = \bmu_m$, ce que l'on suppose dorénavant, et les représentations sur lesquelles $\bmu_m$ (regardé comme un sous-groupe de $\tilde{G}$) opère par $\noyau \mapsto \noyau \cdot \identity$. De telles représentations sont dites spécifiques. Pour l'étude des représentations spécifiques, il suffit de considérer les fonctions $f$ sur $\tilde{G}$ telles que $f(\noyau\tilde{x})=\noyau^{-1}f(\tilde{x})$ pour tout $\noyau \in \bmu_m$, $\tilde{x} \in \tilde{G}$; de telles fonctions sont dites anti-spécifiques.

On montrera qu'un revêtement se scinde de façon canonique au-dessus des sous-groupes unipotents. Lorsque $F$ est global, on suppose de plus qu'un scindage au-dessus de $G(F)$ est fixé. Tel est le formalisme posé dans \cite{MW94}; on dispose alors de la théorie de décomposition spectrale et des séries d'Eisenstein. Mentionnons aussi que d'un revêtement de $G(\A)$ se déduisent des revêtements de $G(F_v)$, où $v$ est une place de $F$, en prenant la fibre de $\rev$ au-dessus de $G(F_v)$.

Or ces hypothèses ne suffisent pas dans le cas adélique. Par exemple, pour avoir un théorème de décomposition tensorielle des représentations lisses irréductibles spécifiques, il faut définir les algèbres de Hecke sphériques anti-spécifiques en presque toute place et montrer qu'elles sont commutatives. On posera des conditions (dites ``non ramifiées") dans \S\ref{sec:algebre-Hecke} qui doivent être vérifiées en dehors d'un ensemble fini de places $V_\text{ram}$ contenant les places archimédiennes. Notre traitement de tels revêtements s'inspire beaucoup de \cite{Sa88, Sa04, Na09}.

D'après une philosophie bien connue, il faut considérer non seulement un revêtement $\rev: \tilde{G} \to G(A)$, mais aussi ses fibres au-dessus des sous-groupes de Lévi; on appelle ces fibres les sous-groupes de Lévi de $\tilde{G}$. Nos hypothèses pour un revêtement local, non ramifié ou adélique sont préservées par passage aux sous-groupes de Lévi de $G$. De plus, si l'on exige que l'algèbre de Hecke sphérique anti-spécifique d'un revêtement non ramifié est commutative, et idem pour tous les sous-groupes de Lévi, alors les conditions posées dans \S\ref{sec:algebre-Hecke} sont bien minimales. Par ailleurs, nos hypothèses sont aussi préservées par pousser-en-avant le groupe $\bmu_m$ par un homomorphisme.

En pratique les revêtements sont souvent dotés de structures supplémentaires. On démontrera dans \S\ref{sec:BD-ext} que les $\mathbf{K}_2$-torseurs multiplicatifs de Brylinski-Deligne \cite{BD01}, qui généralisent la construction de Steinberg, Moore et Matsumoto \cite{Ma69}, fournissent des revêtements vérifiant nos hypothèses. La démonstration est basé sur un résultat de Weissman \cite{We09}.

\paragraph{La formule des traces grossière}
Soit $\rev: \tilde{G} \to G(\A)$ un revêtement au sens ci-dessus. Fixons un sous-groupe de Lévi minimal $M_0$ et un sous-groupe compact maximal $K \subset G(\A)$ en bonne position relativement à $M_0$. Notons $\tilde{G}^1 := \rev^{-1}(G(\A)^1)$ et $C_{c,\asp}^\infty(\tilde{G}^1)$ l'ensemble des fonctions anti-spécifiques dans $C_c^\infty(\tilde{G}^1)$. Comme dans le cas des groupes réductifs connexes, on a la formule des traces grossière

$$ J(f) = \sum_{\mathfrak{o}} J_\mathfrak{o}(f) = \sum_\chi J_\chi(f), \quad f \in C_{c,\asp}^\infty(\tilde{G}^1) $$
où les indices $\mathfrak{o}$ correspondent aux classes de conjugaison semi-simples dans $G(F)$ comme précédemment, et les indices $\chi$ correspondent, en gros, aux représentations automorphes cuspidales spécifiques sur les sous-groupes de Lévi semi-standards de $\tilde{G}$ modulo l'action de $W_0^G$.

Ici on observe une asymétrie: le côté géométrique est indexé par toutes les classes de conjugaison semi-simples dans $G(F)$, tandis que le côté spectral ne fait intervenir que les représentations spécifiques. Nous y remédierons lors du raffinement. 

\paragraph{Raffinement géométrique}
Pour un groupe réductif connexe $G$, le raffinement géométrique d'un terme $J_\mathfrak{o}$ repose sur la descente au terme $J^{G_\sigma}_\text{unip}$ dans la formule des traces grossière associé au commutant connexe $G_\sigma$ de $\sigma$, où $\sigma \in \mathfrak{o}$. On exprime $J^{G_\sigma}_\text{unip}$ en termes des intégrales orbitales pondérées unipotentes définies dans \cite{Ar88LB}. Les intégrales orbitales pondérées satisfont aussi à une formule de descente. En comparant ces formules de descente, on exprime $J_\mathfrak{o}$ en termes des intégrales orbitales pondérées sur le même groupe.

Le procédé pour un revêtement $\rev: \tilde{G} \to G(\A)$ est analogue sauf que le revêtement disparaît après la descente, et le résultat n'est plus $J^{G_\sigma}_\text{unip}$, mais tordu par un certain caractère de $G_\sigma(\A)$ à cause du fait qu'un élément dans $\rev^{-1}(G_\sigma(\A))$ ne commute pas forcément avec le relèvement de $\sigma$ dans $\tilde{G}$. Mentionnons que la partie elliptique de la formule des traces ``avec un caractère" est beaucoup étudiée (eg. \cite{Lab99}), cependant il nous faut l'autre extrême, la partie unipotente.

De même, nous définissons les intégrales orbitales pondérées sur les revêtements et leurs propriétés se déduisent par descente aux intégrales orbitales pondérées unipotentes sur un groupe réductif connexe, et là encore un caractère intervient.

Notre méthode du raffinement est, pour l'essentiel, celle d'Arthur \cite{Ar85,Ar86}. Or d'une part l'adaptation au cas avec caractère n'est pas toujours triviale, et d'autre part nous avons besoin de renseignements plus précis sur les coefficients dans le développement géométrique fin avec caractère. Cela nécessite la longue section \S\ref{sec:formule-traces-caractere}. Une fois que le formalisme avec un caractère est mis en place, la théorie sur les revêtements en découle par descente. L'un des nouveaux ingrédients dans le développement géométrique fin sur les revêtements \ref{prop:developpement-raffine-geometrique} est la notion des bons éléments (voir \ref{def:bon}); pourtant c'est difficile de les caractériser pour les revêtements en général. Le résultat \ref{prop:developpement-raffine-geometrique} s'écrit

$$ J(f) = \sum_{M \in \mathcal{L}(M_0)} |W_0^M| |W_0^G|^{-1} \sum_{\substack{\gamma \in (M(F))_{M,S}^{K,\mathrm{bon}} \\ \gamma \leadsto \widetilde{\gamma_S} }} a^{\tilde{M}}(S, \dot{\widetilde{\gamma_S}}) J_{\tilde{M}}(\dot{\widetilde{\gamma_S}}, f), $$
où
\begin{itemize}
  \item $M$, $W_0^M$ et $S$ sont pareils que dans le cas des groupes réductifs connexes;
  \item $(M(F))_{M,S}^{K,\mathrm{bon}}$ est le sous-ensemble de $(M(F))_{M,S}$ défini dans \ref{def:K-bon}; notons que la seule nouveauté est la bonté, les autres conditions sont implicites dans les travaux d'Arthur (cf. \cite[Lemma 2.1]{Ar02});
  \item la correspondance $\gamma \leadsto \widetilde{\gamma_S} \in \tilde{M}_S$ est définie \ref{def:passage-local}, où on suppose que $\gamma$ est un représentant admissible de la classe de $(M,S)$-équivalence;
  \item $a^{\tilde{M}}(S,\dot{\widetilde{\gamma_S}})$ est le coefficient de ce développement en $\widetilde{\gamma_S}$;
  \item $J_{\tilde{M}}(\dot{\widetilde{\gamma_S}}, f)$ est l'intégrale orbitale pondérée anti-spécifique en $\dot{\widetilde{\gamma_S}}$.
\end{itemize}

Nous démontrerons que le produit $a^{\tilde{M}}(S, \dot{\widetilde{\gamma_S}}) J_{\tilde{M}}(\dot{\widetilde{\gamma_S}}, f)$ ne dépend que de la classe de $(M,S)$-équivalence et de $f$. Donc cette expression est loisible.

Notons en passant que la démonstration sera beaucoup plus simple si l'on considère un revêtement  tel que deux éléments dans $\tilde{G}$ commutent si et seulement si leurs images par $\rev$ commutent. Tel est le cas du revêtement métaplectique de Weil.

Remarquons que notre méthode permet aussi de raffiner le côté géométrique de la formule des traces avec caractère pour un groupe réductif connexe (voir l'exemple \ref{ex:caractere}). Une généralisation aux groupes tordus aura un intérêt arithmétique.

\paragraph{Structure de cet article}
Dans \S\ref{sec:revetement-local}, nous définissons les revêtements dans le cas local, mettons en place le formalisme de base de l'analyse harmonique et fixons les notations. Le traitement n'est nullement original, mais nous essayons de travailler dans un cadre général: il n'y a aucune hypothèse sur le déploiement, la connexité simple du groupe ou sur les racines d'unité du corps en question.

Dans \S\ref{sec:revetement-global}, nous étudions les revêtements ``non ramifiés", établissons un isomorphisme de Satake et puis définissons les revêtements adéliques. Afin de supporter nos hypothèses, nous démontrons que les $\mathbf{K}_2$-torseurs multiplicatifs de Brylinski-Deligne \cite{BD01} fournissent de tels revêtements adéliques.

La section \S\ref{sec:combinatoire} ne sert qu'à fixer les notations sur les fonctions combinatoires de Langlands et les $(G,M)$-familles.

Dans \S\ref{sec:formule-traces-caractere}, nous étudions le côté géométrique de la formule des traces grossière avec un caractère. Après l'étude des intégrales orbitales pondérées avec caractère, nous obtenons le développement géométrique fin dans ce contexte. Enfin, nous étudions diverses propriétés des coefficients dans le développement fin, qui serviront à remonter ce développement au revêtement.

Dans \S\ref{sec:formule-traces-revetement}, nous mettons en place d'abord la formule des traces grossière pour les revêtements. Puisque des structures analogues sont déjà présentes dans \S\ref{sec:formule-grossiere-revetement}, nous procédons rapidement. Ensuite, nous définissons les intégrales orbitales pondérées anti-spécifiques. Le développement géométrique fin découle d'une réduction au cas unipotent. Nous donnons aussi des formules pour les coefficients similaires à celles d'Arthur.

Une grande partie de ce travail consiste en des paraphrases des travaux d'Arthur. Vu l'épaisseur des ses articles, on se contentera souvent d'indiquer les modifications nécessaires.

\opt{sep}{
\paragraph{Remerciements}
Je tiens à remercier Jean-Loup Waldspurger pour avoir attentivement lu le manuscrit de cet article et d'avoir signalé des erreurs et inexactitudes.
}

\paragraph{Conventions}
Les schémas en groupes sur une base $S$ sont désignés par les symboles $G$, $M$ etc. Leurs algèbres de Lie sont désignées par $\mathfrak{g}$, $\mathfrak{m}$ etc. Le centre de $G$ est noté $Z_G$, le centralisateur d'un sous-schéma en groupes $H$ (resp. d'un $S$-point $x$) est noté $Z_G(H)$ (resp. $Z_G(x)$); le normalisateur de $H$ est noté $N_G(H)$. Soit $T$ un $S$-schéma, l'ensemble des $T$-points d'un $S$-schéma $X$ est désigné par $X(T)$. Lorsque $T=\Spec A$ où $A$ est une algèbre, on écrit aussi $X(A)$ au lieu de $X(T)$. Si $A$ est muni d'une topologie, on munit $X(A)$ de la topologie induite. 

Soit $F$ un corps, on fixe une clôture algébrique $\bar{F}$ de $F$. Soit $G$ un $F$-groupe algébrique. On désigne l'ensemble des éléments semi-simples dans $G(F)$ par $G(F)_\text{ss}$. Soit $x \in G(F)$, on pose $G^x := Z_G(x)$ le commutant de $x$ dans $G$, et $G_x$ désigne la composante neutre de $G^x$. On dit que $x \in G(F)_\text{ss}$ est régulier (resp. fortement régulier) si $G_x$ (resp. $G^x$) est un tore. On désigne la sous-variété des éléments semi-simples réguliers par $G_\text{reg}$. La sous-variété des éléments unipotents dans $G$ est désignée par $G_\text{unip}$. De même, pour l'algèbre de Lie $\mathfrak{g}$, on a la sous-variété $\mathfrak{g}_\text{reg}$ des éléments réguliers semi-simples et le cône nilpotent $\mathfrak{g}_\text{nil}$. \index[iFT1]{$G_\text{reg}$}\index[iFT1]{$\mathfrak{g}_\text{reg}$}\index[iFT1]{$G_\text{unip}$}\index[iFT1]{$G_\text{nil}$}\index[iFT1]{$G_x, G^x$}

On désigne le sous-groupe dérivé (schématique) de $G$ par $G_\text{der}$ et le groupe adjoint par $G_\text{AD}$. Si $G$ est réductif et connexe, on désigne le revêtement simplement connexe de $G_\text{der}$ par $\pi: G_\text{SC} \to G_\text{der}$.\index[iFT1]{$\pi: G_\text{SC} \to G_\text{der}$}

On dit que deux éléments $x,y \in G(F)$ sont géométriquement conjugués s'ils sont conjugués par un élément dans $G(\bar{F})$. On définit ainsi les classes de conjugaison géométriques dans $G(F)$.

Soit $F$ un corps complet à valuation discrète. On utilise toujours la valuation normalisée $v$ de sorte que $v(F)=\Z$. L'anneau des entiers est noté $\mathfrak{o}_F$ et l'idéal maximal est noté par $\mathfrak{p}_F$. Soit $F$ un corps global, on prend les valeurs absolues $|\cdot|_v$ de façon usuelle en chaque place $v$ de telle sorte que $\prod_v |x|_v = 1$ pour tout $x \in F^\times$.

Pour deux éléments $u,v$ dans un groupe quelconque, leur commutateur est défini comme\index[iFT1]{$[\cdot,\cdot]$}
$$ [u,v] := u^{-1} v^{-1} u v. $$

On désigne la fonction modulaire d'un groupe topologique $A$ par $\delta_A(\cdot)$. On désigne la mesure d'un espace mesurable $E$ par $\mes(E)$.

\section{Revêtements locaux}\label{sec:revetement-local}
\subsection{Généralités}\label{sec:generalites}
Soient $F$ un corps local de caractéristique nulle et $M$ un $F$-groupe algébrique affine. Un revêtement de $M(F)$ à $m$ feuillets (où $m \in \Z \setminus \{0\}$) est une extension centrale de groupes topologiques\index[iFT1]{$\bmu_m$}\index[iFT1]{$\rev$}
$$ 1 \to \bmu_m \to \tilde{M} \xrightarrow{\rev} M(F) \to 1, $$
où $\bmu_m := \{\noyau \in \C^\times : \noyau^m = 1 \}$. Alors $\tilde{M}$ est unimodulaire si $M(F)$ l'est. Si $F$ est archimédien, alors $\tilde{M}$ appartient à la classe de Harish-Chandra. De plus, si $F=\C$ alors $\rev$ provient d'un revêtement étale de $\C$-groupes algébriques affines (\cite{SGA1} Exp XII, 5.1). Si $F$ est non archimédien, alors $\tilde{M}$ est un groupe localement profini. Cela permet de parler de représentation lisses, admissibles etc. On dit que $\rev$ est modéré si $F$ est non archimédien de caractéristique résiduelle $q$ première à $m$. Nous adoptons la convention de doter les éléments dans $\tilde{M}$ d'un $\sim$, par exemple $\tilde{x}$, et désignons son image dans $M(F)$ par le symbole sans $\sim$, par exemple $x=\rev(\tilde{x})$.

Remarquons que $M(F)$ agit sur $\tilde{M}$ par conjugaison: de chaque $x \in M(F)$ se déduit un homomorphisme $\tilde{m} \mapsto x^{-1} \tilde{m} x$ de $\tilde{M}$. Sauf mention expresse du contraire, un revêtement signifie un revêtement d'un groupe réductif connexe. 

Notons\index[iFT1]{$\widehat{\bmu_m}$}
$$\widehat{\bmu_m} := \Hom(\bmu_m, \C^\times). $$
Pour un revêtement à $m$ feuillets $\rev: \tilde{M} \to M(F)$, on peut définir les objets spécifiques et anti-spécifiques selon l'action de $\bmu_m$, dotés de l'indice $-$ et $\asp$ respectivement, ou plus généralement les objets équivariants par rapport à certain élément dans $\widehat{\bmu_m}$. Plus précisément, soit $\chi_-$ l'inclusion $\bmu_m \hookrightarrow \C^\times$; pour tout $\chi \in \widehat{\bmu_m}$, posons\index[iFT1]{$C_{c,\chi}^\infty(\tilde{M})$}\index[iFT1]{$C_{c,\asp}^\infty(\tilde{M})$}\index[iFT1]{$C_{c,-}^\infty(\tilde{M})$}\index[iFT1]{anti-spécifique}\index[iFT1]{spécifique}
\begin{align*}
  C_{c,\chi}^\infty(\tilde{M}) & := \{f \in C_c^\infty(\tilde{M}) : \forall \noyau \in \bmu_m, \forall \tilde{x} \in \tilde{M},\; f(\noyau \tilde{x}) = \chi(\noyau) f(\tilde{x}) \}, \\
  C_{c,-}^\infty(\tilde{M}) & := C_{c,\chi_-}^\infty, \\
  C_{c,\asp}^\infty(\tilde{M}) & := C_{c,\chi_-^{-1}}^\infty .
\end{align*}

Notons $\Pi(\tilde{M})$\index[iFT1]{$\Pi(\tilde{M})$} l'ensemble de classes d'équivalences de représentations admissibles irréductibles de $\tilde{M}$, posons\index[iFT1]{$\Pi_{\chi}(\tilde{M})$}\index[iFT1]{$\Pi_-(\tilde{M})$}\index[iFT1]{$\Pi_{\asp}(\tilde{M})$}
\begin{align*}
  \Pi_{\chi}(\tilde{M}) & := \{\pi \in \Pi(\tilde{M}) : \forall \noyau \in \bmu_m, \; \pi(\noyau) = \chi(\noyau) \identity \}, \\
  \Pi_-(\tilde{M}) & := \Pi_{\chi_-}(\tilde{M}), \\
  \Pi_{\asp}(\tilde{M}) & := \Pi_{\chi_-^{-1}}(\tilde{M}).
\end{align*}
De même, on définit l'ensemble $\Pi_2(\tilde{M})$ (resp. $\Pi_\text{temp}(\tilde{M})$, $\Pi_\text{unit}(\tilde{M})$) de représentations de la série discrète (resp. tempérées, unitaires) de $\tilde{M}$, et on rajoute les indices $-,\asp$ ou $\chi \in \widehat{\bmu_m}$ pour signifier l'équivariance.\index[iFT1]{$\Pi_2(\tilde{M})$}\index[iFT1]{$\Pi_{\text{temp}}(\tilde{M})$}\index[iFT1]{$\Pi_{\text{unip}}(\tilde{M})$}

On a une décomposition canonique $C_c^\infty(\tilde{M}) = \bigoplus_{\chi \in \widehat{\bmu_m}} C_{c,\chi}^\infty(\tilde{M})$. Cela permet aussi de parler de l'équivariance de distributions de sorte qu'une fonction $\chi$-équivariante localement intégrable fournit une distribution $\chi$-équivariante. L'étude des représentations sur les revêtements se ramène, pour l'essentiel, à l'étude des représentations spécifiques.

\begin{remark}
  Pour l'étude de représentations $\chi$-équivariantes sur $\tilde{M}$, il suffit de considérer les fonctions test $\bar{\chi}$-équivariantes. En effet, supposons fixée une mesure de Haar sur $\tilde{M}$. Soient $\chi, \xi \in \widehat{\bmu_m}$. Pour tout $\pi \in \Pi_{\chi}(\tilde{M})$ et $f \in C_{c,\xi}^\infty(\tilde{M})$, l'opérateur
  $$ \pi(f) = \int_{\tilde{M}} f(\tilde{m}) \pi(\tilde{m}) \dd \tilde{m} $$
  est nul sauf si $\xi=\bar{\chi}$.
\end{remark}

\subsection{Scindage unipotent}
Conservons les notations précédentes.\index[iFT1]{scindage unipotent}

\begin{proposition}\label{prop:scindage-unip}
  Il existe une seule section continue $s: M_\text{unip}(F) \to \tilde{M}$ de $\rev$ telle que
  \begin{itemize}
    \item pour tout sous-groupe unipotent $U$ de $M$ défini sur $F$, $s|_{U(F)}$ est un homomorphisme;
    \item $s$ est invariant par conjugaison.
  \end{itemize}
\end{proposition}
\begin{proof}
  C'est contenu dans \cite[A.1]{MW94}. Donnons une preuve directe pour le cas de caractéristique nulle. L'exponentielle fournit un $F$-isomorphisme de variétés algébriques
  $$ \exp: \mathfrak{m}_\text{nil} \to M_\text{unip}. $$
  Pour tout $x = \exp(X)$ dans $M_\text{unip}(F)$, prenons $\tilde{x}'$ un relèvement quelconque de $\exp\left(\frac{X}{m}\right)$. Alors $s(x) := (\tilde{x}')^m \in \rev^{-1}(x)$ est canoniquement défini; en particulier $s$ est invariant par conjugaison. On vérifie aisément la continuité de $s$.

  Soit $U$ un sous-groupe unipotent de $M$, alors $U(F)$ est divisible et sans torsion. D'après la construction ci-dessus, $\rev$ se scinde au-dessus de $U(F)$ si et seulement si $s|_{U(F)}$ est un homomorphisme; de plus, dans ce cas-là $s|_{U(F)}$ est l'unique scindage.

  Montrons que $\rev$ se scinde au-dessus de $U(F)$. Si $U$ est commutatif, alors $s$ est un homomorphisme d'après la construction, d'où le scindage. En général, les revêtements à $m$ feuillets de $U(F)$ sont classifiés par $H^2(U(F),\bmu_m)$ (la cohomologie continue) et il suffit de montrer que ce $H^2$ est trivial. Supposons que $\dim U \geq 1$. Il existe un sous-groupe algébrique distingué $U_1 \triangleleft U$ tel que $\dim U_1 < \dim U$ et $U/U_1$ est commutatif. Rappelons que $U(F)/U_1(F) = (U/U_1)(F)$ car $H^1(F, U_1)=0$. Pour tout $F$-groupe unipotent $U'$, on a $H^1(U'(F), \bmu_m)=0$. D'où la suite exacte de restriction-inflation
  $$ 0 \to H^2((U/U_1)(F),\bmu_m) \to H^2(U(F),\bmu_m) \to H^2(U_1(F),\bmu_m), $$
  ce qui entraîne que $H^2(U(F),\bmu_m)=0$ par récurrence. Par conséquent $s|_{U(F)}$ est un homomorphisme. Comme $M_\text{unip}(F)$ est la réunion des $U(F)$, cela caractérise $s$.
\end{proof}

Ce scindage canonique s'appelle le scindage unipotent. Identifions désormais $M_\text{unip}(F)$ comme un sous-ensemble de $\tilde{M}$ via $s$. Cela permet de généraliser la décomposition de Jordan.\index[iFT1]{décomposition de Jordan}

\begin{proposition}
  Pour tout $\tilde{x} \in \tilde{M}$, il existe $\tilde{\sigma} \in \tilde{M}$ et $u \in M_\text{unip}(F)$ tels que $\sigma$ est semi-simple et $\tilde{x}=\tilde{\sigma} u = u \tilde{\sigma}$. Cette décomposition est unique.
\end{proposition}
On dit que $\tilde{\sigma}$ (resp. $u$) est la partie semi-simple (resp. unipotente) de $\tilde{x}$.
\begin{proof}
  Soit $x = \sigma u = u \sigma$ la décomposition de Jordan dans $M(F)$ avec $\sigma \in M(F)_\text{ss}$ et $u \in M_\text{unip}(F)$. Prenons l'unique $\tilde{\sigma} \in \rev^{-1}(\sigma)$ de sorte que $\tilde{x} = \tilde{\sigma} u$. L'unicité de $(\tilde{\sigma}, u)$ provient de celle de $(\sigma,u)$. De plus, on a $\tilde{\sigma} u = u \tilde{\sigma}$ par l'invariance du scindage unipotent, d'où le résultat cherché.
\end{proof}

\begin{corollary}
  Soit $\tilde{x} = \tilde{\sigma}u$ la décomposition de Jordan. Soit $\tilde{y} \in \tilde{M}$, alors $\tilde{y}$ commute à $\tilde{x}$ si et seulement si $\tilde{y} \tilde{\sigma} = \tilde{\sigma} \tilde{y}$ et $yu=uy$.
\end{corollary}
\begin{proof}
  Cela résulte de l'unicité de la décomposition de Jordan et l'invariance du scindage unipotent.
\end{proof}

\subsection{Sous-groupes de Lévi et paraboliques}
Passons en revue la description des sous-groupes paraboliques. Les détails se trouvent dans \cite[\S 5]{ArIntro}. Soit $F$ un corps quelconque et $G$ un $F$-groupe réductif connexe. Fixons un sous-groupe de Lévi minimal $M_0$ de $G$: c'est le centralisateur d'un $F$-tore déployé maximal $A_0$. Un sous-groupe de Lévi $M$ est dit semi-standard si $M \supset M_0$, un sous-groupe parabolique $P$ est dit semi-standard si $P \supset A_0$. Tout sous-groupe parabolique semi-standard $P$ admet une décomposition de Lévi canonique $P=M_P U_P$ avec $M_P$ semi-standard et $U_P$ le radical unipotent de $P$. Notons $\overline{P} = M_P  U_{\overline{P}}$ le sous-groupe parabolique opposé de $P$.

Pour un sous-groupe de Lévi semi-standard $M$, définissons les ensembles finis suivants
\begin{align*}
  \mathcal{L}(M) &:= \{\text{les Lévis contenant $M$} \},\\
  \mathcal{P}(M) &:= \{\text{les paraboliques dont $M$ est un facteur de Lévi}\},\\
  \mathcal{F}(M) &:= \{\text{les paraboliques contenant $M$}\}.
\end{align*}

Nous indiquons le groupe ambiant $G$ en exposant dans ces notations: $\mathcal{L}^G(M)$, $\mathcal{P}^G(M)$, $\mathcal{F}^G(M)$ lorsqu'il y a crainte de confusion.\index[iFT1]{$\mathcal{L}^G(M)$}\index[iFT1]{$\mathcal{P}^G(M)$}\index[iFT1]{$\mathcal{F}^G(M)$}

Notons $A_M$\index[iFT1]{$A_M$} le $F$-tore central déployé maximal dans $M$. Si $P \in \mathcal{P}(M)$, notons $A_P := A_M$. Posons aussi $X^*(M) := \Hom_\text{alg}(M,\Gm)$ et $\mathfrak{a}_P = \mathfrak{a}_M := \Hom(X^*(M),\R)$. Relativisons ces constructions. Pour tous $L,M$ semi-standards tels que $L \supset M$, on sait définir les $\R$-espaces vectoriels de dimension finie $\mathfrak{a}^L_M$ avec une suite exacte courte scindée canonique\index[iFT1]{$\mathfrak{a}_M$, $\mathfrak{a}^G_M$}
$$ 0 \to \mathfrak{a}_L \to \mathfrak{a}_M \leftrightarrows \mathfrak{a}^L_M \to 0. $$

Ainsi, on regarde $\mathfrak{a}^L_M$ comme sous-espace de $\mathfrak{a}_0$. En dualisant, on en déduit des suites exactes courtes scindées pour $(\mathfrak{a}^L_M)^*$, $\mathfrak{a}_M^*$ etc. Les complexifiés des espaces sont notés par $\mathfrak{a}^L_{M,\C}, (\mathfrak{a}^L_{M,\C})^*$, etc.

Pour tout $M \in \mathcal{L}(M_0)$, notons $W_0^M$ le groupe de Weyl de $M$. Si $M=G$, on le note aussi $W_0$. Pour deux sous-groupes paraboliques semi-standards $P, P'$, ``l'ensemble de Weyl'' $W(\mathfrak{a}_P, \mathfrak{a}_{P'})$ est l'ensemble des isomorphismes linéaires $\mathfrak{a}_{P} \to \mathfrak{a}_{P'}$ obtenus en restreignant les isomorphismes $\mathfrak{a}_0 \to \mathfrak{a}_0$ induits par $W_0^G$. En particulier, on peut définir les groupes $W(\mathfrak{a}_P) := W(\mathfrak{a}_P, \mathfrak{a}_P)$. Deux sous-groupes paraboliques semi-standards $P, P'$ sont dits associés si $W(\mathfrak{a}_P, \mathfrak{a}_{P'}) \neq \emptyset$.

Fixons $P_0 \in \mathcal{P}(M_0)$. Un sous-groupe parabolique $P$ est dit standard si $P \supset P_0$. Un sous-groupe de Lévi $M$ est dit standard s'il existe un sous-groupe parabolique standard $P$ avec décomposition de Lévi canonique $P=MU$. Soit $P$ un sous-groupe parabolique. Il existe un sous-ensemble fini $\Sigma_P \subset X^*(A_P) \subset \mathfrak{a}_P^*$ paramétrisant la décomposition
$$ \mathfrak{u}_P := \Lie(U_P) = \bigoplus_{\alpha \in \Sigma_P} \mathfrak{u}_\alpha $$
en espaces propres pour l'action adjointe de $A_P$. Par abus de notation, on dit aussi que $\Sigma_P$ est l'ensemble des racines pour $(A_P,P)$, bien qu'il ne forme pas un système de racines en général. Notons $\Sigma_P^\text{red}$ le sous-ensemble de $\Sigma_P$ des racines réduites, i.e. indivisibles\index[iFT1]{$\Sigma_P$, $\Sigma_P^\text{red}$}. Posons\index[iFT1]{$\rho_P$}
$$\rho_P := \frac{1}{2} \sum
_{\alpha \in \Sigma_P} (\dim \mathfrak{u}_\alpha) \alpha. $$

Soit $\Delta_0 = \Delta^G_0$ l'ensemble des racines simples de $(A_0, P_0)$, c'est une base pour $(\mathfrak{a}_0^G)^*$. Les paraboliques standards sont en correspondance biunivoque $P \leftrightarrow \Delta_0^P $ avec les sous-ensembles de $\Delta_0$ préservant l'ordre. Plus précisément, supposons que $P \supset P_0$ et soit $P=MU$ la décomposition de Lévi canonique; posons $\Delta_P := \Delta_0 \setminus \Delta^P_0$. On peut identifier $\Delta_P$ à un sous-ensemble de $\Sigma_P^\text{red}$ par restriction. Tout élément dans $\Sigma_P$ admet une unique écriture en combinaison linéaire d'éléments de $\Delta_P$ à coefficients dans $\Z_{\geq 0}$.

On obtient ainsi les bases
\begin{align*}
  \Delta_0 & \subset (\mathfrak{a}_0^G)^*  : \quad \text{racines simples},\\
  \Delta_0^\vee & \subset \mathfrak{a}_0^G  : \quad \text{coracines simples},\\
  \widehat{\Delta_0} & \subset (\mathfrak{a}_0^G)^*  : \quad \text{la base duale de } \Delta_0^\vee, \\
  \widehat{\Delta_0^\vee} & \subset \mathfrak{a}_0^G  : \quad \text{la base duale de } \Delta_0 .
\end{align*}

On peut aussi relativiser cette situation: étant donnés sous-groupes paraboliques standards $P \supset Q$, on obtient les bases
\begin{align*}
  \Delta^P_Q & \subset (\mathfrak{a}^P_Q)^*, \\
  {\Delta^P_Q}^\vee & \subset \mathfrak{a}^P_Q, \\
  \widehat{\Delta^P_Q} & \subset (\mathfrak{a}^P_Q)^*, \\
  \widehat{{\Delta^P_Q}^\vee} & \subset \mathfrak{a}^P_Q.
\end{align*}

\subsection{L'application de Harish-Chandra: le cas local}\label{sec:app-HC}
On se donne $F$ un corps local, $G$ un $F$-groupe réductif dont $M$ est un sous-groupe de Lévi. On définit l'homomorphisme de Harish-Chandra local $H_M: M(F) \to \mathfrak{a}_M$ par\index[iFT1]{$H_M$}\index[iFT1]{application de Harish-Chandra}
$$ \forall \chi \in X^*(M), \quad \angles{\chi, H_M(x)} = \log |\chi(x)|. $$

\begin{definition}
  On dit qu'un sous-groupe compact maximal $K \subset G(F)$ est en bonne position relativement à $M$ (et réciproquement) si
  \begin{itemize}
    \item dans le cas $F$ archimédien, les algèbres de Lie de $K$ et de $A_M(F)$ sont orthogonales par rapport à la forme de Killing de $G$;
    \item dans le cas $F$ non archimédien, $K$ est associé à un sommet spécial dans l'immeuble de Bruhat-Tits élargi de $G$, noté $\mathscr{I}(G)$, qui appartient à l'image d'une immersion équivariante $\mathscr{I}(M) \hookrightarrow \mathscr{I}(G)$.
  \end{itemize}

  Arthur l'appelle admissible dans \cite{Ar81}.
\end{definition}

Notons $M(F)^1 := \Ker(H_M)$. Si $P \in \mathcal{P}(M)$ et $K$ est un sous-groupe compact maximal en bonne position relativement à $M$, alors la décomposition d'Iwasawa $G(F)=P(F)K$ permet de prolonger $H_M$ en une fonction $H_P: G(F) \to \mathfrak{a}_M$ en posant
$$ H_P(umk) = H_M(m), \quad u \in U(F), m \in M(F), k \in K.$$
Pour tout $x \in G(F)$, $H_P(x)$ est déterminé par la classe de $x$ dans $U(F)\backslash G(F)/K$. La fonction modulaire $\delta_P$ de $P(F)$ s'exprime comme $\delta_P(x) = e^{\angles{2\rho_P, H_P(p)}}$.

Nous adoptons la convention suivante: soit $x \in G(F)$, écrivons-le comme
\begin{gather*}
  x = u_P(x) m_P(x) k_P(x) \in G(F), \\
  u_P(x) \in U_P(F), m_P(x) \in M_P(F), k_P(x) \in K;
\end{gather*}
à l'aide de la décomposition d'Iwasawa; l'élément $m_P(x)$ (resp. $k_P(x)$) est uniquement déterminé comme une classe dans $M(F)/M(F) \cap K$ (resp. dans $P(F) \cap K \backslash K$).

Posons\index[iFT1]{$\mathfrak{a}_{M,F}, \tilde{\mathfrak{a}}_{M,F}$}
\begin{align*}
  \mathfrak{a}_{M,F} & := H_M(M(F)), \\
  \tilde{\mathfrak{a}}_{M,F} & := H_M(A_M(F)).
\end{align*}
Ils coïncident avec $\mathfrak{a}_M$ si $F$ est archimédien; sinon ils sont des réseaux dans $\mathfrak{a}_G$. Définissons leurs réseaux duaux dans $i\mathfrak{a}_M^*$
\begin{align*}
  \mathfrak{a}_{M,F}^\vee & := \Hom(\mathfrak{a}_{M,F}, 2\pi i\Z),\\
  \tilde{\mathfrak{a}}_{M,F}^\vee & := \Hom(\tilde{\mathfrak{a}}_{M,F}, 2\pi i\Z).
\end{align*}
Ils se réduisent à $\{0\}$ si $F$ est archimédien; sinon $i\mathfrak{a}_M^*/\mathfrak{a}_{M,F}^\vee$ et $i\mathfrak{a}_M^*/\tilde{\mathfrak{a}}_{M,F}^\vee$ sont des tores réels compacts.

Considérons un revêtement à $m$ feuillets $\rev: \tilde{G} \to G(F)$. On prend les images réciproques $\tilde{M}$ (resp. $\tilde{P}$) par $\rev$ des sous-groupes de Lévi $M$ (resp. sous-groupes paraboliques $P$) de $G$. Soit $M \in \mathcal{L}(M_0)$. En composant $H_M$ avec $\rev$, on obtient $H_M: \tilde{M} \to \mathfrak{a}_M$; en particulier on sait définir
$$\tilde{M}^1 := \Ker(H_M) = \rev^{-1}(M(F)^1).$$

\subsection{Mesures et intégrales}\label{sec:mesure}
Soient $F$ un corps local de caractéristique nulle et $G$ un groupe $F$-réductif connexe. Supposons fixées des mesures de Haar sur $M(F)$ pour tout sous-groupe de Lévi $M$. Imposons les règles suivantes
\begin{itemize}
  \item un sous-groupe compact maximal fixé de $G(F)$ est de masse totale $1$;
  \item un groupe discret est muni de la mesure de comptage.
\end{itemize}

Fixons des mesures de Haar sur $\mathfrak{a}_M$ pour tout sous-groupe de Lévi $M$ de $G$, d'où les mesures de Haar duales sur $i\mathfrak{a}_M^*$ au sens que
$$ \iint_{i\mathfrak{a}_M^* \times \mathfrak{a}_M} \phi(H) e^{-\angles{\lambda,H}} \dd H \dd\lambda = \phi(0) $$
pour tout $h \in C_c(\mathfrak{a}_M)$. Si $F$ est non archimédien, nous demandons que
$$ \mes(i\mathfrak{a}_M^*/\tilde{\mathfrak{a}}_{M,F}^\vee) = 1. $$

Comme $\tilde{\mathfrak{a}}_{M,F}$ est soit discret, soit égal à $\mathfrak{a}_M$, et $\Ker(H_M|_{A_M(F)})$ est compact, on normalise ainsi la mesure de Haar sur $A_M(F)$. De même, une mesure de Haar sur $M(F)$ induit une mesure de Haar sur $M(F)^1$.

Fixons désormais une forme quadratique définie positive $W_0$-invariante sur $\mathfrak{a}_0$. Soient $L \supset M$ deux sous-groupes de Lévi de $G$, on vérifie que la décomposition canonique
$$ \mathfrak{a}_M = \mathfrak{a}_M^L \oplus \mathfrak{a}_L $$
est orthogonale par rapport à la forme quadratique $W_0$-invariante. Puisque les mesures de Haar sur $\mathfrak{a}_M$ et $\mathfrak{a}_L$ sont déjà fixées, on en déduit une mesure canonique sur $\mathfrak{a}_M^L$. En dualisant, on normalise la mesure de Haar sur $(\mathfrak{a}_M^L)^*$.

Soient $P=MU \in \mathcal{P}(M)$ et $K$ un sous-groupe compact maximal en bonne position relativement à $M$, alors on dispose de la décomposition d'Iwasawa $G(F)=U(F)M(F)K$. Il existe une  mesure de Haar sur $U(F)$ de sorte que pour tout $f \in C_c(G(F))$,
\begin{align}
  \label{eqn:UMK} \int_{G(F)} f(x) \dd x & = \iiint_{U(F) \times M(F) \times K} f(umk) \delta_P(m)^{-1} \dd k \dd m \dd u .
\end{align}

Dans le cas $F$ non archimédien et $G$ non ramifié, la compatibilité des mesures est simple. Prenons $K$ hyperspécial. Prenons la mesure de Haar sur $G(F)$ (resp. $M(F), U(F)$) telle que $G(F) \cap K$ (resp. $M(F) \cap K$, $U(F) \cap K$) a masse totale $1$. Alors ces mesures vérifient \eqref{eqn:UMK}.

Considérons maintenant les revêtements. Conservons les conventions précédentes pour les groupes réductifs et leurs sous-groupes. Imposons la règle suivante pour les mesures sur les revêtements:
\begin{itemize}
  \item supposons que $\rev: A \to B$ est un revêtement fini de groupes topologiques localement compacts, et $B$ est muni d'une mesure de Haar, alors $A$ est muni de la mesure de Haar telle que $\mes_B(E) = \mes_A(\rev^{-1}(E))$ pour tout $E \subset A$ mesurable.
\end{itemize}

Montrons qu'avec nos définitions, appliquées au revêtements de $G(F)$, les formules d'intégration habituelles restent valables. Soit $\rev: \tilde{G} \to G(F)$ un revêtement à $m$ feuillets. En prenant les images réciproques par $\rev$ et en utilisant le scindage unipotent, on a $\tilde{G}=U(F)\tilde{M}\tilde{K}$. Prenons les mesures de Haar sur $\tilde{G}$, $\tilde{M}$ et $\tilde{K}$ selon la règle ci-dessus. Alors pour tout $f \in C_c(\tilde{G})$, on a
\begin{align}
  \label{eqn:UMK-rev} \int_{\tilde{G}} f(\tilde{x}) \dd \tilde{x} & = \iiint_{U(F) \times \tilde{M} \times \tilde{K}} f(u\tilde{m}\tilde{k}) \delta_P(m)^{-1} \dd\tilde{k} \dd\tilde{m} \dd u .
\end{align}

En effet, il suffit de le vérifier pour les fonctions $f$ qui se factorisent par $\rev: \tilde{G} \to G(F)$. La convention sur les mesures permet de remplacer l'intégrale sur $\tilde{G}$ par celle sur $G(F)$, et idem pour $\tilde{M}, M(F)$ et $\tilde{K}, K$. L'identité cherchée en résulte. Les compatibilités avec d'autres décompositions (eg. la décomposition $G=K A K$) se vérifient de la même manière.

\subsection{Commutateurs}\label{sec:commutateurs}
On revient aux notations de \S\ref{sec:generalites}. Définissons la sous-variété
$$\text{Comm}(M) := \{(x,y) \in M \times M : xy=yx \}.$$
Pour $(x,y) \in \text{Comm}(M)(F)$, choisissons des relèvements $\tilde{x},\tilde{y} \in \widetilde{M}$, alors le commutateur
$$ [\tilde{x},\tilde{y}] := \tilde{x}^{-1}\tilde{y}^{-1}\tilde{x}\tilde{y} \in \bmu_m; $$
ne dépend pas du choix de relèvements. On en déduit une application continue $[\cdot, \cdot]: \text{Comm}(M)(F) \to \bmu_m$, notée $(x,y) \mapsto [x,y]$. Les propriétés suivantes sont immédiates.

\begin{itemize}
  \item Si $x,x'$ commutent à $y$, alors $[xx',y]=[x,y][x',y]$.
  \item Soit $t \in M(F)$, alors $[txt^{-1},tyt^{-1}]=[x,y]$ pour tout $(x,y) \in \text{Comm}(M)(F)$.
  \item Pour tout $(x,y) \in \text{Comm}(M)(F)$, on a $[x,y]=[y,x]^{-1}$.
  \item Si $(x,y) \in \text{Comm}(M)(F)$ et s'ils appartiennent à un sous-groupe de $M(F)$ sur lequel $\rev$ est scindé, alors $[x,y]=1$.
\end{itemize}

Soit $\gamma \in M(F)$, on a
\begin{align*}
  y^{-1} \tilde{\gamma} y & = [\gamma,y] \tilde{\gamma}, \\
  y \tilde{\gamma} y^{-1} & = [y,\gamma] \tilde{\gamma}, \quad y \in M^\gamma(F) .
\end{align*}
D'où un homomorphisme continu $M^\gamma(F) \to \bmu_m$, noté $[\cdot, \gamma]: y \mapsto [y,\gamma]$.

\begin{definition}[cf. {\cite[I.8]{KF86}} ]\label{def:bon}
  Un élément $\gamma \in M(F)$ est dit bon si $[\cdot,\gamma]=1$ sur $M^\gamma(F)$. Cette propriété ne dépend que de la classe de conjugaison de $\gamma$. On dit qu'une classe de conjugaison dans $\widetilde{M}$ est bonne si son image par $\rev$ l'est.
\end{definition}

Montrons que la bonté est stable par petite perturbation par le centre. Posons\index[iFT1]{$A_M(F)^\dagger$}
\begin{align}\label{eqn:A_M^dagger}
  A_M(F)^\dagger & := A_M(F)^m, \\
  \widetilde{A_M} & := \rev^{-1}(A_M(F)), \\
  \widetilde{A_M}^\dagger & := \rev^{-1}(A_M(F)^\dagger).
\end{align}
Alors $\widetilde{A_M}^\dagger$ (resp. $A_M(F)^\dagger$) est un sous-groupe ouvert et fermé d'indice fini de $\widetilde{A_M}$ (resp. de $A_M(F)$). De plus, $\widetilde{A_M}^\dagger$ est central dans $\tilde{M}$.

\begin{lemma}\label{prop:bon-perturbation}
  Pour tout $\gamma \in M(F)$ et tout $a \in A_M(F)^\dagger$, $\gamma$ est bon si et seulement si $a\gamma$ l'est.
\end{lemma}
\begin{proof}
  On a $M^{a\gamma}=M^\gamma$ car $a \in A_M(F)$. Soit $x \in M^{a\gamma}$, on a $[x,a\gamma] = [x,\gamma]$ car $\widetilde{A_M}^\dagger$ est central. Cela permet de conclure.
\end{proof}


\section{Revêtements non ramifiés et adéliques}\label{sec:revetement-global}
\subsection{Le cas non ramifié}\label{sec:algebre-Hecke}
Soit $F$ un corps local non archimédien avec $q : = |\mathfrak{o}_F/\mathfrak{p}_F|$. On se donne un revêtement $\rev: \tilde{M} \to M(F)$ à $m$ feuillets tel que $M$ est non ramifié. Fixons un sous-groupe hyperspécial $K \subset M(F)$ et supposons qu'il existe un scindage continu $s: K \to \tilde{M}$ de $\rev$ au-dessus de $K$.\index[iFT1]{revêtement non ramifié}

Regardons $K$ comme un sous-groupe de $\tilde{M}$ en fixant un tel scindage $s$. Prenons la mesure de Haar sur $M(F)$ telle que $\mes(K)=1$, d'où une mesure de Haar sur $\tilde{M}$ selon les conventions de \S\ref{sec:mesure}. Cette mesure est canonique car les sous-groupes hyperspéciaux sont conjugués par $M_\text{AD}(F)$.

On définit l'algèbre de Hecke sphérique $\mathcal{H}(\tilde{G}/\!/K)$: c'est l'espace des fonctions $K$-bi-invariantes à support compact, muni du produit de convolution. Soit $\chi \in \widehat{\bmu_m}$, posons $\mathcal{H}_\chi(\tilde{G}/\!/K) := \mathcal{H}(\tilde{G}/\!/K) \cap C_{c,\chi}^\infty(\tilde{G})$; c'est une sous-algèbre et on a $\mathcal{H}(\tilde{G}/\!/K) = \prod_{\chi \in \widehat{\bmu_m}} \mathcal{H}_\chi(\tilde{G}/\!/K)$. Définissons la fonction $f_{K,\chi} \in \mathcal{H}_\chi(\tilde{G}/\!/K)$ à support dans $\tilde{K}$ telle que
\begin{align*}
  \forall \noyau \in \bmu_m, \;\forall k \in K, \quad &  f_{K,\chi}(\noyau K) = \chi(\noyau), \\
  \forall \tilde{x} \notin \tilde{K}, \quad & f_{K,\chi}(\tilde{x}) = 0.
\end{align*}
Selon notre convention de mesures, $f_{K,\chi}$ est l'unité de $\mathcal{H}_\chi(\tilde{G}/\!/K)$. Si $\chi=\chi_-^{-1}$ (i.e. on considère l'algèbre de Hecke sphérique anti-spécifique), posons $f_K = f_{K,\chi}$.\index[iFT1]{$f_K$}

En particulier, on peut définir l'algèbre de Hecke anti-spécifique associée à $K$, notée $\mathcal{H}_{\asp}(\tilde{G}/\!/K)$\index[iFT1]{$\mathcal{H}_{\asp}(\tilde{G}//K)$}\index[iFT1]{algèbre de Hecke anti-spécifique} dont $f_K$ est l'unité. De même, on peut définir l'algèbre d'Iwahori-Hecke anti-spécifique (ou plus généralement, $\chi$-équivariante) sous les mêmes hypothèses.

\begin{definition}\label{def:cond-nr}
  On dit qu'un triplet $(\rev, K, s)$ vérifie la condition non ramifiée si
  \begin{itemize}
    \item $\rev: \tilde{M} \to M(F)$ est un revêtement;
    \item $K \subset M(F)$ est un sous-groupe hyperspécial;
    \item $s: K \to \tilde{M}$ est un scindage de $\rev$ au-dessus de $K$ par lequel $K$ s'identifie à un sous-groupe de $\tilde{M}$;
    \item $q$ est premier avec $m:= |\Ker(\rev)|$, i.e. $\rev$ est modéré;
    \item soient $T$ un $F$-tore déployé maximal et $M_0 := Z_M(T)$ en bonne position relativement à $K$, alors le groupe
      \begin{gather}\label{eqn:def-H}
        \tilde{H} := Z_{\widetilde{M_0}}(K \cap M_0(F))
      \end{gather}
      est commutatif\index[iFT1]{$\tilde{H}$}.
  \end{itemize}

  Par abus de notations, on dit aussi que $\rev: \tilde{M} \to M(F)$ muni des données $(K,s)$ est un revêtement non ramifié.
\end{definition}

La dernière condition technique sert à garantir la commutativité de l'algèbre de Hecke, ce qui fait l'objet du paragraphe suivant. Observons aussi que les $F$-tores déployés maximaux en bonne position relativement à $K$ sont conjugués par $K$, d'après \cite[7.4.9 (i)]{BT72}.

\begin{lemma}\label{prop:compatibilite-K-unipotent}
  Si la condition \ref{def:cond-nr} est vérifiée, alors le scindage unipotent \ref{prop:scindage-unip} coïncide avec $s$ sur $K \cap M_\text{unip}(F)$.
\end{lemma}
\begin{proof}
  Il suffit de le vérifier sur $K \cap U(F)$ où $U$ est un sous-groupe unipotent quelconque. Notons $p$ la caractéristique résiduelle de $F$. Comme $U(F)$ est une union croissante de pro-$p$-groupes, $K \cap U(F)$ est un pro-$p$-groupe. Donc l'application $u \mapsto u^m$ est un homéomorphisme de $K \cap U(F)$ sur lui-même car $m$ est premier à $p$. Vu la construction du scindage unipotent, on voit qu'il n'existe qu'un seul scindage possible de $\rev$ au-dessus de $K \cap U(F)$.
\end{proof}

\begin{remark}
  Soit $\rev: \tilde{M} \to M(F)$ un revêtement à $m$ feuillets. Soit $\bmu_m \to \bmu_{m'}$ un homomorphisme quelconque et posons $\rev': \tilde{M}' \to M(F)$ la poussée-en-avant de $\rev$ via $\bmu_m \to \bmu_{m'}$. Alors le triplet $(\rev', K, s)$ vérifie la condition non ramifiée si $(\rev, K, s)$ la vérifie.
\end{remark}

\begin{remark}
  Soit $\rev: \tilde{G} \to G(F)$ un revêtement et $(\rev, K, s)$ un triplet vérifiant la condition non ramifiée pour $\tilde{G}$. Soient $M$ un sous-groupe de Lévi en bonne position relativement à $M$ et $\rev_M : \tilde{M} \to M(F)$ le revêtement induit. Alors $(\rev_M, K \cap M(F), s|_{K \cap M(F)})$ satisfait aussi à la condition non ramifiée pour $\tilde{M}$.
\end{remark}

\subsection{Isomorphisme de Satake}
Considérons un revêtement $\rev: \tilde{G} \to G(F)$ avec un sous-groupe hyperspécial $K$ et un scindage $s: K \to \tilde{G}$ vérifiant la condition non ramifiée. Nous allons établir une variante de l'isomorphisme de Satake\index[iFT1]{isomorphisme de Satake}.

Définissons le support de l'algèbre de Hecke sphérique anti-spécifique par
$$ \Supp (\mathcal{H}_{\asp}(\tilde{G}/\!/ K)) := \bigcup \left\{ \Supp(f) : f \in \mathcal{H}_{\asp}(\tilde{G}/\!/ K) \right\}.$$

Fixons désormais un $F$-tore déployé maximal $T$ en bonne position relativement à $K$. Alors $M_0 := Z_G(T)$ est un sous-groupe de Lévi minimal de $G$; de plus, $M_0$ est un $F$-tore non ramifié. Posons $K_0 := K \cap M_0(F)$. Définissons $\tilde{H} \subset \tilde{T}$ comme dans \ref{def:cond-nr}.

\begin{lemma}[cf. {\cite[9.2]{Na09}} ]\label{prop:Hecke-supp}
  On a $\Supp (\mathcal{H}_{\asp}(\tilde{G}/\!/ K)) = K\tilde{H}K$.
\end{lemma}
\begin{proof}
  Dans \cite{Na09} on ne considère que les groupes déployés, or la même preuve s'adapte aux groupes réductifs connexes non ramifiés sans modification.
\end{proof}

\begin{remark}
  C'est loisible d'identifier $W_0^G$ à $(N_G(T)(F) \cap K)/K_0$. Comme $K_0$ centralise $\tilde{H}$, on voit que $W_0^G$ opère sur $\tilde{H}$. D'autre part, \ref{prop:Hecke-supp} appliqué à $\widetilde{M_0}$ et $K_0$ affirme que
  $$ \Supp (\mathcal{H}_{\asp}(\widetilde{M_0} /\!/ K_0)) = \tilde{H} $$
  (on peut aussi le vérifier directement). Cela permet de faire opérer $W_0^G$ sur $\mathcal{H}_{\asp}(\widetilde{M_0} /\!/ K_0)$ de façon canonique.
\end{remark}

Posons
\begin{align*}
  \Lambda & := \{\lambda \in X_*(T) : \lambda(\varpi_F) \in \rev(\tilde{H}) \},
\end{align*}
Alors $\Lambda$ est un sous-réseau de $X_*(T)$ ayant le même rang; en effet, $\Lambda \supset m X_*(T)$.

\begin{lemma}\label{prop:Hecke-commutatif-tore}
  L'algèbre $\mathcal{H}_{\asp}(\widetilde{M_0} /\!/ K_0)$ est commutative. De plus, elle est isomorphe à l'algèbre $\C[\Lambda]$, ce qui s'identifie à l'algèbre en polynômes de $\dim X_*(T)$ variables.
\end{lemma}
\begin{proof}
  Il suffit de considérer le support de $\mathcal{H}_{\asp}(\widetilde{M_0} /\!/ K_0)$. On a déjà remarqué que $\Supp(\mathcal{H}_{\asp}(\widetilde{M_0} /\!/ K_0)) = \tilde{H}$, qui est commutatif selon \ref{def:cond-nr}.

  Choisissons une $\Z$-base $\lambda_1, \ldots, \lambda_r$ de $\Lambda$. Pour tout $1 \leq i\ \leq r$, prenons une fonction $f_i \in \mathcal{H}_{\asp}(\widetilde{M_0} /\!/ K_0)$ à support dans $K_0 \rev^{-1}(\lambda(\varpi_F)) K_0$. Alors $\lambda_i \mapsto f_i$ se prolonge en un isomorphisme $\C[\Lambda] \rightiso \mathcal{H}_{\asp}(\widetilde{M_0} /\!/ K_0)$.
\end{proof}

Fixons $P_0=M_0 U_0 \in \mathcal{P}(M_0)$. Prenons la mesure de Haar sur $U_0(F)$ telle que $\mes(K \cap U_0(F))=1$. Définissons l'application
\begin{align*}
  \mathcal{S}: & \mathcal{H}_{\asp}(\tilde{G} /\!/ K) \to \mathcal{H}_{\asp}(\widetilde{M_0} /\!/ K_0) \\
  & \mathcal{S}(f)(\tilde{x}) = \delta_{P_0}(x)^{-\frac{1}{2}} \int_{U_0(F)} f(u\tilde{x}) \dd u .
\end{align*}

On vérifie que $\mathcal{S}$ est un homomorphisme de $\C$-algèbres et il est à image dans $\mathcal{H}_{\asp}(\widetilde{M_0} /\!/ K_0)^{W_0^G}$. Les arguments sont identiques à ceux pour le cas des groupes réductifs, cf. \cite[\S 4]{Car79} \S 4. Par contre, le lemme suivant fait intervenir le revêtement.

\begin{lemma}\label{prop:UtK}
  Soit $\tilde{t} \in \tilde{T}$ tel que $|\alpha(t)| \geq 1$ pour toute racine positive $\alpha$ pour $(T, P)$. Alors on a
  $$ K \tilde{t} K \cap U_0(F) \tilde{t} K = \tilde{t} K . $$
  Si $u \in U_0(F)$, $k \in K$ satisfont à $u\tilde{t} = \tilde{t} k$, alors $u \in U_0(F) \cap K$.
\end{lemma}
\begin{proof}
  L'inclusion $K \tilde{t} K \cap U_0(F) \tilde{t} K \supset \tilde{t} K$ est claire. Prouvons l'autre inclusion dans le premier énoncé. Dans $G(F)$ on a $KtK \cap U_0(F)tK = tK$ d'après \cite[4.4.4]{BT72}. Soient $u \in U_0(F)$ et $k \in K$ tels que $u\tilde{t}k \in K \tilde{t} K \cap U_0(F) \tilde{t} K$. Il existe donc $\noyau \in \bmu_m$ et $k' \in K$ tels que $u\tilde{t}k = \noyau \tilde{t}k'$. Posons $k'' := k' k^{-1}$, alors
  $$ u\tilde{t} = \noyau \tilde{t}k'', $$
  ou encore
  $$ t^{-1} u t = \noyau k'' . $$
  D'après l'invariance du scindage unipotent \ref{prop:scindage-unip} et la compatibilité \ref{prop:compatibilite-K-unipotent}, on a $\noyau = 1$. Cela prouve à la fois les deux énoncés voulus.
\end{proof}

\begin{proposition}
  On a l'isomorphisme d'algèbres
  $$ \mathcal{H}_{\asp}(\tilde{G} /\!/ K) \xrightarrow{\mathcal{S}} \mathcal{H}_{\asp}(\widetilde{M_0} /\!/ K_0)^{W_0^G}. $$
\end{proposition}
\begin{proof}
  Il suffit de reprendre la démonstration usuelle de l'isomorphisme de Satake sauf qu'il faut utiliser \ref{prop:UtK}. Plus précisément, soient $\lambda, \lambda' \in X_*(T)$, écrivons $\lambda \leq_{P_0} \lambda'$ si $\angles{\alpha,\lambda' - \lambda} \geq 0$ pour tout $\alpha \in \Delta_0$. Posons
  $$\Lambda^- := \{\lambda \in \Lambda : \lambda \leq_{P_0} 0 \}.$$

  Pour tout $\tilde{t} \in \tilde{T} \cap \tilde{H}$, on peut prendre $f_{\tilde{t}}$ l'élément de $\mathcal{H}_{\asp}(\tilde{G} /\!/ K)$ à support dans $K\tilde{t}K$ tel que $f_{\tilde{t}}(\tilde{t})=1$ d'après \ref{prop:Hecke-supp}. Pour tout $\lambda \in \Lambda^-$, sélectionnons une image réciproque $\tilde{t}$ de $t = \lambda(\varpi_F) \in T$ et posons $f_{\lambda} := f_{\tilde{t}}$. D'après la décomposition de Cartan et \ref{prop:Hecke-supp}, on voit que $\mathcal{B} := \{f_\lambda : \lambda \in \Lambda^- \}$ est une base pour $\mathcal{H}_{\asp}(\tilde{G} /\!/ K)$.

  La même construction fournit une base $\{g_{\lambda} : \lambda \in \Lambda \}$ pour $\mathcal{H}_{\asp}(\widetilde{M_0} /\!/ K_0)$. Pour tout $[\lambda] \in \Lambda/W_0^G$, posons
  $$ f_{[\lambda]} := \sum_{\lambda \in [\lambda]} g_{\lambda}. $$
  Alors $\mathcal{B}_0 := \{f_{[\lambda]} : [\lambda] \in \Lambda/W_0^G\}$ est une base pour $\mathcal{H}_{\asp}(\widetilde{M_0} /\!/ K_0)^{W_0^G}$. Chaque $W_0^G$-orbite dans $\Lambda$ rencontre $\Lambda^-$ en un et un seul point, par conséquent $\mathcal{B}$ et $\mathcal{B}_0$ sont en bijection canonique.

  Sélectionnons un ordre total $\leq$ sur $X_*(T)$ tel que $\lambda \leq_{P_0} \lambda'$ entraîne $\lambda \leq \lambda'$. Identifions $X_*(T)$ et $T(F)/T(F)\cap K$ à l'aide de $\lambda \mapsto \lambda(\varpi_F)$ et notons $\nu: T(F) \to X_*(T)$ l'homomorphisme ainsi obtenu. Dans $G(F)$, on a
  $$ \forall t, t' \in T(F), \quad Kt'K \cap U_0(F)tK \neq \emptyset \Rightarrow  \nu(t) \leq_{P_0} \nu(t') $$
  d'après \cite[4.4.4]{BT72}. Il en résulte que $\mathcal{S}$ s'écrit dans les bases $\mathcal{B}, \mathcal{B}_0$ comme
  $$ \mathcal{S}f_{\lambda'} = \sum_{\lambda \leq \lambda'} c(\lambda,\lambda') f_{[\lambda]}, \quad c(\lambda,\lambda') \in \C^\times.$$

  C'est une matrice triangulaire inférieure. Montrons qu'il n'y a pas de zéro dans la diagonale. Étant fixé $\lambda' \in \Lambda^-$, sélectionnons $\tilde{t}' \in \rev^{-1}(\lambda'(\varpi_F))$ comme précédemment, alors $\tilde{t}'$ satisfait à l'hypothèse de \ref{prop:UtK}. Maintenant \ref{prop:UtK} entraîne que
  $$
    (\mathcal{S}f_{\lambda'})(\tilde{t}') = \delta_{P_0}(t')^{-\frac{1}{2}} \int_{U_0(F)} f_{\lambda'}(u\tilde{t}') \dd u = \delta_{P_0}(t')^{-\frac{1}{2}} \int_{U_0(F) \cap K} 1 \dd u.
  $$
  Cela entraîne que $c(\lambda',\lambda') = \delta_{P_0}(t')^{-\frac{1}{2}} \neq 0$, ce qu'il fallait démontrer.
\end{proof}

Vu \ref{prop:Hecke-commutatif-tore}, on en déduit
\begin{corollary}
  L'algèbre $\mathcal{H}_{\asp}(\tilde{G} /\!/ K)$ est commutative de type fini sur $\C$.
\end{corollary}

\subsection{Le cas adélique}
Considérons un corps de nombres $F$ et posons $\A = \Resprod_v F_v$ son anneau d'adèles. Notons $V_F$ l'ensemble de places de $F$. Notons $V_\infty := \{v \in V_F : v|\infty \}$. Pour $S \subset V_F$, nous utilisons l'indice $S$ (eg. $F_S$, $\pi_S$, $f_S$) pour signifier les composantes $v \in S$ et l'exposant $S$ (eg. $F^S$, $\pi^S$, $f^S$) pour signifier les composantes $v \notin S$.

Comme dans le cas local, on se donne un $F$-groupe réductif $M$, un entier $m$ et on considère une extension centrale de groupes topologiques
$$ 1 \to \bmu_m \to \tilde{M} \xrightarrow{\rev} M(\A) \to 1. $$
Soit $S \subset V_F$ fini. Notons $\rev_S : \tilde{M}_S \to M(F_S)$ la fibre de $\rev$ au-dessus de $M(F_S)$. Lorsque $S=\{v\}$ on écrit tout simplement $\rev_v : \tilde{M}_v \to M(F_v)$, c'est un revêtement de $M(F_v)$.

On dit que $\rev: \tilde{M} \to M(\A)$ est un revêtement à $m$ feuillets si l'on se donne les données\index[iFT1]{revêtement adélique}
\begin{itemize}
  \item une immersion $\mathbf{i}: M(F) \to \tilde{M}$ qui scinde $\rev$ au-dessus de $M(F)$;
  \item un ensemble fini de places $V_\text{ram} \supset V_\infty$;
  \item un modèle lisse et connexe de $M$ sur $\mathfrak{o}_\text{ram}$, l'anneau de $(V_\text{ram} \setminus V_\infty)$-entiers dans $F$;
\end{itemize}
vérifiant les conditions suivantes
\begin{enumerate}\renewcommand{\labelenumi}{(G\arabic{enumi})}
  \item\label{enu:G1} pour toute $v \notin V_\text{ram}$, posons $K_v := M(\mathfrak{o}_v)$, alors il existe un scindage continu $s_v: K_v \to \tilde{M}_v$ que l'on fixe;
  \item\label{enu:G2} le triplet $(\rev_v: \tilde{M}_v \to M(F_v), K_v, s_v)$ vérifie la condition non ramifiée \ref{def:cond-nr};
  \item\label{enu:G3} pour tout voisinage $\tilde{\mathcal{V}}$ de $1$ dans $\tilde{M}$, il existe un ensemble fini de places $S \supset V_\text{ram}$ tel que $s_v(K_v) \subset \tilde{\mathcal{V}}$ pour tout $v \notin S$.
\end{enumerate}\renewcommand{\labelenumi}{\arabic{enumi}}
Ces propriétés passent aux sous-groupes de Lévi et sont stables par pousser-en-avant en $\bmu_m$.

Le lemme \ref{prop:compatibilite-K-unipotent} permet de définir le scindage unipotent adélique $M_\text{unip}(\A) \to \tilde{M}$ en rassemblant les scindages unipotents locaux. Vu la construction du scindage unipotent, le résultat suivant est clair.

\begin{lemma}
  Le scindage $\mathbf{i}$ et le scindage unipotent coïncident sur $M_\mathrm{unip}(F)$.
\end{lemma}

Nous supprimons systématiquement les symboles $\mathbf{i}$ et $(s_v)_{v \notin V_\text{ram}}$, et nous regardons $G(F)$ et $K_v$ comme des sous-groupes de $\tilde{M}$. 

Soit $S \subset V_F$, on a l'isomorphisme canonique
$$ \left( \Resprod_{v \in S} \tilde{M}_v  \right) / \mathbf{N}_S \rightiso \tilde{M}_S $$
où le produit restreint est pris par rapport aux $K_v$ pour $v \notin V_\text{ram}$ si $|S| = \infty$, et
$$ \mathbf{N}_S := \left\{ (\noyau_v)_{v \in S} \in \bigoplus_{v \in S} \bmu_m : \prod_v \noyau_v = 1 \right\}. $$
Lorsque $S=V_F$, posons tout simplement $\mathbf{N} = \mathbf{N}_S$.

Le choix de $V_\text{ram}$, le $\mathfrak{o}_\text{ram}$-modèle lisse et le relèvement de $K_v$ n'ont pas d'importance essentielle, quitte à passer à un ensemble fini de places plus grand. Étant donné un revêtement adélique, nous supposons fixées des telles données dans l'article.

Un élément $\tilde{x} \in \tilde{M}_S$ s'exprime comme $[\tilde{x}_v]_{v \in S}$, où $(\tilde{x}_v)_{v \in S}$ est un représentant de $\tilde{x}$ dans $\Resprod_{v \in S} \tilde{M}_v$. Par abus de notation, on écrit les décompositions tensorielles $\pi = \bigotimes_{v \in S} \pi_v$ pour des représentations irréductibles admissibles (resp. $f = \prod_{v \in S} f_v$ pour des fonctions) sur $\tilde{M}_S$, où $\pi_v$ (resp. $f_v$) sont des représentations (resp. fonctions) sur $\tilde{M}_v$, bien que $\pi$ et $f$ sont définies sur le quotient $\Resprod_{v \in S} \tilde{M}_v / \mathbf{N}_S$. Soit $\chi \in \widehat{\bmu_m}$, alors $f=\prod_v f_v$ est $\chi$-équivariant si et seulement si chaque $f_v$ l'est. Idem pour les représentations.

Les mêmes conventions de mesures de \S\ref{sec:mesure} s'imposent dans ce cadre; nous demandons de plus que
\begin{itemize}
  \item si $v \notin V_\text{ram}$, on utilise la mesure sur $M(F_v)$ pour laquelle $\mes_{M(F_v)}(K_v)=1$;
  \item pour tout sous-groupe unipotent $U \subset M$, on prend la mesure sur $U(\A)$ pour laquelle
  $$\mes(U(F)\backslash U(\A))=1. $$
\end{itemize}
De tels choix sont possibles. On a les mêmes formules d'intégration comme précédemment.

\subsection{L'application de Harish-Chandra: le cas adélique}\label{sec:app-HC-global}
L'application de Harish-Chandra s'adapte au cas adélique: soient $G$ un $F$-groupe réductif connexe et $M$ un sous-groupe de Lévi. Définissons $H_M: M(\A) \to \mathfrak{a}_M$ par\index[iFT1]{$H_M$}\index[iFT1]{application de Harish-Chandra}
$$ \forall \chi \in X^*(M), \quad \angles{\chi, H_M(x)} = \log |\chi(x)| $$
où $|\cdot| = \prod_v |\cdot|_v$ est la valeur absolue adélique. Notons $M(\A)^1 := \Ker(H_M)$, alors $M(F) \subset M(\A)^1$.

\begin{definition}
  On dit qu'un sous-groupe compact maximal $K=\prod_v K_v$ de $G(\A)$ est en bonne position (ou réciproquement) relativement à $M$ si $K_v$ l'est pour tout $v$.\index[iFT1]{en bonne position}
\end{definition}
Fixons un tel sous-groupe compact maximal $K$. Soit $P=MU \in \mathcal{P}(M)$, on obtient l'application $H_P: G(\A) \to \mathfrak{a}_M$ en posant
$$ H_P(umk) = H_M(m), \quad u \in U(\A), m \in M(\A), k \in K. $$

Soit $\rev: \tilde{G} \to G(\A)$ un revêtement adélique. Prenons les groupes $K$ et $P=MU$ comme précédemment et notons $\tilde{M}=\rev^{-1}(M(\A))$, $\tilde{K}=\rev^{-1}(K)$. On a la formule d'intégration (cf. \eqref{eqn:UMK-rev})
$$ \int_{\tilde{G}} f(\tilde{x}) \dd \tilde{x} = \iiint_{U(\A) \times \tilde{M} \times \tilde{K}} f(u\tilde{m}\tilde{k}) \delta_P(m)^{-1} \dd\tilde{k} \dd\tilde{m} \dd u . $$

En composant les applications $H_M$, $H_P$ ci-dessus avec $\rev$, on obtient leurs avatars sur le revêtement, notés encore $H_M$, $H_P$. Posons $\tilde{G}^1 = \Ker(H_G) = \rev^{-1}(G(\A)^1)$, on a $G(F) \subset \tilde{G}^1$. Les notions de domaines de Siegel, hauteurs etc. se généralisent à cette situation. Cela permet de développer la théorie des formes automorphes et la décomposition spectrale sur les revêtements (voir \cite[I.2]{MW94}).

Signalons une décomposition utile pour l'étude de la formule des traces. Posons $F_\infty := \prod_{v|\infty} F_v$. Le $F$-tore $A_G$ étant déployé, il est l'extension des scalaires d'un $\Q$-tore déployé $A_{G,\Q}$. L'immersion canonique $\R \to F_\infty$ fournit une immersion $A_{G,\Q}(\R) \hookrightarrow A_G(F_\infty)$. Notons $A_{G,\infty}$ la composante neutre de $A_{G,\Q}(\R)$ pour la topologie usuelle. On vérifie que $G(\A) = G(\A)^1 \times A_{G,\infty}$. Si $M$ est un sous-groupe de Lévi, alors il existe une immersion canonique $A_{G,\infty} \hookrightarrow A_{M,\infty}$. Rappelons que $A_{G,\infty}$ est un produit de $\R_{>0}$, donc simplement connexe. Cela permet de relever la décomposition ci-dessus canoniquement au revêtement: $\tilde{G} = \tilde{G}^1 \times A_{G,\infty}$.

\subsection{$\mathbf{K}_2$-torseurs multiplicatifs de Brylinski-Deligne}\label{sec:BD-ext}\index[iFT1]{extension de Brylinski-Deligne}
Montrons que les revêtements provenant des $\mathbf{K}_2$-torseurs multiplicatifs de Brylinski et Deligne \cite{BD01} satisfont nos hypothèses pour un revêtement adélique. Dans le cas $G$ simplement connexe et déployé, ce sont exactement les extensions considérées dans \cite{Ma69}. Rappelons très brièvement la construction.

Soit $S$ un schéma quelconque, notons $S_\text{Zar}$ le gros site de Zariski associé. La $K$-théorie de Quillen fournit des faisceaux en groupes $(\mathbf{K}_n)_{n \geq 0}$ sur $S_\text{Zar}$; notons que $\mathbf{K}_1 = \Gm$. Soit $G$ un $S$-schéma en groupes réductif connexe. Une extension centrale par $\mathbf{K}_2$ est alors un $\mathbf{K}_2$-torseur $\tilde{G}(\cdot)$ sur $G$ muni d'une structure multiplicative convenable (voir \cite[\S 1]{BD01}) dans la catégorie des faisceaux en groupes sur $S_\text{Zar}$. Nous l'appelons un $\mathbf{K}_2$-torseur multiplicatif. Lorsque $S$ est régulier de type fini sur un corps, la catégorie de ces torseurs est concrètement décrite dans \cite{BD01}.

Dans ce qui suit, $G$ désigne toujours un groupe réductif connexe sur la base en question et $\tilde{G}(\cdot)$ désigne un $\mathbf{K}_2$-torseur multiplicatif sur $G$. Puisque $\tilde{G}(\cdot) \to G$ est un torseur pour la topologie de Zariski, si $S$ est le spectre d'un corps ou d'un anneau à valuation discrète, alors $\mathbf{K}_2(S) \hookrightarrow \tilde{G}(S) \twoheadrightarrow G(S)$ est une extension centrale de groupes.

\paragraph{Revêtements locaux}
Prenons $F$ un corps local, $X := \Spec F$. Le théorème de Matsumoto assure que $K_2(F)$ est l'objet initial de la catégorie des applications (dites ``symboles'')
$$ \{\cdot,\cdot\}: F^\times \times F^\times \to A, \quad A: \text{ un groupe abélien}, $$
tel que $\{\cdot,\cdot\}$ est bi-multiplicatif, alterné et $\{x,y\}=1$ lorsque $x+y=1$. Ainsi on peut parler des symboles localement constants sur $F^\times \times F^\times$. D'après un théorème de Moore, cette sous-catégorie admet un objet initial $K_2^\text{cont}(F)$ muni d'un homomorphisme naturel $K_2(F) \to K_2^\text{cont}(F)$. Désignons le groupe de racines d'unité dans $F$ par $\mu(F)$. On a
$$ K_2^\text{cont}(F) = \begin{cases} \{1\}, & \text{ si } F=\C ; \\ \mu(F), & \text{sinon}.  \end{cases} $$

Posons $n_F := |K_2^\text{cont}(F)|$. Alors $F^\times \times F^\times \to K_2^\text{cont}(F)$ s'identifie au $n_F$-ième symbole de Hilbert, et $K_2^\text{cont}(F) \simeq \bmu_{n_F}$.

Soit $\tilde{G}(\cdot)$ un $\mathbf{K}_2$-torseur multiplicatif sur $G$. On prend les $F$-points et on obtient une extension centrale $\tilde{G}(F)$ de $G(F)$ par $K_2(F)$. On la pousse via $K_2(F) \to K_2^\text{cont}(F)$. De la structure de torseur se déduisent des trivialisations locales (pour la topologie de Zariski) de cette extension centrale. Ces cartes se recollent via des sections locales de $\mathbf{K}_2$ sur $G$, qui fournissent des fonctions dans $K_2^\text{cont}(F)$ sur des ouverts de $G(F)$. Elles sont localement constantes sur $G(F)$ (pour la topologie induite par $|\cdot|_F$) d'après \cite[10.2]{BD01}, donc on obtient une extension centrale topologique $K_2^\text{cont}(F) \hookrightarrow \tilde{G} \twoheadrightarrow G(F)$. Si l'on choisit un isomorphisme $K_2^\text{cont}(F) \simeq \bmu_{n_F}$, alors on obtient un revêtement à $n_F$ feuillets de $G(F)$ selon la définition dans \S\ref{sec:revetement-local}.

\paragraph{Extension résiduelle}\index[iFT1]{extension résiduelle}
Décrivons la construction dans \cite[12.11]{BD01} qui sera bientôt utile. On se donne un corps local non archimédien $F$. Posons $V := \Spec(\mathfrak{o}_F)$, $\eta$ son point générique et $s$ son point spécial. Soient $X_V$ un $V$-schéma lisse, $X_\eta$ (resp. $X_s$) sa fibre générique (resp. spéciale) et $E$ un $\mathbf{K}_2$-torseur sur $X_\eta$. Avec les notations standards, on a les inclusions
$$ X_\eta \stackrel{j}{\hookrightarrow} X_V \stackrel{i}{\hookleftarrow} X_s .$$

Imposons d'abord la condition suivante:
\begin{description}
 \item[(\textasteriskcentered)] chaque point de $X_s$ admet un voisinage ouvert $U$ dans $X_V$ tel que $E$ se trivialise sur $U \cap X_\eta$.
\end{description}

Cette condition dit que $j_* E$ est un $j_* \mathbf{K}_2$-torseur sur $X_V$. Via l'homomorphisme de résidu $j_* \mathbf{K}_2 \to i_* \mathbf{K}_1 = i_* \Gm$, on obtient un $i_* \Gm$-torseur sur $X_V$, ou ce qui revient au même, un $\Gm$-torseur sur $X_s$. Notons-le $E_s$.

Prenons maintenant $X_V = G_V$ un schéma en groupes réductif, avec fibre générique $G$ et fibre spécial $G_s$. Prenons $E = \tilde{G}(\cdot)$ un $\mathbf{K}_2$-torseur multiplicatif sur $G$. Alors $\tilde{G}(\cdot)_s$ hérite la structure multiplicative: on obtient ainsi une extension centrale de $G_s$ par $\Gm$.

En général, la condition (\textasteriskcentered) est satisfaite quitte à passer à un revêtement étale $V' \to V$. On construit ainsi le $\Gm$-torseur $\tilde{G}(\cdot)_s$ par descente galoisienne. On l'appelle l'extension résiduelle de $\tilde{G}(\cdot)$. Si l'extension résiduelle est scindée, on dit que $\tilde{G}(\cdot)$ est résiduellement scindé.

\paragraph{Revêtements adéliques}
Prenons $F$ un corps de nombres, $X := \Spec(\mathfrak{o}_F)$; posons $n_F := |\mu(F)|$. Soient $G$ un $F$-groupe réductif connexe et  $\tilde{G}(\cdot)$ un $\mathbf{K}_2$-torseur sur $G$.

Prenons $S_1$ un ensemble fini de points fermés de $X$ (i.e. des places non archimédiennes). Notons $S$ l'union de $S_1$ avec les places archimédiennes de $F$. Pour $S_1$ suffisamment grand, on peut supposer que:
\begin{itemize}
  \item $G$ admet un modèle lisse sur $X \setminus S_1$;
  \item $\tilde{G}(\cdot)$ est la fibre générique d'un $\mathbf{K}_2$-torseur sur $G$ défini sur $X \setminus S_1$, notée encore $\tilde{G}(\cdot)$ (voir \cite[10.5]{BD01});
  \item $n_{F_v}$ est premier avec la caractéristique résiduelle de $F_v$ pour tout $v \in X \setminus S_1$.
\end{itemize}
Soient $(S_1, G, \tilde{G}(\cdot))$ et $(S'_1, G', \tilde{G}'(\cdot))$ deux données comme ci-dessus, alors elles deviennent isomorphes si l'on se restreint à $X \setminus S''_1$ où $S''_1 \supset S_1 \cup S'_1$ est fini et assez grand.

Soit $v$ une place de $F$, on construit l'extension centrale topologique
\begin{gather}\label{eqn:BD-local}
  1 \to K_2^\text{cont}(F_v) \to \tilde{G}_v \to G(F_v) \to 1 .
\end{gather}

D'autre part, \cite[10.6]{BD01} affirme que $H^1(X \setminus S_1, \mathbf{K}_2)=0$, d'où une extension centrale
\begin{gather}\label{eqn:BD-S-global}
  1 \to H^0(X \setminus S_1, \mathbf{K}_2) \to \tilde{G}(X \setminus S_1) \to G(X \setminus S_1) \to 1.
\end{gather}
Pour toute place $v$, il y a un morphisme naturel de \eqref{eqn:BD-S-global} dans \eqref{eqn:BD-local}. Lorsque $v \in X \setminus S_1$, ce morphisme se factorise via
$$\xymatrix{
  1 \ar[r] & K_2(\mathfrak{o}_v) \ar[r] \ar[d] & \tilde{G}(\mathfrak{o}_v) \ar[r] \ar[d] & G(\mathfrak{o}_v) \ar[d] \ar[r] & 1 \\
  1 \ar[r] & K_2^\text{cont}(F_v) \ar[r] & \tilde{G}_v \ar[r] & G(F_v) \ar[r] & 1 .
}$$

Or $n_{F_v}$ est premier avec la caractéristique résiduelle de $F_v$, donc la composée $K_2(\mathfrak{o}_v) \to K_2(F_v) \to K_2^\text{cont}(F_v)$ est triviale et ce diagramme fournit un scindage de \eqref{eqn:BD-local} au-dessus de $G(\mathfrak{o}_v)$.

Réunissant ce que l'on a obtenu, il a un diagramme commutatif avec lignes exactes
$$\xymatrix{
  1 \ar[r] & H^0(X \setminus S_1, \mathbf{K}_2) \ar[r] \ar[d] & \tilde{G}(X \setminus S_1) \ar[d] \ar[r] & G(X \setminus S_1) \ar[d] \ar[r] \ar@{-->}[ldd] & 1 \\
  1 \ar[r] & \prod_{\substack{v \in S \\ F_v \neq \C }} \mu(F_v) \ar[d]_{\alpha_S} \ar[r] & \prod_{v \in S} \tilde{G}_v \times \prod_{v \in X \setminus S_1} G(\mathfrak{o}_v) \ar[d] \ar[r] & \prod_{v \in S} G(F_v) \times \prod_{v \in X \setminus S_1} G(\mathfrak{o}_v) \ar@{=}[d] \ar[r] & 1 \\
  1 \ar[r] & \mu(F) \ar[r] & \tilde{G}_S \times \prod_{v \in X \setminus S_1} G(\mathfrak{o}_v) \ar[r] & \prod_{v \in S} G(F_v) \times \prod_{v \in X \setminus S_1} G(\mathfrak{o}_v) \ar[r] & 1 
}$$
où $\alpha_S((\zeta_v)_{\substack{v \in S \\ F_v \neq \C}}) = \prod_v \zeta_v^{[\mu(F_v):\mu(F)]}$ et la dernière ligne s'obtient de la deuxième en poussant-en-avant via $\alpha_S$. La flèche $\dashrightarrow$ provient du fait que l'application $H^0(X \setminus S_1, \mathbf{K}_2) \to \mu(F)$ ainsi obtenue est triviale, ce qu'assure la réciprocité de Moore \cite[(10.4.2)]{BD01}.

On passe à la limite par rapport à $S_1$ et on obtient une extension centrale topologique $\mu(F) \hookrightarrow \tilde{G} \stackrel{\rev}{\twoheadrightarrow} G(\A)$. Elle se scinde canoniquement au-dessus de $\prod_{v \in X \setminus S_1} G(\mathfrak{o}_v)$ et au-dessus de $G(F)$ (à l'aide de $\dashrightarrow$). Enfin, on peut identifier $\mu(F)$ et $\bmu_{n_F}$, mais il n'y a pas de choix canonique. On vérifie sans peine que $\rev$ satisfait aux conditions d'un revêtement adélique (avec $V_\text{ram} := S$) sauf la commutativité du groupe $\tilde{H}$ dans \eqref{eqn:def-H}, ce qui fait l'objet du paragraphe suivant.

Remarquons aussi que, pour toute place $v$, la fibre locale $\rev_v : \tilde{G}_v \to G(F_v)$ de $\rev$ est la poussée-en-avant de \eqref{eqn:BD-local} via $\bmu_{n_{F_v}} \to \bmu_{n_F}$, $\zeta \mapsto \zeta^{[\mu(F_v):\mu(F)]}$.

\paragraph{Vérification des hypothèses}
Plaçons-nous dans le cas $F$ un corps de nombres avec $G$, $\tilde{G}(\cdot)$ comme précédent. On construit l'extension centrale topologique $\mu(F) \hookrightarrow \tilde{G} \stackrel{\rev}{\twoheadrightarrow} G(\A)$. Identifions $\mu(F)$ et $\bmu_{n_F}$ en fixant un générateur de $\mu(F)$. Le but est l'énoncé suivant.

\begin{theorem}\label{prop:BD-revetement-global}
  Les données ci-dessus forment un revêtement de $G(\A)$ à $n_F$-feuillets.
\end{theorem}

Soit $S \supset V_\text{ram}$ un ensemble fini de places de $F$ vérifiant les conditions dans le paragraphe précédent. Soit $v \notin S$, on prend $K_v := G(\mathfrak{o}_v)$, $T_v$ un $F_v$-tore déployé maximal dans $G_v := G \times_F F_v$ en bonne position relativement à $K_v$, et posons $M_{0,v} := Z_{G_v}(T_v)$. C'est un $F$-tore maximal car $G$ est non ramifié. Définissons $\tilde{H}_v \subset \widetilde{M_{0,v}}$ comme dans \eqref{eqn:def-H}. D'après ce qui précède, il suffit de vérifier la commutativité de $\tilde{H}_v$ pour tout $v \notin S$ afin de prouver \ref{prop:BD-revetement-global}.

\begin{lemma}
  Conservons le formalisme ci-dessus. Si $\tilde{G}(\cdot)$ est résiduellement scindé en $v$, alors $\tilde{H}_v$ est commutatif.
\end{lemma}
\begin{proof}
  On se ramène aussitôt au cas $G=M_0$, qui est un $F_v$-tore non ramifié. Dans ce cas-là, c'est l'assertion de \cite[6.5]{We09}.
\end{proof}

\begin{lemma}
  Pour toute place $v \notin S$, $\tilde{G}(\cdot)$ est résiduellement scindé en $v$.
\end{lemma}
\begin{proof}
  Cf. \cite[12.14 (iii)]{BD01} 12.14. Rappelons que les conditions sur $S$ entraînent que $G$ admet un modèle lisse sur $V = \Spec(\mathfrak{o}_v)$ et $\tilde{G}(\cdot) \times_F F_v$ est la fibre générique d'un $\mathbf{K}_2$-torseur multiplicatif défini sur $V$, disons $\tilde{G}_V(\cdot)$.

  Dans cette situation (\textasteriskcentered) est évidemment satisfait. Donc l'extension résiduelle est obtenue en poussant-en-avant $\tilde{G}_V(\cdot)$ via les homomorphismes de faisceaux sur $G_V$:
  $$ \mathbf{K}_2 \to j_* \mathbf{K}_2  \to i_* \mathbf{K}_1 . $$

  Cela faisant partie de la suite de localisation en $K$-théorie, la composition est triviale. D'où la trivialité de l'extension résiduelle.
\end{proof}

\begin{proof}[Démonstration de \ref{prop:BD-revetement-global}]
  D'après les remarques après \ref{prop:BD-revetement-global}, il suffit de combiner les deux lemmes précédents.
\end{proof}

\section{La combinatoire}\label{sec:combinatoire}
Fixons un corps $F$, un $F$-groupe réductif connexe $G$, un sous-groupe de Lévi minimal $M_0$ et $P_0 \in \mathcal{P}(M_0)$.

\subsection{Analyse convexe}
Les résultats ici se trouvent dans \cite{Ar81}. Soient $P,Q$ deux sous-groupes paraboliques semi-standards de $G$ tels que $Q \supset P$, définissons des cônes ouverts dans $\mathfrak{a}_0$\index[iFT1]{${\mathfrak{a}^Q_P}^+$}\index[iFT1]{${}^+\mathfrak{a}^Q_P$}
\begin{align*}
  {\mathfrak{a}^Q_P}^+ & := \{ H \in \mathfrak{a}_0 : \forall \alpha \in \Delta^Q_P,\; \alpha(H) > 0 \}, \\
  {}^+\mathfrak{a}^Q_P & := \{H \in \mathfrak{a}_0 : \forall \varpi \in \widehat{\Delta^Q_P}, \; \varpi(H) > 0 \};
\end{align*}
et leurs fonctions caractéristiques\index[iFT1]{$\tau^Q_P$}\index[iFT1]{$\hat{\tau}^Q_P$}
\begin{align*}
  \tau^Q_P & := \mathbbm{1}_{{\mathfrak{a}^Q_P}^+} , \\
  \hat{\tau}^Q_P & := \mathbbm{1}_{{}^+\mathfrak{a}^Q_P}.
\end{align*}

Notons $\Z{\Delta^Q_P}^\vee$ (resp. $\Z\widehat{{\Delta^Q_P}^\vee}$) le réseau dans $\mathfrak{a}^Q_P$ engendré par ${\Delta^Q_P}^\vee$ (resp. $\widehat{{\Delta^Q_P}^\vee}$) et posons\index[iFT1]{$\theta^Q_P(\lambda)$}\index[iFT1]{$\hat{\theta}^Q_P(\lambda)$}
\begin{align*}
  \theta^Q_P(\lambda) & := \mes(\mathfrak{a}^Q_P/\Z{\Delta^Q_P}^\vee)^{-1} \prod_{\alpha^\vee \in {\Delta^Q_P}^\vee} \lambda(\alpha^\vee), \\
  \hat{\theta}^Q_P(\lambda) & := \mes(\mathfrak{a}^Q_P/\Z\widehat{{\Delta^Q_P}^\vee})^{-1} \prod_{\varpi^\vee \in \widehat{{\Delta^Q_P}^\vee}} \lambda(\varpi^\vee)
\end{align*}
pour tout $\lambda \in (\mathfrak{a}^Q_P)^*_\C$. Ce sont des fonctions holomorphes en $\lambda$. Lorsque $Q=G$, on supprime les exposants et on les note $\mathfrak{a}_P^+$, ${}^+\mathfrak{a}_P$, $\tau_P$, $\hat{\tau}_P$, $\theta_P$ et $\hat{\theta}_P$.

Étant donnés des sous-groupes paraboliques semi-standards $Q \supset Q' \supset P' \supset P$ et $Y \in \mathfrak{a}^Q_P$, notons $Y^{Q'}_{P'}$ l'image de $Y$ via $\mathfrak{a}^Q_P \to \mathfrak{a}^{Q'}_{P'}$. Lorsque $Q=Q'$ (resp. $P=P'$), on simplifie les notations en supprimant les exposants (resp. les indices) comme précédemment.

\begin{proposition}[cf. {\cite[2.1]{Ar81}}]
  Soient $R \supset P$ deux sous-groupes paraboliques semi-standards, on a
  $$ \sum_{Q: R \supset Q \supset P} (-1)^{\dim(A_P/A_Q)} \tau^Q_P(X^Q)\hat{\tau}^R_Q(X_Q) = \begin{cases} 1, & \text{si } P=R, \\ 0, & \text{sinon}. \end{cases} $$
\end{proposition}

\begin{proposition}[cf. {\cite[\S 2]{Ar81}}]
  Conservons les notations précédentes et posons
  $$ \Gamma^R_P(X,Y) := \sum_{Q: R \supset Q \supset P} (-1)^{\dim (A_Q/A_R)} \tau^Q_P(X^Q) \hat{\tau}^R_Q(X_Q-Y_Q), \quad X,Y \in \mathfrak{a}^R_P . $$
  Alors $\Gamma^R_P(\cdot,Y)$ est à support compact. On a la relation de récurrence suivante
  $$ \hat{\tau}^R_P(X-Y) = \sum_{Q: R \supset Q \supset P} (-1)^{\dim(A_Q/A_R)} \hat{\tau}^Q_P(X^Q) \Gamma^R_Q(X_Q, Y_Q). $$
\end{proposition}

Soit $E$ un espace de Banach. Soit $c_P: i\mathfrak{a}_P^* \to E$ une fonction, définissons\index[iFT1]{$c'_P$}
\begin{gather}\label{eqn:c'_P}
  c'_P(\lambda) := \sum_{Q \supset P} (-1)^{\dim(A_P/A_Q)} \hat{\theta}^Q_P(\lambda)^{-1} c_Q(\lambda_Q) \theta_Q(\lambda_Q)^{-1}
\end{gather}
où $c_Q := c|_{i\mathfrak{a}_Q^*}$. C'est bien défini sur le complément dans $i\mathfrak{a}_P^*$ des murs associés aux coracines et copoids simples.

\begin{proposition}[cf. {\cite[6.1]{Ar81}}]
  Si $c_P$ est lisse, alors $c'_P$ se prolonge en une fonction lisse sur $\lambda \in i\mathfrak{a}_P^*$.
\end{proposition}

\begin{proposition}[cf. {\cite[2.2]{Ar81}}]\label{prop:c'_P-polynome}
  S'il existe $X \in \mathfrak{a}_P$ tel que $c_P(\lambda) = e^{\lambda(X)}$, alors
  $$ c'_P(\lambda) = \int_{\mathfrak{a}^G_P} \Gamma^G_P(H,X^G) e^{\lambda(H)} \dd H, \quad \lambda \in i\mathfrak{a}_P^* .$$
  Cela étant la transformée de Fourier d'une fonction à support compact, $c'_P$ se prolonge en une fonction holomorphe sur $\mathfrak{a}_{P,\C}$. De plus, $c'_P(0)$ est un polynôme homogène en $X$ de degré $\dim(A_P/A_G)$.
\end{proposition}

\subsection{$(G,M)$-familles}\label{sec:GM-famille}
Passons en revue la définition et les propriétés de $(G,M)$-familles. Les références sont \cite[\S 6]{Ar81} et \cite[\S 7]{Ar88-TF1}.\index[iFT1]{$(G,M)$-famille}

\begin{definition}
  Soit $E$ un espace de Banach. Une $(G,M)$-famille à valeurs dans $E$ est une famille des fonctions lisses
  $$ c_P: i\mathfrak{a}_M^* \to E, \quad P \in \mathcal{P}(M) $$
  telle que pour tous $P,P' \in \mathcal{P}(M)$ adjacents et tout $\lambda \in i(\mathfrak{a}_M^*)_P^+ \cap i(\mathfrak{a}_M^*)_{P'}^+$, on a $c_P(\lambda)=c_{P'}(\lambda)$.
\end{definition}

\begin{proposition}
  La fonction\index[iFT1]{$c_M(\lambda), c_M$}
  $$ c_M(\lambda) := \sum_{P \in \mathcal{P}(M)} c_P(\lambda) \theta_P(\lambda)^{-1} $$
  est bien définie et lisse sur $i\mathfrak{a}^*_M$.
\end{proposition}
On en déduit des fonctions lisses $c'_P$ sur $i\mathfrak{a}_P$ selon \eqref{eqn:c'_P}. Posons $c_M := c_M(0)$, c'est le terme qui nous intéresse.

\begin{example}
  On dit qu'un ensemble $\mathcal{Y} = (Y_P)_{P \in \mathcal{P}(M)}$ de points dans $\mathfrak{a}_M$ indexé par $\mathcal{P}(M)$ est un ensemble $(G,M)$-orthogonal (resp. $(G,M)$-orthogonal positif) si pour tous $P, P' \in \mathcal{P}(M)$ adjacents séparés par $\alpha \in \Delta_P$, on a
  $$ Y_P - Y_{P'} \in \R\alpha^\vee \qquad (\text{resp.}\quad Y_P-Y_{P'} \in \R_{\geq 0}\alpha^\vee). $$
  À un tel ensemble $\mathcal{Y}$ est associée une $(G,M)$-famille
  $$ c_P(\lambda, \mathcal{Y}) = e^{\lambda(Y_P)}. $$

  Vu \ref{prop:c'_P-polynome}, on a
  $$ c'_P(\lambda, \mathcal{Y}) = \int_{\mathfrak{a}^G_P} \Gamma^G_P(H,Y_P^G) e^{\lambda(H)} \dd H. $$
\end{example}

Ci-dessous une récapitulation des opérations utiles. Soit $(c_P)_P$ une $(G,M)$-famille à valeurs dans $E$.
\begin{enumerate}
  \item Supposons que $E$ est une algèbre de Banach. Soit $(d_P)_P$ une autre $(G,M)$-famille à valeurs dans $E$. Posons $(cd)_P(\lambda) := c_P(\lambda)d_P(\lambda)$, alors $(cd)_P$ est encore une $(G,M)$-famille.
  \item Fixons $L \in \mathcal{L}(M)$. En rappelant que $\mathfrak{a}_L^* \hookrightarrow \mathfrak{a}_M^*$ canoniquement, posons
  $$ c_Q(\lambda) = c_P(\lambda), \quad Q \in \mathcal{P}(L), \lambda \in i\mathfrak{a}_L^* $$
  où $P \in \mathcal{P}(M)$ est tel que $P \subset Q$; on vérifie que $c_Q(\lambda)$ ne dépend pas du choix de $P$. Alors $(c_Q)_Q$ est une $(G,L)$-famille.
  \item Fixons $L \in \mathcal{L}(M)$ et $Q \in \mathcal{P}(L)$ comme ci-dessus. Si $R \in \mathcal{P}^L(M)$, notons $Q(R)$ l'unique élément de $\mathcal{P}(M)$ tel que $Q(R) \subset Q$ et $Q(R) \cap L = R$. Posons
  $$ c^Q_R(\lambda) := c_{Q(R)}(\lambda), \quad R \in \mathcal{P}^L(M), \lambda \in i\mathfrak{a}_M^* .$$
  Alors $(c^Q_R)_R$ est une $(L,M)$-famille. Lorsque les fonctions $c^Q_R$ ne dépendent pas de $Q$, on les note aussi $c^L_R$.
  \item\label{enu:satellite} Soient $F_1$ une extension de $F$ et $M_1$ un sous-groupe de Lévi de $G_1 := G \times_F F_1$. Supposons que $M \supset M_1$ sur $F_1$, d'où une inclusion canonique $\mathfrak{a}_M^* \hookrightarrow \mathfrak{a}_{M_1}^*$. Soit $(c_{P_1})_{P_1}$ une $(G_1,M_1)$-famille, posons
  $$ c_P(\lambda) := c_{P_1}(\lambda), \quad P \in \mathcal{P}(M), \lambda \in i\mathfrak{a}_M^* $$
  où $P_1 \in \mathcal{P}(M_1)$ est tel que $P_1 \subset P$ sur $F_1$; on vérifie que $c_P(\lambda)$ ne dépend pas du choix de $P_1$. Alors $(c_P(\lambda))_P$ est une $(G,M)$-famille.
\end{enumerate}

Dorénavant, les $(G,M)$-familles sont supposées à valeurs dans une algèbre de Banach fixée.

\begin{lemma}[{\cite[6.3]{Ar81}}]\label{prop:descente-cd'}
  On a
  $$ (cd)_M(\lambda) = \sum_{Q \in \mathcal{F}(M)} c^Q_M(\lambda^Q) d'_Q(\lambda_Q). $$
  En particulier,
  $$ (cd)_M = \sum_{Q \in \mathcal{F}(M)} c^Q_M d'_Q. $$
\end{lemma}
\begin{corollary}[{\cite[6.4 et 6.5]{Ar81}}]\label{prop:descente-d'}
  Soient $(c_P),(d_P)$ des $(G,M)$-familles.
  \begin{itemize}
    \item On a $d_M(\lambda) = \sum_{Q \in \mathcal{P}(M)} d'_Q(\lambda_Q)$.
    \item Supposons que $L \in \mathcal{L}(M)$ et $Q \in \mathcal{P}(L)$. Si la famille $(c^Q_R)$ ne dépend pas du choix de $Q$, alors
      $$ (cd)_M(\lambda) = \sum_{L \in \mathcal{L}(M)} c^L_M(\lambda^L) d_L(\lambda_L). $$
  \end{itemize}
\end{corollary}

Plaçons-nous dans la situation \ref{enu:satellite}. Soit $L \in \mathcal{L}(M_1)$. Si l'homomorphisme canonique
$$\Sigma: \mathfrak{a}^M_{M_1} \oplus \mathfrak{a}^{L_1}_{M_1} \to \mathfrak{a}^G_{M_1} $$
est un isomorphisme, posons\index[iFT1]{$d^G_{M_1}(M,L_1)$}
\begin{gather}\label{eqn:coef-d}
  d^G_{M_1}(M,L_1) := \dfrac{\text{la mesure sur } \mathfrak{a}^G_{M_1}}{\Sigma_* \left( \text{la mesure sur } \mathfrak{a}^M_{M_1} \oplus \mathfrak{a}^{L_1}_{M_1}\right)}
\end{gather}
en rappelant que l'on a fixé des mesures de Haar sur les espaces en question; sinon, posons $d^G_{M_1}(M,L_1):=0$.

Prenons
\begin{gather}\label{eqn:xi-1}
  \xi \in \mathfrak{a}^M_{M_1} \quad \text{en position générale.}
\end{gather}
Pour $L_1 \in \mathcal{L}(M_1)$ tel que $d^G_{M_1}(M,L_1) \neq 0$, on voit que $(\xi+\mathfrak{a}^G_M) \cap \mathfrak{a}^G_{L_1}$ consiste en un seul point non singulier; ce point appartient donc à $\mathfrak{a}_{Q_1}^+$ pour un unique $Q_1 \in \mathcal{P}(L_1)$. Cela définit une  application $L_1 \mapsto Q_1$ pour de tels $L_1$.

\begin{lemma}[{\cite[7.4]{Ar88-TF1}}]
  Avec le choix précédent de $\xi \in \mathfrak{a}^M_{M_1}$, on a
  $$ c_M(\lambda) = \sum_{L_1 \in \mathcal{L}(M_1)} d^G_{M_1}(M,L_1) c^{Q_1}_{M_1}(\lambda^{Q_1}), \quad \lambda \in i\mathfrak{a}_M^* .$$
  En particulier,
  $$ c_M = \sum_{L_1 \in \mathcal{L}(M_1)} d^G_{M_1}(M,L_1) c^{Q_1}_{M_1}. $$
\end{lemma}

Considérons maintenant une variante. Soient $L_1, L_2 \in \mathcal{L}(M)$, on dispose toujours d'une application canonique
$$ \Sigma: \mathfrak{a}_M^{L_1} \oplus \mathfrak{a}_M^{L_2} \to \mathfrak{a}_M^G. $$
Cela permet de définir le coefficient $d^G_M(L_1, L_2)$ comme en \eqref{eqn:coef-d}. De même, prenons
\begin{gather}\label{eqn:xi-2}
  \xi \in \mathfrak{a}^M_{\mathcal{M}} := \{(H,-H) : H \in \mathfrak{a}_M \}
\end{gather}
en position générale; ce choix fournit une application $(L_1,L_2) \mapsto (Q_1,Q_2)$ pour les $L_1, L_2$ avec $d^G_M(L_1,L_2) \neq 0$, et on a $Q_i \in \mathcal{P}(L_i)$, $i=1,2$.

\begin{lemma}[{\cite[7.4]{Ar88-TF1}}]
  Avec les notations précédentes, on a
  $$ (cd)_M(\lambda) = \sum_{L_1, L_2 \in \mathcal{L}(M)} d^G_M(L_1,L_2) c^{Q_1}_M(\lambda^{Q_1}) c^{Q_2}_M(\lambda^{Q_2}). $$
  En particulier,
  $$ (cd)_M = \sum_{L_1, L_2 \in \mathcal{L}(M)} d^G_M(L_1,L_2) c^{Q_1}_M c^{Q_2}_M. $$
\end{lemma}

\section{La formule des traces avec caractère: la partie unipotente}\label{sec:formule-traces-caractere}
Dans cette section, nous fixons les objets suivants
\begin{itemize}
  \item $F$: un corps de nombres,
  \item $\A$: l'anneau d'adèles de $F$,
  \item $G$: un $F$-groupe réductif connexe,
  \item $M_0$: un sous-groupe de Lévi minimal de $G$,
  \item $P_0 \in \mathcal{P}(M_0)$,
  \item $K = \prod_v K_v$: un sous-groupe compact maximal de $G(\A)$ en bonne position relativement à $M_0$,
  \item $\bomega: G(\A) \to \C^\times$ un caractère unitaire continu tel que $\bomega|_{G(F)}=1$.\index[iFT1]{$\bomega$}
\end{itemize}
On appelle un caractère $\bomega$ vérifiant la condition ci-dessus un caractère automorphe de $G$.\index[iFT1]{caractère automorphe}

De tels objets passent de façon évidente aux sous-groupes de Lévi standards, voire semi-standards si l'on ôte la donnée $P_0$. Fixons aussi des mesures de Haar selon les conventions de \S\ref{sec:mesure}.

Soit $T \in \mathfrak{a}_0$. Étant donné $P \in \mathcal{F}(M_0)$, par abus de notation, nous écrirons $T$ au lieu de $T_P$ pour désigner sa projection dans $\mathfrak{a}_P$.

\subsection{Le $\mathfrak{o}$-développement}\label{sec:trace-omega-geometrique}
Notons $R$ la représentation régulière de $G(\A)$ sur $L^2(G(F) \backslash G(\A)^1) = L^2(G(F)A_{G,\infty} \backslash G(\A))$, c'est-à-dire
$$ (R(y)\phi): x \mapsto \phi(xy), \quad y \in G(\A)^1, \phi \in L^2(G(F) \backslash G(\A)^1). $$

Notons $A_\bomega: L^2(G(F) \backslash G(\A)^1) \to L^2(G(F) \backslash G(\A)^1)$ l'application $\phi \mapsto \phi \bomega$. La formule des traces pour $(G,\bomega)$ concerne les opérateurs
$$ R(f) \circ A_\bomega : L^2(G(F) \backslash G(\A)^1) \to L^2(G(F) \backslash G(\A)^1), \quad f \in C_c^\infty(G(\A)^1). $$

Fixons $f$, alors $R(f) \circ A_\bomega$ admet le noyau
$$ K^\bomega(x,y) = \sum_{\gamma \in G(F)} \bomega(y) f(x^{-1} \gamma y), $$
cela signifie que $R(f) \circ A_\bomega$ est donné par $\phi \mapsto \int_{G(F) \backslash G(\A)^1} K^\bomega(\cdot, y) \phi(y) \dd y$.

Rappelons la procédure de troncature d'Arthur. Soient $T \in \mathfrak{a}_0^+$ et $P=M_P U_P \supset P_0$ un sous-groupe parabolique standard. Définissons
\begin{align*}
  K^\bomega_P(x,y) & := \bomega(y) \int_{U_P(\A)} \sum_{\gamma \in M_P(F)} f(x^{-1}\gamma u y) \dd u, \\
  k^{T,\bomega}(x) & := \sum_{P \supset P_0} (-1)^{\dim A_P/A_G} \sum_{\delta \in P(F) \backslash G(F)} K^\bomega_P(\delta x,\delta x) \hat{\tau}_P(H_P(\delta x) - T).
\end{align*}
Lorsque $\bomega=1$, on revient aux objets construits par Arthur \cite{Ar78} et on supprime l'exposant $\bomega$.

Remarquons que $K^\bomega_P(x,y) = \bomega(y) K_P(x,y)$ et $k^{T,\bomega}(x) = \bomega(x)k^T(x)$. Donc la somme définissant $k^{T,\bomega}$ est finie pour $x$ dans un sous-ensemble compact.

On dit que $\gamma_1, \gamma_2 \in G(F)$ sont $\mathcal{O}$-équivalents si leurs parties semi-simples sont conjuguées. Notons $\mathcal{O}$ l'ensemble de classes de $\mathcal{O}$-équivalences dans $G(F)$. Il est en bijection naturelle avec l'ensemble de classes de conjugaison semi-simples dans $G(F)$. Comme d'habitude, lorsqu'une ambiguïté sera à craindre sur $G$, on les notera $\mathcal{O}^G$-équivalence et $\mathcal{O}^G$.\index[iFT1]{$\mathcal{O}$-équivalence}

Soit $M$ un sous-groupe de Lévi de $G$, l'inclusion $M(F) \hookrightarrow G(F)$ induit une application $\mathcal{O}^M \to \mathcal{O}^G$ à fibres finies.

Soit $\mathfrak{o} \in \mathcal{O}$, définissons
\begin{align*}
  K^\bomega_{P,\mathfrak{o}}(x,y) & := \bomega(y) \int_{U_P(\A)} \sum_{\gamma \in M_P(F) \cap \mathfrak{o}} f(x^{-1}\gamma u y) \dd u, \\
  k^{T,\bomega}_{\mathfrak{o}}(x) & := \sum_{P \supset P_0} (-1)^{\dim A_P/A_G} \sum_{\delta \in P(F) \backslash G(F)} K^\bomega_P(\delta x,\delta x) \hat{\tau}_P(H_P(\delta x) - T).
\end{align*}

Alors $\sum_{\mathfrak{o}} K^{\bomega}_{P,\mathfrak{o}} = K^\bomega_P$ et $\sum_\mathfrak{o} k^{T,\bomega}_{\mathfrak{o}} = k^{T,\bomega}$. Comme remarqué plus haut, on a $K^\bomega_{P,\mathfrak{o}}(x,y) = \bomega(y) K_{P,\mathfrak{o}}(x,y)$ et $k^{T,\bomega}_{\mathfrak{o}}(x)=\bomega(x)k^T_{\mathfrak{o}}(x)$. Puisque $\bomega$ est unitaire, le résultat suivant découle immédiatement du cas usuel $\bomega=1$.

\begin{theorem}[cf. {\cite[7.1]{Ar78}}]
  Soit $T \in \mathfrak{a}_0^+$ suffisamment régulier, alors
  $$ \sum_{\mathfrak{o} \in \mathcal{O}} \int_{G(F) \backslash G(\A)^1} k^{T,\bomega}_{\mathfrak{o}}(x) \dd x $$
  converge absolument.
\end{theorem}

Soit $f \in C_c^\infty(\tilde{G}^1)$ quelconque et notons $k^{T,\bomega}(x,f)$ la fonction ainsi associée. Il est donc loisible de définir la distribution
$$ f \mapsto  J^{T,\bomega}_\mathfrak{o}(f) := \int_{G(F) \backslash G(\A)^1} k^{T,\bomega}_\mathfrak{o}(x,f) \dd x .$$

On indiquera le groupe en question en exposant les notations, eg. $J^{G,T,\bomega}_\mathfrak{o}$. Si $\mathfrak{o} \ni 1$ (on l'appelle la classe unipotente dans $\mathcal{O}$), nous notons les objets associés par $K^\bomega_{P,\text{unip}}$, $k^{T,\bomega}_{\text{unip}}$ et $J^{T,\bomega}_{\text{unip}}$.

Si $M \in \mathcal{L}(M_0)$, $\mathfrak{o} \in \mathcal{O}^G$ et $f \in C_c^\infty(M(\A)^1)$, posons
$$ J^{M,T,\bomega}_{\mathfrak{o}}(f) = \sum_{\substack{\mathfrak{o}' \in \mathcal{O}^M \\ \mathfrak{o}' \mapsto \mathfrak{o}}} J^{M,T,\bomega}_{\mathfrak{o}'}(f). $$

\subsection{Comportement des distributions}
\paragraph{Modification de troncature}
Le fait suivant sera utilisé à plusieurs reprises.
\begin{proposition}\label{prop:omega-SC-trivial}
  Si $G$ est simplement connexe, alors $\bomega = 1$.
\end{proposition}
\begin{proof}
  Cela résulte immédiatement de la paramétrisation de tels caractères par Langlands, cf.  \cite[pp.122-123]{Lan89}.
\end{proof}
\begin{corollary}\label{prop:omega-trivial-unip}
  Le caractère $\bomega$ est trivial sur $G_\mathrm{unip}(\A)$.
\end{corollary}
\begin{proof}
  Notons $\pi: G_\text{SC} \to G$ le revêtement simplement connexe de $G_\text{der}$, alors $\pi$ induit un isomorphisme $(G_\text{SC})_{\text{unip}} \rightiso G_\text{unip}$ de $F$-schémas, d'où un homéomorphisme pour leurs points adéliques.
\end{proof}

\begin{lemma}\label{prop:omega-aGaM}
  Soit $M$ un sous-groupe de Lévi de $G$. Alors $\bomega$ est trivial sur $A_{M,\infty} \cap G(\A)^1$.
\end{lemma}
\begin{proof}
  L'application de Harish-Chandra adélique fournit un isomorphisme de groupes topologiques
  \begin{equation}\label{eqn:aGM}
    H_M: A_{M,\infty} \cap G(\A)^1 \rightiso \mathfrak{a}^G_M.
  \end{equation}
  Notons toujours $\pi: G_\text{SC} \to G$ le revêtement simplement connexe de $G_\text{der}$ et $M_\text{sc} \to M$ sa fibre au-dessus de $M$. On obtient l'analogue de \eqref{eqn:aGM} pour $G_\text{SC}$ et $M_\text{sc}$. Vu la description de $\mathfrak{a}^G_M$ en termes de coracines, on voit que $\pi$ induit $\mathfrak{a}^{G_\text{SC}}_{M_\text{sc}} \simeq \mathfrak{a}^G_M$; l'identification est compatible avec $H_M$ et $H_{M_\text{sc}}$. Ainsi, on se ramène à prouver la même assertion pour $G_\text{SC}$, $M_\text{sc}$ et $\bomega_\text{SC} := \bomega \circ \pi$. Or $\bomega_\text{SC}$ est encore un caractère automorphe, on conclut à l'aide de \ref{prop:omega-SC-trivial}.
\end{proof}

La notion suivante facilitera l'étude du comportement des distributions $J^{T,\bomega}_\mathfrak{o}$.
\begin{definition}
  Pour $P_0$ fixé, une modification de troncature est une famille des fonctions continues
  $$ \mathcal{Y} := \{Y_Q: Q(\A) \cap K \backslash K \to \mathfrak{a}_Q, \quad Q \in \mathcal{F}(M_0), Q \supset P_0 \} $$
  telle que le diagramme suivant commute pour tout $Q \supset P \supset P_0$:
  $$\xymatrix{
  P(\A) \cap K \backslash K \ar[rr]^{Y_P} \ar@{->>}[d] & & \mathfrak{a}_P \ar@{->>}[d] \\
  Q(\A) \cap K \backslash K \ar[rr]_{Y_Q} & & \mathfrak{a}_Q
  }$$
  À une telle famille sont associées des fonctions
  $$ u_Q(\lambda, k;\mathcal{Y}) := e^{\angles{\lambda,Y_Q(k)}}, \quad \lambda \in i\mathfrak{a}_Q^*, \ $$
  lisses en $\lambda$, dont $u'_Q(k;\mathcal{Y}) := u'_Q(0,k;\mathcal{Y})$ est la fonction associée via \eqref{eqn:c'_P}.

  Soient $\mathcal{Y}$ une modification de troncature, $f \in C_c^\infty(G(\A)^1)$ et $Q \supset P_0$, définissons une variante pondérée de la descente parabolique comme suit
  $$ f_{Q,\mathcal{Y}}^\bomega(m) := \delta_Q(m)^{\frac{1}{2}} \int_K \int_{U_Q(\A)} \bomega(k) f(k^{-1}muk) u'_Q(k;\mathcal{Y}) \dd u \dd k, \quad m \in M_Q(\A)^1 . $$
  On vérifie que ceci définit une application continue $C_c^\infty(G(\A)^1) \to C_c^\infty(M_Q(\A)^1)$.
\end{definition}

\begin{theorem}[cf. {\cite[(2.4)]{Ar81}}]\label{prop:modification-troncature}
  Soit $\mathfrak{o} \in \mathcal{O}$. Soit $\mathcal{Y}$ une modification de troncature. Posons
  \begin{align*}
    k^{T,\bomega}_{\mathfrak{o}}(x;\mathcal{Y}) & := \sum_{P \supset P_0} (-1)^{\dim A_P/A_G} \sum_{\delta \in P(F) \backslash G(F)} K^\bomega_{P,\mathfrak{o}}(\delta x,\delta x) \hat{\tau}_P(H_P(\delta x) - T - Y_P(k_P(\delta x))).
  \end{align*}
  alors
  $$ J_{\mathfrak{o}}^{T,\bomega}(f;\mathcal{Y}) = \int_{G(F) \backslash G(\A)^1} k^{T,\bomega}_{\mathfrak{o}}(x;\mathcal{Y}) \dd x $$
  est convergent pour $T \in \mathfrak{a}_0^+$ suffisamment régulier. De plus, on a
  $$ J_{\mathfrak{o}}^{T,\bomega}(f;\mathcal{Y}) = \sum_{Q \supset P_0} J_{\mathfrak{o}}^{M_Q, T,\bomega}(f_{Q;\mathcal{Y}}^\bomega). $$
\end{theorem}
On se débarrassera de la condition sur $T$ dans \ref{prop:J^T-polynomial}.
\begin{proof}
  Pour tout sous-groupe parabolique $P$, on a
  \begin{multline*}
    \hat{\tau}_P(H_P(\delta x)-T-Y_P(k_P(\delta x))) = \\
    \sum_{Q \supset P} (-1)^{\dim A_Q/A_G} \hat{\tau}^Q_P(H_P(\delta x)-T) \Gamma^G_Q(H_Q(\delta x)-T, Y_Q(k_Q(\delta x))).
  \end{multline*}
  Via le changement de variables
  \begin{align*}
    (G(F) \backslash G(\A)^1) \times (P(F) \backslash G(F)) & \rightiso (Q(F) \backslash G(\A)^1) \times (P(F) \backslash Q(F))
  \end{align*}
  la formule définissant $J_\mathfrak{o}^{T,\bomega}(f;\mathcal{Y})$ s'écrit
  \begin{multline*}
    J_\mathfrak{o}^{T,\bomega}(f;\mathcal{Y}) = \sum_{Q \supset P_0} \; \int_{Q(F) \backslash G(\A)^1} \sum_{P: Q \supset P \supset P_0} (-1)^{\dim A_P/A_Q} \sum_{\delta \in P(F) \backslash Q(F)} \\
    K_{P,\mathfrak{o}}^\bomega(\delta x, \delta x) \hat{\tau}^Q_P(H_P(\delta x)-T) \Gamma^G_Q(H_Q(\delta x)-T, Y_Q(k_Q(\delta x))) \dd x. 
  \end{multline*}

  Écrivons
  \begin{align*}
    Q(F) \backslash G(\A)^1 & = (U_Q(F) \backslash U_Q(\A)) \times (A_{M_Q,\infty} \cap G(\A)^1) \times (M_Q(F) \backslash M_Q(\A)^1) \times K , \\
    x & = uamk , \\
    \dd x &= \delta_Q(a)^{-1} \dd u \dd a \dd m \dd k , \\
    K_{P,\mathfrak{o}}^\bomega(\delta x, \delta x) & = \bomega(m)\bomega(k) K_{P,\mathfrak{o}}(\delta uamk, \delta uamk),
  \end{align*}
  où on a utilisé \ref{prop:omega-aGaM} qui assure $\bomega(a)=1$. On vérifie de plus
  \begin{gather*}
    Y_Q(k_Q(\delta uamk)) = Y_Q(k), \\
    K_{P,\mathfrak{o}}(\delta uamk, \delta uamk) = \delta_Q(a) K_{P,\mathfrak{o}}(\delta mk, \delta mk), \\
    H_Q(\delta uamk) = H_Q(a), \\
    H_P(\delta uamk) \in H_P(\delta m) + \mathfrak{a}_Q.
  \end{gather*}
  Rappelons que $A_{M_Q,\infty} \cap G(\A)^1$ s'identifie à $\mathfrak{a}^G_Q$ via $H_{M_Q}$. Les équations ci-dessus entraînent que
  \begin{multline*}
    J_\mathfrak{o}^{T,\bomega}(f;\mathcal{Y}) = \sum_{Q \supset P_0} \; \int_{M_Q(F) \backslash M_Q(\A)^1} \bomega(m) \sum_{P: Q \supset P \supset P_0} (-1)^{\dim A_P/A_Q} \int_K \bomega(k) \\
    \sum_{\delta \in (P\cap M_Q)(F) \backslash M_Q(F)} \left(\int_{\mathfrak{a}^G_Q} \Gamma^G_Q(H, Y_P(k)) \dd H \right) K_{P,\mathfrak{o}}(\delta mk, \delta mk) \hat{\tau}^Q_P(H_P(\delta m)-T) \dd k \dd m.
  \end{multline*}
  Grâce à \ref{prop:c'_P-polynome}, l'intégrale sur $\mathfrak{a}^G_Q$ vaut $u'_Q(k;\mathcal{Y})$. L'application $P \mapsto P \cap M_Q$ induit une bijection entre $\{P: Q \supset P \supset P_0 \}$ et l'ensemble de sous-groupes paraboliques standards de $M_Q$. On vérifie que, pour tous $m_1, m_2 \in M_Q(\A)^1$ on a
  $$ \int_K \bomega(k) K_{P,\mathfrak{o}}(m_1 k, m_2 k) u'_Q(k;\mathcal{Y}) \dd k = \sum_{\gamma \in M_P(F) \cap \mathfrak{o}} \; \int_{(U_P \cap M_Q)(\A)} f_{Q;\mathcal{Y}}^\bomega(m_1^{-1} \gamma u m_2) \dd u, $$
  cf. \cite[p.17]{Ar81}. En l'appliquant à $m_1=m_2=\delta m$, on en déduit l'assertion. 
\end{proof}

\paragraph{Dépendance de $T$}
Pour tout $Q \in \mathcal{F}(M_0)$, posons\index[iFT1]{$f_Q^\bomega$}
\begin{gather}\label{eqn:descente-parabolique}
  f_Q^\bomega(m) := \delta_Q(m)^{\frac{1}{2}} \int_K \int_{U_Q(\A)} \bomega(k) f(k^{-1}muk) \dd u \dd k, \quad m \in M_Q(\A)^1;
\end{gather}
ceci fournit une application continue $C_c^\infty(G(\A)^1) \to C_c^\infty(M_Q(\A)^1)$.

\begin{corollary}[cf. {\cite[(2.4)]{Ar81}}]\label{prop:dependance-T}
  Soient $\mathfrak{o} \in \mathcal{O}$, $f \in C_c^\infty(G(\A)^1)$. Soient $T, T_1 \in \mathfrak{a}_0^+$ suffisamment réguliers, alors
  $$ J_{\mathfrak{o}}^{T_1,\bomega}(f) = \sum_{Q \supset P_0} J_\mathfrak{o}^{M_Q, T, \bomega}(f_Q^\bomega) \cdot \int_{\mathfrak{a}^G_Q} \Gamma^G_Q(H, T_1-T) \dd H . $$
\end{corollary}
\begin{proof}
  Prenons $Y_P(k) := T_1-T \in \mathfrak{a}_P$ pour tout $P \supset P_0$ et tout $k \in K$, cela définit une modification de troncature. D'après \ref{prop:c'_P-polynome}, on a
  $$ f_{Q,\mathcal{Y}}^\bomega = f_Q^\bomega \int_{\mathfrak{a}^G_Q} \Gamma^G_Q(H, T_1-T) \dd H . $$
  On a aussi $J^{T,\bomega}_\mathfrak{o}(f;\mathcal{Y}) = J^{T_1,\bomega}_{\mathfrak{o}}(f)$. L'assertion résulte immédiatement de \ref{prop:modification-troncature}.
\end{proof}

\begin{corollary}\label{prop:J^T-polynomial}
  La distribution $f \mapsto J_\mathfrak{o}^{T,\bomega}(f)$, définie au début pour $T \in \mathfrak{a}_0^+$ suffisamment régulier, est polynomiale en $T$ de degré $\leq \dim\mathfrak{a}_0^G$. Par conséquent, la distribution est bien définie comme un polynôme en $T \in \mathfrak{a}_0$.

  La formule dans \ref{prop:modification-troncature} reste valable pour tout $T \in \mathfrak{a}_0$.
\end{corollary}
\begin{proof}
  La première assertion découle de \ref{prop:c'_P-polynome}. La deuxième en résulte en notant que les deux côtés de \ref{prop:modification-troncature} sont tous polynomiaux en $T$.
\end{proof}

\paragraph{Non-invariance}
Fixons $\mathfrak{o} \in \mathcal{O}^G$. Soient $f \in C_c^\infty(G(\A))$, $y \in G(\A)$, définissons
$$ f^y(x) = f(y x y^{-1}), \quad x \in G(\A)^1 . $$

Définissons une modification de troncature $\mathcal{Y}_y$ en posant $Y_P(k) = -H_P(ky)$. Posons
\begin{align}
  \label{eqn:u'_Q(k,y)} u'_Q(k,y) & := u'_Q(k;\mathcal{Y}_y), \\
  \label{eqn:descente-parabolique-y} f_{Q,y}^\bomega & := f_{Q,\mathcal{Y}_y}^\bomega.
\end{align}

\begin{theorem}\label{prop:non-variance-Jomega}
  Avec les notations précédentes, on a
  \begin{align*}
    J_{\mathfrak{o}}^{T,\bomega}(f^y) = \bomega(y) \sum_{Q \supset P_0} J_{\mathfrak{o}}^{M_Q, T, \bomega}(f_{Q,y}^\bomega).
  \end{align*}
\end{theorem}
\begin{proof}
  Pour tout sous-groupe parabolique standard $P$, notons $K^\bomega_{P,\mathfrak{o},f^y}$ le noyau associé à $f^y$ au lieu de $f$. Pour tout $\delta \in G(F)$, on vérifie que
  $$ K^\bomega_{P,\mathfrak{o},f^y}(\delta x, \delta x) = \bomega(y) K^\bomega_{P, \mathfrak{o}}(\delta xy^{-1}, \delta xy^{-1}). $$
  D'où
  \begin{align*}
    J_{\mathfrak{o}}^{T,\bomega}(f^y) &= \bomega(y) \int_{G(F) \backslash G(\A)^1} \sum_{P \supset P_0} (-1)^{\dim(A_P/A_G)} \sum_{\delta \in P(F) \backslash G(F)} K^\bomega_{P,\mathfrak{o}}(\delta xy^{-1}, \delta xy^{-1}) \hat{\tau}_P(H_P(\delta x)-T) \dd x \\
    & = \bomega(y) \int_{G(F) \backslash G(\A)^1} \sum_{P \supset P_0} (-1)^{\dim(A_P/A_G)} \sum_{\delta \in P(F) \backslash G(F)} K^\bomega_{P,\mathfrak{o}}(\delta x, \delta x) \hat{\tau}_P(H_P(\delta xy)-T) \dd x.
  \end{align*}

  Comme $H_P(\delta xy) = H_P(\delta x) + H_P(k_P(\delta x)y)$, on voit que $J_\mathfrak{o}^{T,\bomega}(f^y) = \bomega(y) J_{\mathfrak{o}}^{T,\bomega}(f; \mathcal{Y}_y)$. Cela permet de conclure d'après \ref{prop:modification-troncature}.
\end{proof}

\begin{definition}
  Soit $w \in W_0^G$, prenons des représentants $\hat{w} \in G(F)$ et $\tilde{w} \in K$, alors $H_{P_0}(\hat{w}) = H_{M_0}(\hat{w}\tilde{w}^{-1})$; comme $H_{M_0}$ est trivial sur $M_0(F)$, cela ne dépend que de $w$, $M_0$ et $K$. Notons-le $H_{P_0}(w)$ bien qu'il ne dépend pas de $P_0$.

  Dans \cite{Ar81}, Arthur définit un unique point $T_0 \in \mathfrak{a}_0^G$ tel que
  \begin{gather}\label{eqn:T_0}
    H_{P_0}(w^{-1}) = T_0 - w^{-1} T_0, \quad w \in W_0^G.
  \end{gather}
  On l'appelle le paramètre de troncature canonique pour $(G,M_0,K)$. Définissons $J_{\mathfrak{o}}^\bomega := J_{\mathfrak{o}}^{T_0, \bomega}$. Nous allons démontrer que $J_{\mathfrak{o}}^\bomega$ ne dépend pas du choix de $P_0$.
\end{definition}

Notons $K_\text{sc}$ l'image réciproque de $K$ par $\pi: G_\text{SC}(\A) \to G(\A)$.

\begin{proposition}\label{prop:transport-w}
  Soient $L, L' \in \mathcal{L}(M_0)$ et $w \in W_0^G$ avec un représentant $\hat{w} \in G(F)$ tel que $L'=w^{-1}Lw$. Soit $\tilde{w} \in \pi(K_\text{sc})$ un autre représentant, i.e. $\tilde{w}\hat{w}^{-1} \in M_0(\A)$. Pour $f \in C_c^\infty(L(\A)^1)$, posons
  $$ f'(x') := f(\tilde{w} x' \tilde{w}^{-1}), \quad x' \in L'(\A)^1. $$
  Alors
  $$ J^{L',\bomega}_{\mathfrak{o}}(f') = J^{L,\bomega}_{\mathfrak{o}}(f) $$
  où $J^{L,\bomega}_{\mathfrak{o}}$ (resp. $J^{L',\bomega}_{\mathfrak{o}}$) est défini par rapport à $K_L := K \cap L(\A)$ (resp. $K_{L'} := K \cap L'(\A)$) et $R_0 := P_0 \cap L$ (resp. $R'_0 := w^{-1}(P_0 \cap L)w$).
\end{proposition}
\begin{proof}
  Le paramètre de troncature canonique pour $L$ (resp. $L'$) s'obtient en projetant $T_0$ via $\mathfrak{a}^G_0 \to \mathfrak{a}^L_0$ (resp. $\mathfrak{a}^G_0 \to \mathfrak{a}^{L'}_0$). Posons
  \begin{align*}
    f^\circ(x) & := f(\tilde{w}\hat{w}^{-1} x \hat{w}\tilde{w}^{-1}), \quad x \in L(\A)^1, \\
    K^\circ & := \hat{w} K_{L'} \hat{w}^{-1} \subset L(\A).
  \end{align*}
  Prenons $T \in (\mathfrak{a}^{L'}_{R'_0})^+$ suffisamment régulier. Par le transport de structure $x \mapsto \hat{w} x \hat{w}^{-1}$, on a
  $$ J^{L',T,\bomega}_\mathfrak{o}(f';K_{L'}) = J^{L,wT,\bomega}_\mathfrak{o}(f^\circ; K^\circ). $$

  Soit $R \supset R_0$, notons $K^\bomega_{R,\mathfrak{o}}$ le noyau associé à $f$ et $K$. En utilisant le fait que $\bomega(\hat{w})=\bomega(\tilde{w})=1$, dont la dernière égalité résulte de \ref{prop:omega-trivial-unip}, l'argument pour \ref{prop:non-variance-Jomega} montre que
  \begin{multline*}
    J^{L,wT,\bomega}_\mathfrak{o}(f^\circ; K^\circ) = \int_{L(F) \backslash L(\A)^1} \sum_{R \supset R_0} (-1)^{\dim(A_R/A_L)} \sum_{\delta \in R(F) \backslash L(F)} \\
    K^\bomega_{R,\mathfrak{o}}(\delta x, \delta x) \hat{\tau}_R(H_R(\delta x \tilde{w}\hat{w}^{-1}; K^\circ)-wT) \dd x,
  \end{multline*}
  où $H_R(\cdot; K^\circ)$ désigne l'application de Harish-Chandra définie par rapport à $K^\circ$.

  D'autre part, considérons $J^{L,wT-wT_0+T_0,\bomega}_\mathfrak{o}(f; K_L)$; il s'exprime de la même manière sauf que le terme $\hat{\tau}_R(\cdots)$ est remplacé par
  $$ \hat{\tau}_R(H_R(\delta x; K)+wT_0-T_0-wT). $$

  Soit $\delta x = umk$ une décomposition d'Iwasawa où $u \in U_R(\A)$, $m \in M_R(\A)$, $k \in K \cap L(\A)$, alors $H_R(\delta x; K)=H_{M_R}(m)$. On a
  \begin{align*}
    \delta x \tilde{w}\hat{w}^{-1} = um \underbrace{\tilde{w}\hat{w}^{-1}}_{\in M_0(\A)} \cdot \underbrace{\hat{w}\tilde{w}^{-1} k\tilde{w}\hat{w}^{-1}}_{\in K^\circ},
  \end{align*}
  donc $H_R(\delta x \tilde{w}\hat{w}^{-1}; K^\circ)=H_{M_R}(m)+H_{M_R}(\tilde{w}\hat{w}^{-1})$. Montrons que $H_{M_R}(\tilde{w}\hat{w}^{-1})=wT_0-T_0$. Il suffit de considérer le cas $M_R=M_0$. Posons $m_0 := \hat{w}\tilde{w}^{-1} \in M_0(\A)$, i.e. $\hat{w} = m_0 \tilde{w}$, alors 
  $$ H_{P_0}(\hat{w})=H_{M_0}(m_0)=-H_{M_0}(\tilde{w}\hat{w}^{-1});$$
  or $H_{P_0}(\hat{w})=T_0-wT_0$ d'après la définition de $T_0$.

  On en déduit que
  $$ J^{L',T,\bomega}_\mathfrak{o}(f';K_{L'}) = J^{L,wT-wT_0+T_0,\bomega}_\mathfrak{o}(f;K_L).$$
  Les deux côtés sont polynomiaux en $T$. On conclut en prenant $T=T_0$.
\end{proof}

\begin{corollary}[cf. {\cite[\S 2]{Ar81}}]\label{prop:independance-P_0}
  Les distributions $J_{\mathfrak{o}}^\bomega$ ne dépendent pas du choix de $P_0 \in \mathcal{P}(M_0)$.
\end{corollary}
\begin{proof}
  Soit $P'_0 \in \mathcal{P}(M_0)$. Prenons $w \in W_0^G$ tel que $P'_0 = w^{-1}P_0 w$ et un représentant $\tilde{w} \in \pi(K_\text{sc})$. Vu \ref{prop:transport-w}, il suffit de montrer que
  $$ J_{\mathfrak{o}}^\bomega(f) = J_{\mathfrak{o}}^\bomega(f^{\tilde{w}^{-1}}). $$
  Soit $Q \supset P_0$. Puisque $\tilde{w}^{-1} \in K$, on a
  $$ u'_Q(\cdot,\tilde{w}^{-1}) = \int_{\mathfrak{a}^G_Q} \Gamma^G_Q(H,0)\dd H = \begin{cases} 1,& \text{si } Q=G \\ 0,& \text{sinon}. \end{cases}. $$
  En rappelant que $\bomega(\tilde{w})=1$, on conclut par \ref{prop:non-variance-Jomega}.
\end{proof}

\begin{corollary}\label{prop:non-invariance-Jomega-Weyl}
  Soit $y \in G(\A)^1$, on a
  $$ J_{\mathfrak{o}}^\bomega(f^y) = \bomega(y) \sum_{Q \in \mathcal{F}(M_0)} |W_0^{M_Q}| |W_0^G|^{-1} J_{\mathfrak{o}}^{M_Q,\bomega}(f_{Q,y}^\bomega). $$
\end{corollary}
\begin{proof}
  Nous allons le déduire de \ref{prop:non-variance-Jomega}. Supposons que $Q' \in \mathcal{F}(M_0)$, alors il existe un unique $Q \supset P_0$ et un $w \in W_0^G$ tel que $Q' = w^{-1}Qw$. De plus, l'application $Q' \mapsto Q$ est à fibres de $|W_0^{M_Q}|^{-1}|W_0^G|$ éléments. Prenons un représentant $\tilde{w} \in \pi(K_\text{sc})$ de $w$. On vérifie que
  $$ \forall k \in K, \quad u'_{Q'}(\tilde{w}^{-1}k,y) = u'_Q(k,y) $$
  (cf. \cite[p.21]{Ar81}), donc $f^\bomega_{Q',y}(\tilde{w}^{-1}x\tilde{w})=f^\bomega_{Q,y}(x)$ pour tout $x \in M_Q(\A)^1$. D'après \ref{prop:transport-w}, $J_\mathfrak{o}^{M_{Q'},\bomega}(f^\bomega_{Q',y}) = J_\mathfrak{o}^{M_Q,\bomega}(f^\bomega_{Q,y})$. Alors \ref{prop:non-variance-Jomega} permet de conclure.
\end{proof}

\paragraph{Dépendance de $K$}
Conservons les conventions du paragraphe précédent. Fixons un autre sous-groupe compact maximal $K_1$ de $G(\A)$ en bonne position relativement de $M_0$ et notons $T_0$ (resp. $T_1$) le paramètre de troncature canonique pour $K$ (resp. $K_1$).

Fixons $P_0 \in \mathcal{P}(M_0)$. Considérons la famille des fonctions continues $K \to \mathfrak{a}_P$ indexée par $Q \in \mathcal{F}(M_0)$
\begin{align*}
  Y_Q(k) & := -H_Q(k; K_1)+T_1-T_0.
\end{align*}
On vérifie que $\mathcal{Y} := (Y_Q)_{Q \supset P_0}$ est une modification de troncature. On définit ainsi
\begin{align}
  u'_Q(k; K_1|K) & := u'_Q(k; \mathcal{Y}), \\
  f_{Q; K_1|K}^\bomega & := f_{Q;\mathcal{Y}}^\bomega.
\end{align}

\begin{lemma}\label{prop:dependance-K-pre}
  Avec les notations précédentes, on a
  \begin{align*}
    J^\bomega_{\mathfrak{o}}(f;K_1) & = \sum_{Q \supset P_0} J_{\mathfrak{o}}^{M_Q,\bomega}(f^\bomega_{Q;K_1|K}).
  \end{align*}
\end{lemma}
\begin{proof}
  Supposons $T \in \mathfrak{a}_0^+$ suffisamment régulier, on a
  \begin{multline*}
    J^{T+T_1}(f;K_1) = \int_{G(F) \backslash G(\A)^1} \sum_{P \supset P_0} (-1)^{\dim (A_P/A_G)} \sum_{\delta \in P(F) \backslash G(F)} \\
    K^\bomega_P(\delta x, \delta x) \hat{\tau}_P(H_P(\delta x; K_1) - T - T_1) \dd x ,
  \end{multline*}
  tandis que $J^{T+T_0}(f)$ est défini de la même façon sauf que le terme $\hat{\tau}_P(\cdots)$ est remplacé par $\hat{\tau}_P(H_P(\delta x) - T - T_0)$. Pour conclure, il suffit de noter
  $$ H_P(\delta x; K_1)-T-T_1 = H_P(\delta x)-T-T_0 + H_P(k_P(\delta x); K_1) + T_0-T_1  $$
  puis appliquer \ref{prop:modification-troncature} et évaluer des polynômes en $T=0$.
\end{proof}

\begin{lemma}
  Soit $w \in W_0^G$. Si $\tilde{w}$ (resp. $\tilde{w}_1$) est un représentant de $w$ dans $K$ (resp. $K_1$), alors $H_{M_0}(\tilde{w}^{-1}\tilde{w}_1)=(w^{-1}-1)(T_0-T_1)$.
\end{lemma}
\begin{proof}
  Prenons un représentant rationnel $\hat{w} \in G(F)$ de $w$. La définition des paramètres $T_0, T_1$ affirme que
  \begin{align*}
    (1-w)T_0 & = H_{P_0}(\hat{w}) = H_{M_0}(\hat{w}\tilde{w}^{-1}), \\
    (1-w)T_1 & = H_{P_0}(\hat{w};K_1) = H_{M_0}(\hat{w}\tilde{w}_1^{-1}).
  \end{align*}
  En prenant la différence, on obtient $(1-w)(T_0-T_1)=H_{M_0}(\tilde{w}_1 \tilde{w}^{-1})$, ce qui est égal à $wH_{M_0}(\tilde{w}^{-1}\tilde{w}_1)$.
\end{proof}

\begin{theorem}\label{prop:dependance-K}
  On a
  $$ J^\bomega_{\mathfrak{o}}(f;K_1) = \sum_{Q \in \mathcal{F}(M_0)} |W_0^{M_Q}| |W_0^G|^{-1} J_{\mathfrak{o}}^{M_Q,\bomega}(f^\bomega_{Q;K_1|K}). $$
\end{theorem}
\begin{proof}
  D'abord, supposons que $w \in W_0^G$, $Q' \in \mathcal{F}(M_0)$ sont tels que $Q' = w^{-1}Qw$ avec $Q \supset P_0$. Prenons un représentant $\tilde{w} \in \pi(K_\text{sc})$. D'après \ref{prop:transport-w}, on a
  \begin{gather}\label{eqn:dependance-K-0}
    J^{M_{Q'},\bomega}_\mathfrak{o}(f^\bomega_{Q';K_1|K}) = J^{M_Q,\bomega}_\mathfrak{o}((f^\bomega_{Q';K_1|K})^{\tilde{w}^{-1}}).
  \end{gather}
  Pour tout $m \in M_Q(\A)^1$, on a
  \begin{align}
    \label{eqn:dependance-K-1} f^\bomega_{Q';K_1|K}(\tilde{w}^{-1}m\tilde{w}) & = \delta_{Q'}(\tilde{w}^{-1}m\tilde{w})^{\frac{1}{2}} \int_K \int_{U_{Q'}(\A)} \bomega(k) f(k^{-1}\tilde{w}^{-1}m\tilde{w}u'k) u'_{Q'}(k; K_1|K) \dd u' \dd k \\
    \label{eqn:dependance-K-2} & = \delta_Q(m)^{\frac{1}{2}} \int_K \int_{U_Q(\A)} \bomega(k) f(k^{-1}muk) u'_{Q'}(\tilde{w}^{-1}k; K_1|K).
  \end{align}
  Écrivons $k=qk_1$ avec $q \in Q(\A)$ et $k_1 \in K_1$, alors $H_Q(k;K_1)=H_Q(q)$. D'autre part,
  $$ \tilde{w}^{-1}k = \underbrace{\tilde{w}^{-1}\tilde{w}}_{\in Q'(\A)} \cdot \underbrace{\tilde{w}^{-1}\tilde{w}_1}_{\in M_0(\A)} \cdot \underbrace{\tilde{w}_1^{-1}k_1}_{\in K_1} $$
  entraîne que
  \begin{align*}
    -wH_{Q'}(\tilde{w}^{-1}k; K_1) &= -w( w^{-1}H_Q(k;K_1)+H_{M_{Q}'}(\tilde{w}^{-1}\tilde{w}_1)) \\
    & = -H_Q(k;K_1) - (1-w)(T_0-T_1),\\
    w(-H_{Q'}(\tilde{w}^{-1}k; K_1)+T_1-T_0) & = -H_Q(k;K_1)+T_1-T_0.
  \end{align*}
  d'après le lemme précédent. D'où $u'_{Q'}(\tilde{w}^{-1}k; K_1|K) = u_Q(k; K_1|K)$, donc
  \begin{gather}\label{eqn:dependance-K-3}
    (f^\bomega_{Q';K_1|K})^{\tilde{w}^{-1}} = f^\bomega_{Q;K_1|K}.
  \end{gather}

  Maintenant on fait varier $Q$ et $w$. Les équations \eqref{eqn:dependance-K-1}-\eqref{eqn:dependance-K-3} entraînent que
  \begin{gather*}
    \sum_{Q' \in \mathcal{F}(M_0)} |W_0^{M_{Q'}}| |W_0^G|^{-1} J_{\mathfrak{o}}^{M_{Q'},\bomega}(f^\bomega_{Q';K_1|K}) =
    \sum_{Q \supset P_0} J^{M_Q, \bomega}_\mathfrak{o}(f^\bomega_{Q;K_1|K}).
  \end{gather*}

  Vu \ref{prop:dependance-K-pre}, cela achève la démonstration.
\end{proof}

\subsection{Intégrales orbitales pondérées avec caractère}\label{sec:int-orb-omega}
Fixons un ensemble fini $S$ de places de $F$. Décomposons les objets en question comme $K = K_S \times K^S$, $G(\A) = G(F_S) \times G(F^S)$, $M(\A) = M(F_S) \times M(F^S)$, etc. Notons la restriction de $\bomega$ sur $M(F_S)$, où $M$ est un sous-groupe de Lévi quelconque de $G$, par le même symbole $\bomega$. Fixons des mesures sur $G(F_S)$ et sur les $M(F_S)$ selon \S\ref{sec:mesure}.

\begin{remark}
  Bien que ces objets-là sont supposés d'origine globale, ici il ne s'agit que d'une étude locale. Par exemple, la seule propriété de $\bomega$ qui interviendra est qu'il est un caractère unitaire de $G(F_S)$; il existe aussi des versions locales de \ref{prop:omega-SC-trivial} et \ref{prop:omega-trivial-unip}.
\end{remark}

\paragraph{Intégrales orbitales avec caractère}
Soient $M$ un sous-groupe de Lévi de $G$, $f \in  C_c^\infty(M(F_S))$ et $y \in M(F_S)$. Posons toujours\index[iFT1]{$f^y$}
$$ f^y(x) = f(yxy^{-1}), \quad x \in M(F_S). $$

Soit $D$ une distribution sur $M(F_S)$, posons\index[iFT1]{${}^y D$}
$$ {}^y D: f \mapsto \angles{D, f^y}, \quad f \in C_c^\infty(M(F_S)) $$
Cela définit une action à gauche (resp. à droite) de $M(F_S)$ sur l'espace des distributions (resp. des fonctions) sur $M(F_S)$.

\begin{definition}\index[iFT1]{$\bomega$-équivariante}
  Une fonction $f$ (resp. distribution $D$) est dite $\bomega$-équivariante si $f^y = \bomega(y)f$ (resp. ${}^y D = \bomega(y)D$) pour tout $y$.
\end{definition}
Par exemple, une fonction $f$ localement intégrable est $\bomega$-équivariante si et seulement la distribution $\phi \mapsto \int_{M(F_S)} f(x)\phi(x) \dd x $ l'est. Nous nous intéressons aux distributions $\bomega$-équivariantes.

\paragraph{Conventions sur la mesure}
On considère les paires $(\mathcal{O},\mu)$, où
\begin{itemize}
  \item $\mathcal{O}$ est une classe de conjugaison dans $M(F_S)$,
  \item $\mu$ est une mesure de Radon complexe non triviale sur $\mathcal{O}$ qui est $\bomega$-équivariante; autrement dit $\mu(y^{-1}xy) =\bomega(y)\mu(x)$ pour tout $x \in \mathcal{O}$ et tout $y \in M(F_S)$.
\end{itemize}
Le groupe $M(F_S)$ opère sur ces paires par conjugaison; la conjugaison ne change pas $\mathcal{O}$ mais elle transporte $\mu$. On écrit $(\mathcal{O}_1, \mu_1) \sim (\mathcal{O}_2, \mu_2)$ si $\mathcal{O}_1=\mathcal{O}_2$. Notons
\begin{align*}
  \dot{\Gamma}(M(F_S))^\bomega & := \{(\mathcal{O},\mu) \},\\
  \Gamma(M(F_S))^\bomega & := \{(\mathcal{O}, \mu)\}/\sim .
\end{align*}
Alors $\dot{\Gamma}(M(F_S))^\bomega \to \Gamma(M(F_S))^\bomega$ est un $\C^\times$-torseur, ce qui permet de définir $\dot{\gamma}/\dot{\eta} \in \C^\times$ si $\dot{\gamma}$ et $\dot{\eta}$ ont la même classe de conjugaison sous-jacente.\index[iFT1]{$\dot{\Gamma}(M(F_S))^\bomega, \Gamma(M(F_S))^\bomega$}

\begin{definition}\index[iFT1]{$\bomega$-bon}
  Une classe de conjugaison $\mathcal{O}$ dans $M(F_S)$ est dite $\bomega$-bonne si elle admet une mesure de Radon $\bomega$-équivariante comme ci-dessus. Autrement dit, $\mathcal{O}$ est bonne si pour tout $\gamma \in \mathcal{O}$, on a $\bomega|_{M^\gamma(F_S)}=1$. On dit que $\gamma \in M(F_S)$ est $\bomega$-bon si sa classe de conjugaison l'est.
\end{definition}

Nous utilisons les symboles pointés pour désigner un élément dans $\dot{\Gamma}(M(F_S))^\bomega$, eg. $\dot{\gamma}$; la classe de conjugaison sous-jacente est notée $\Supp(\dot{\gamma})$.

Une paire $\dot{\gamma} = (\mathcal{O},\mu)$ donne naissance à l'intégrale orbitale
\begin{gather}
  J_M^\bomega(\dot{\gamma}, f) := |D^M(\gamma)|^{\frac{1}{2}} \int_{\mathcal{O}} f \dd \mu, \quad f \in C_c^\infty(M(F_S))
\end{gather}
avec $\gamma \in \mathcal{O}$ quelconque, où $D^M$ est le discriminant de Weyl (voir \ref{def:S-adm}). Pour montrer qu'elle converge, il suffit de remplacer $\mu$ par $|\mu|$. On obtient ainsi une mesure de Radon invariante sur $\mathcal{O}$, ce qui permet d'appliquer le résultat de Rao \cite{Rao72}. On vérifie que, pour tout $y \in M(F_S)$
\begin{gather}
  J_M^\bomega(y\dot{\gamma}y^{-1}, f) = J_M^\bomega(\gamma, f^y) = \bomega(y) J_M^\bomega(\dot{\gamma}, f).
\end{gather}
Cela permet d'immerger $\dot{\Gamma}(M(F_S))^\bomega$ dans l'espace de distributions $\bomega$-équivariantes.

Donnons-en une construction directe. Soit $\gamma \in M(F_S)$ bon et notons $\mathcal{O}$ sa classe de conjugaison. Fixons une mesure invariante sur $M_\gamma(F_S) \backslash M(F_S)$. Alors on peut choisir une unique mesure complexe $\mu[\gamma]$ sur $\mathcal{O}$ de sorte que
\begin{gather}\label{eqn:int-omega-naive}
  J_M^\bomega((\mathcal{O},\mu[\gamma]), f) = \int_{M_\gamma(F_S) \backslash M(F_S)} \bomega(x) f(x^{-1}\gamma x) \dd x.
\end{gather}
Il serait tentant de l'écrire comme $J_M^\bomega(\gamma,f)$; cependant il faut prendre garde qu'il dépend du choix de $\gamma$ dans sa classe de conjugaison.

\paragraph{Induction de classes unipotentes}\index[iFT1]{induction de Lusztig-Spaltenstein}
Supposons que $\gamma \in M(F_S)$ est $\bomega$-bon; soit $\dot{\gamma} \in \dot{\Gamma}(M(F_S))^\bomega$ tel que $\gamma \in \Supp(\dot{\gamma})$.

Supposons pour l'instant que $M_\gamma=G_\gamma$. On peut regarder $\dot{\gamma}$ comme un élément de $\dot{\Gamma}(G(F_S))^\bomega \sqcup \{0\}$ (comme distributions sur $G(F_S)$) en choisissant l'unique mesure telle que
\begin{gather}\label{eqn:de-M-a-G}
  J_G^\bomega(\dot{\gamma}, f) = \int_{M(F_S) \backslash G(F_S)} \bomega(x) J_M^\bomega(\dot{\gamma}, f^{x^{-1}}) \dd x, \quad f \in C_c^\infty(G(F_S)).
\end{gather}
Si l'on fixe des choix comme dans \eqref{eqn:int-omega-naive}, c'est tout simplement
$$ |D^M(\gamma)|^{\frac{1}{2}} \int_{M_\gamma(F_S) \backslash G(F_S)} \bomega(x) f(x^{-1}\gamma x) \dd x;$$
et on a $\dot{\gamma}=0$ dans $G$ si et seulement si $\gamma$ n'est pas $\bomega$-bon dans $G(F_S)$.

En général, notons $A_{M,\text{reg}}$ l'ouvert dense de $A_M$ des éléments $a$ tels que
\begin{eqnarray*}
  \prod_{\beta \in \Sigma_P^{\text{red}}} (\beta(a) - \beta(a)^{-1}) \text{ soit inversible }, \quad  P \in \mathcal{P}(M);
\end{eqnarray*}
alors pour $a \in A_{M,\text{reg}}(F_S)$ en position générale, on a $M_\gamma = M_{a\gamma} = G_{a\gamma}$. Notons $a\dot{\gamma} \in \dot{\Gamma}(M(F_S))^\bomega$ l'élément obtenu via le transport de structure. D'après \eqref{eqn:de-M-a-G}, on regarde $a\dot{\gamma}$ comme un élément de $\dot{\Gamma}(G(F_S))^\bomega$.

Soit maintenant $\gamma \in M(F_S)$ unipotent. Lusztig et Spaltenstein \cite{LS79} ont défini une classe de conjugaison géométrique unipotente $\gamma^G$ dans $G(F_S)$. C'est une réunion finie de classes de conjugaison dans $G(F_S)$, disons $\gamma^G = \sqcup_{i \in I} \gamma^G_i$. Notons $I_0$ l'ensemble des $i \in I$ tels que $\gamma^G_i$ est $\bomega$-bon.

\begin{lemma}[cf. {\cite[(6.6)]{Ar88LB}}]\label{prop:induction-omega}
  Il existe des uniques mesures de Radon $\bomega$-équivariantes non triviales sur  $\{\gamma^G_i\}_{i \in I_0}$ telles que si l'on note
  $$ J_G^\bomega(\dot{\gamma}^G, \cdot) := \sum_{i \in I_0} J_G^\bomega(\dot{\gamma}^G_i, \cdot) $$
  alors
  $$ J_G^\bomega(\dot{\gamma}^G, \cdot) = \lim_{\substack{a \to 1 \\ a \in A_{M,\text{reg}}(F_S)}} J_G^\bomega(a\dot{\gamma}, \cdot) $$
  où les $a$ dans la limite sont supposés en position générale de sorte que $M_{a\gamma}=G_{a\gamma}$.
\end{lemma}
\begin{proof}
  Le cas $\bomega=1$ est démontré dans \cite{Ar88LB}. La même démonstration marche si l'on utilise les mesures $\bomega$-équivariantes sur les classes de conjugaison.
\end{proof}
Cela étant, on peut définir $\dot{\gamma}^G$ comme une combinaison linéaire des éléments dans $\dot{\Gamma}(G(F_S))^\bomega$.

\paragraph{Intégrales orbitales pondérées}\index[iFT1]{$J^\bomega_M(\dot{\gamma}, \cdot)$}
Soit $M$ un sous-groupe de Lévi de $G$. Soient $\gamma \in M(F_S)$ bon et $\mathcal{O}$ sa classe de conjugaison, prenons une paire $\dot{\gamma}=(\mathcal{O},\mu) \in \dot{\Gamma}(M(F_S))^\bomega$.

\begin{definition}
  Supposons que $M_\gamma = G_\gamma$. Si $\gamma$ n'est pas $\bomega$-bon dans $G(F_S)$, posons
  $$ J^\bomega_M(\dot{\gamma}, \cdot)=0;$$
  sinon, $\dot{\gamma}$ induit une paire $(\mathcal{O}', \mu') \in \dot{\Gamma}(G(F_S))^\bomega$ via \eqref{eqn:de-M-a-G}. Arthur a défini une $(G,M)$-famille\index[iFT1]{$v_P, v_M$}
  $$ v_P(\lambda, x) := e^{-\angles{\lambda, H_P(x)}}, \quad P \in \mathcal{P}(M), \lambda \in i\mathfrak{a}_M^* ; $$
  associée à l'ensemble $(G,M)$-orthogonal positif $Y_P(x) = -H_P(x)$. Notons $v_M(x)$ la fonction associée; elle est une fonction sur $M(F_S) \backslash G(F_S)/K_S$. Pour tout $t \in \mathcal{O}'$, écrivons $t = x^{-1}\gamma x$ et définissons une nouvelle mesure en posant $\mu'_M(t) = v_M(x) \mu'(t)$ (avec abus de notations); cela ne dépend pas du choix de $x$.

  Pour $f \in C_c^\infty(G(F_S))$, posons
  $$ J^\bomega_M(\dot{\gamma}, f) := |D^G(\gamma)|^{\frac{1}{2}} \int_{\mathcal{O}'} f \dd\mu'_M .$$
\end{definition}
La convergence découle en remplaçant $\mu'_M$ par $|\mu'_M|$, ce qui nous ramène au cas usuel $\bomega=1$. Si l'on fixe des choix comme dans \eqref{eqn:int-omega-naive}, c'est tout simplement
$$ |D^G(\gamma)|^{\frac{1}{2}} \int_{G_\gamma(F_S) \backslash G(F_S)} \bomega(x) f(x^{-1}\gamma x) v_M(x) \dd x .$$

Revenons au cas général. Soit $L \in \mathcal{L}(M)$, Arthur définit un facteur $r^L_M(\gamma,a)$ pour tous $\gamma \in M(F_S)$, $a \in A_{M,\text{reg}}(F_S)$ (\cite[\S 5]{Ar88LB}), par lequel est définie l'intégrale orbitale pondérée générale. Ce facteur ne dépend que de $L,M,a$ et la classe de conjugaison de $\gamma$ dans $M$. Dans le cas $M_\gamma = G_\gamma$, on a
$$r^L_M(\gamma,a)= \begin{cases} 1, & \text{si } L = M, \\ 0, & \text{sinon}. \end{cases} $$
Si $a \in A_{M,\text{reg}}(F_S)$ est en position générale, alors $M_\gamma = M_{a\gamma} = G_{a\gamma}$ et on sait définir $a\dot{\gamma} \in \dot{\Gamma}(G(F_S))^\bomega$ à l'aide de \eqref{eqn:de-M-a-G}.

\begin{theorem}[cf. {\cite[5.2]{Ar88LB}}]
  Pour tout $f \in C_c^\infty(G(F_S))$, la limite
  $$ J_M^\bomega(\dot{\gamma}, f) := \lim_{\substack{a \to 1 \\ a \in A_{M,\text{reg}}(F_S)}} \sum_{L \in \mathcal{L}(M)} r^L_M(\gamma,a) J_L^\bomega(a\dot{\gamma}, f) $$
  existe, où les $a$ dans la limite sont supposés en position générale de sorte que $M_{a\gamma}=G_{a\gamma}$. Si $M_\gamma = G_\gamma$, elle coïncide avec la définition \ref{def:int-ponderee-nr}. On a
  $$ \forall y \in M(F_S), \quad J_M^\bomega(y \dot{\gamma} y^{-1}, f) = \bomega(y) J_M^\bomega(\dot{\gamma}, f). $$
\end{theorem}
\begin{proof}
  Les termes à droite sont bien définis. La démonstration de l'existence de la limite est pareille que celle dans \cite{Ar88LB}: il suffit de tordre les mesures invariantes sur $G(F_S)$, $K_S$ ou sur les orbites par $\bomega$, et on vérifie que cela n'affecte pas les démonstrations car $\bomega$ est unitaire. L'assertion sur l'équivariance est claire si $M_\gamma = G_\gamma$; le cas général s'en suit par définition.
\end{proof}

Lorsqu'une ambiguïté sur $G$ sera à craindre, nous noterons les objets par $J^{G,\bomega}_M(\dot{\gamma},f)$, etc.

\begin{proposition}[cf. {\cite[6.2]{Ar88LB}}]\label{prop:int-ponderee-abscont}
  Soit $\dot{u} \in \dot{\Gamma}(M(F_S))^\bomega$ supporté sur une classe de conjugaison unipotente. Alors $f \mapsto J_M^\bomega(\dot{u}, f)$ définie une mesure complexe $\bomega$-équivariante sur l'induite $u^G$ qui est absolument continue par rapport à la mesure définie dans \ref{prop:induction-omega}.
\end{proposition}

\paragraph{Le cas non ramifié}
Fixons $G$ et $M$ comme précédemment. Supposons que $S$ consiste en places non archimédiennes et fixons $K_S := \prod_{v \in S} K_v$ tel que $K_v$ est un sous-groupe hyperspécial de $G(F_v)$ en bonne position relativement à $M$ pour chaque $v \in S$. Notons $\mathbbm{1}_{K_S}$ la fonction caractéristique de $K_S$.

\begin{definition}\index[iFT1]{$r^\bomega_{M,K_S}(\dot{\gamma})$}
  Les intégrales orbitales pondérées non ramifiées sont définies par
  $$ r^\bomega_{M,K_S}(\dot{\gamma}) = r^{G,\bomega}_{M,K_S}(\dot{\gamma}) := J_M^\bomega(\dot{\gamma}, \mathbbm{1}_{K_S}), \quad \dot{\gamma} \in \dot{\Gamma}(G(F_S))^\bomega .$$
\end{definition}

Lorsque $\gamma$ est fortement régulier dans $G$, notre définition est celle dans \cite{Wa09}.

\subsection{Comportement des intégrales orbitales pondérées avec un caractère}\label{sec:comport-int-ponderee-omega}
Soient $\gamma \in M(F_S)$ $\bomega$-bon et prenons $\dot{\gamma} \in \dot{\Gamma}(M(F_S))^\bomega$ (i.e. on  fixe la mesure) comme précédemment.

\paragraph{$(M,\sigma)^\bomega$-équivalence}
Soient $\sigma \in M(F_S)_\text{ss}$ et $\Sigma \subset \sigma M_\sigma(F_S)$ un ouvert invariant par $M_\sigma(F_S)$. Notons
$$\dot{\Gamma}(\Sigma)^\bomega := \{ \dot{\gamma} \in \dot{\Gamma}(M(F_S))^\bomega : \Supp(\dot{\gamma} ) \subset \Sigma \}. $$
Supposons désormais que l'adhérence de $\Sigma$ dans $\sigma M_\sigma(F_S)$ contient un voisinage invariant de $\sigma$. Suivant \cite[p.235]{Ar88LB}, on dit que deux fonctions $\phi_1$, $\phi_2$ sur $\dot{\Gamma}(\Sigma)^\bomega$ sont $(M,\sigma)^\bomega$-équivalentes s'il existe $f \in C_c^\infty(M(F_S))$ et un voisinage $U$ de $\sigma$ dans $M(F_S)$ tels que
$$ (\phi_1-\phi_2)(\dot{\gamma}) = J^{M,\bomega}_M(\dot{\gamma},f) $$
pour tout $\dot{\gamma} \in \dot{\Gamma}(\Sigma)^\bomega$ tel que $\Supp(\dot{\gamma}) \subset U$. Si cette condition est vérifiée, on écrit\index[iFT1]{$(M,\sigma)^\bomega$-équivalence}
$$ \phi_1 \stackrel{(M,\sigma)^\bomega}{\sim} \phi_2 .$$

\begin{proposition}[cf. {\cite[2.2]{Ar88LB}}]
  Si $M_\sigma=G_\sigma$, alors pour tout $f \in C_c^\infty(M(F_S))$ on a
  $$ J_M^\bomega(\dot{\gamma},f) \stackrel{(M,\sigma)^\bomega}{\sim} 0 $$
  pour tout $\dot{\gamma} \in \dot{\Gamma}(M(F_S))^\bomega$ assez proche de $\sigma$ modulo conjugaison.
\end{proposition}
\begin{proof}
  La démonstration est identique à celle du cas $\bomega=1$.
\end{proof}

\paragraph{Formules de descente}\index[iFT1]{formule de descente}
\begin{proposition}[cf. {\cite[6.2]{Ar88LB}}]
  Supposons $\gamma$ unipotent. Soit $L_1 \in \mathcal{L}(M)$, alors pour tout $f \in C_c^\infty(G(F_S))$, on a
  $$ J_{L_1}^\bomega(\dot{\gamma}^{L_1}, f) = \lim_{\substack{a \to 1 \\ a \in A_{M,\text{reg}}(F_S)}} \sum_{L \in \mathcal{L}(L_1)} r^L_{L_1}(\gamma,a) J_L^\bomega(a\dot{\gamma}, f). $$
\end{proposition}

Soit $Q = L U_Q \in \mathcal{F}(M)$. Définissons la version locale de \eqref{eqn:descente-parabolique}:
$$ f_Q^\bomega(m) := \delta_Q(m)^{\frac{1}{2}} \int_{K_S} \int_{U_Q(F_S)} \bomega(k) f(k^{-1}muk) \dd u \dd k, \quad m \in L(F_S). $$

\begin{corollary}[cf. {\cite[\S 8]{Ar88-TF1}}]\label{prop:J-descente}
  Conservons les notations précédentes et fixons $\xi \in \mathfrak{a}^{L_1}_M$ en position générale comme dans \eqref{eqn:xi-1}, ce qui permet d'associer à chaque $L \in \mathcal{L}(M)$ tel que $d^G_M(L_1,L) \neq 0$ un $Q_L \in \mathcal{P}(L)$. Alors
  $$ J_{L_1}^\bomega(\dot{\gamma}^{L_1},f) = \sum_{L \in \mathcal{L}(M)} d^G_M(L_1,L) J^{L,\bomega}_M(\dot{\gamma}, f^\bomega_{Q_L}). $$
\end{corollary}
\begin{proof}
  L'énoncé dans \cite{Ar88-TF1} est pour les distributions invariantes; or le cas $\bomega=1$ du résultat voulu y est implicite. Les arguments d'Arthur s'adaptent de façon usuelle au cas général.
\end{proof}

De même, on a la formule de déploiement pour intégrales orbitales pondérées. Prenons $\xi \in \mathfrak{a}^M_\mathcal{M}$ en position générale comme dans \eqref{eqn:xi-2}, ce qui permet de définir une application $(L_1,L_2) \mapsto (Q_1, Q_2)$ avec $Q_i \in \mathcal{P}(L_i)$ ($i=1,2$) pourvu que $d^G_M(L_1,L_2) \neq 0$.

\begin{proposition}[cf. {\cite[9.1]{Ar88-TF2}}]\label{prop:J-deploiement}
  Supposons $S=S_1 \sqcup S_2$. Soit $\dot{\gamma} = \dot{\gamma}_1 \dot{\gamma}_2$, où $\dot{\gamma}_i \in \dot{\Gamma}(M(F_{S_i}))^\bomega$ pour $i=1,2$. Soit $f = f_1 f_2 \in C_c^\infty(G(F_S))$ où $f_i \in C_c^\infty(G(F_{S_i}))$ pour $i=1,2$. Alors
  $$ J_M^\bomega(\dot{\gamma},f) = \sum_{L_1, L_2 \in \mathcal{L}(M)} d^G_M(L_1, L_2) J^{L_1}_M(\dot{\gamma}_1, f_{Q_1}^\bomega) J^{L_2}_M(\dot{\gamma}_2, f_{Q_2}^\bomega). $$
\end{proposition}

\paragraph{Non-invariance}
Conservons les mêmes notations. Soit $Q \in \mathcal{F}(M_0)$. On a la $(G,M_Q)$-famille
$$ u_P(\lambda,x,y) = e^{-\angles{\lambda,H_P(k_P(x)y)}}, \quad x,y \in G(F_S). $$

On définit la version locale de \eqref{eqn:descente-parabolique-y}
$$ f^\bomega_{Q,y}(m) = \delta_Q(m)^{\frac{1}{2}} \iint_{K_S \times U_Q(F_S)} \bomega(k) f(k^{-1}muk) u'_Q(k,y) \dd u \dd k, \quad m \in M_Q(F_S). $$

\begin{proposition}[cf. {\cite[(8.2)]{Ar81}}]\label{prop:non-invariance-omega-pondere}
  Pour tout $y \in G(F_S)$, on a
  $$ J_M^\bomega(\dot{\gamma}, f^y) = \bomega(y) \sum_{Q \in \mathcal{F}(M)}  J_M^{M_Q,\bomega}(\dot{\gamma}, f^\bomega_{Q,y}). $$
\end{proposition}
\begin{proof}
  Prenons $\gamma \in \Supp(\dot{\gamma})$. Traitons d'abord le cas $M_\gamma = G_\gamma$. Fixons des mesures comme dans \eqref{eqn:int-omega-naive}, alors
  \begin{align*}
    J_M^\bomega(\dot{\gamma},f^y) & = |D^M(\gamma)|^{\frac{1}{2}} \int_{M_\gamma(F_S) \backslash G(F_S)} \bomega(x) f(yx^{-1} \gamma xy^{-1}) v_M(x) \dd x \\
    & = |D^M(\gamma)|^{\frac{1}{2}} \bomega(y) \int_{M_\gamma(F_S) \backslash G(F_S)} \bomega(x) f(x^{-1} \gamma x) v_M(xy) \dd x .
  \end{align*}
  On a $v_P(\lambda, xy) = v_P(\lambda) u_P(\lambda, x, y)$, donc \ref{prop:descente-cd'} entraîne que
  $$ v_M(xy) = \sum_{Q \in \mathcal{F}(M)} v^Q_M(x) u'_Q(x,y). $$
  Pour $u$ fixé, on a
  \begin{multline*}
    \int_{M_\gamma(F_S) \backslash G(F_S)} \bomega(x)f(x^{-1}\gamma x) v^Q_M(x) u'_Q(x,y) \dd x = \\
    \iiint_{U_Q(F_S) \times M_Q(F_S) \times K_S} \bomega(k)\bomega(m) \delta_Q(m)^{-1} f(k^{-1}m^{-1} u^{-1} \gamma u m k) v^Q_M(m)  u'_Q(k,y) \dd u \dd m \dd k = \\
    \int_{M_Q(F_S)} \bomega(m) v^Q_M(m) \iint_{K_S \times U_Q(F_S)} \bomega(k) \delta_Q(m)^{\frac{1}{2}} f(k^{-1}m^{-1} \gamma m u k) u'_Q(k,y) \dd u \dd k \dd m .
  \end{multline*}
  La $(M_Q,M)$-famille $v^Q_M(\lambda,m)$ ne dépend que de $M$ et $M_Q$ lorsque $m \in M_Q(F_S)$ (cf. \cite[p.41]{Ar81}), on peut la noter $v^{M_Q}_M(\lambda,m)$ et on vérifie qu'elle donne la fonction de poids pour $M_Q, M, K_S$. Cela conclut le cas $M_\gamma = G_\gamma$.

  En général, on en déduit
  \begin{align*}
    J_M^\bomega(\dot{\gamma}, f^y) & = \lim_{a \to 1} \sum_{L \in \mathcal{L}(M)} r^L_M(\gamma, a) \bomega(y) \sum_{Q \in \mathcal{F}^G(L)} J_L^{M_Q,\bomega}(a\dot{\gamma}, f_{Q,y}^\bomega) \\
    & = \bomega(y) \sum_{Q \in \mathcal{F}^G(M)} \lim_{a \to 1} \left( \sum_{L \in \mathcal{L}^{M_Q}(M)} r^L_M(\gamma,a) J^{M_Q,\bomega}_L(a\dot{\gamma}, f^\bomega_{Q,y}) \right) \\
    & = \bomega(y) \sum_{Q \in \mathcal{F}^G(M)} J_M^{M_Q,\bomega}(f_{Q,y}^\bomega).
  \end{align*}
\end{proof}

\begin{proposition}\label{prop:J-transport}
  Soit $y \in G(F_S)$ tel que $yMy^{-1} \in \mathcal{L}(M_0)$ , alors
  $$ J_{yMy^{-1}}^\bomega(y\dot{\gamma}y^{-1}, f) = \bomega(y) J_M^\bomega(\dot{\gamma}, f). $$
\end{proposition}
\begin{proof}
  Comme $K_S$ est spécial, on peut écrire $y = km$ où $k \in \pi((K_S)_\text{sc})$ (i.e. $k$ provient du revêtement simplement connexe de $G_\text{der}$) et $m \in M(F_S)$. Le problème se divise ainsi en deux cas: $y \in M(F_S)$ et $y \in \pi((K_S)_\text{sc})$.

  Si $y \in M(F_S)$, l'assertion découle de la proposition précédente. Si $y \in \pi((K_S)_\text{sc})$, un transport de structure donne
  $$ J_{yMy^{-1}}^\bomega(y\dot{\gamma}y^{-1}, f) = J_M^\bomega(\dot{\gamma}, f^y). $$

  Comme $\bomega(y)=1$, il suffit de montrer que $f_{Q,y}^\bomega=0$ si $Q \neq G$. Or c'est clair que $u_Q(k,y)=0$ si $Q \neq G$, ce qui permet de conclure.
\end{proof}

\paragraph{Dépendance de $K_S$}
Soit $K_{1,S}$ un sous-groupe compact maximal de $G(F_S)$ en bonne position relativement à $M_0$. Ajoutons l'affixe $K_1$ aux objets définis par rapport à $K_{1,S}$. Soient $M \in \mathcal{L}(M_0)$, $T \in \mathfrak{a}_M$. Définissons les $(G,M)$-familles
\begin{align*}
  v_P(\lambda,x,T) & := e^{\angles{\lambda, -H_P(x)+T}}, \\
  v_P(\lambda,x,T;K_1) & := e^{\angles{\lambda, -H_P(x; K_1)+T}}, \quad P \in \mathcal{P}(M).
\end{align*}
La définition originelle des fonctions poids correspond au cas $T=0$, mais $T$ n'affecte pas les fonctions $v_M(x;K_1) = v_M(x,T;K_1)$ et $v_M(x) = v_M(x,T)$.

Définissons la $(G,M)$-famille
$$ u_P(\lambda,x; K_1|K, T) := e^{\angles{\lambda, -H_P(k_P(x);K_1)+T}}, \quad P \in \mathcal{P}(M). $$

Posons
$$ f^\bomega_{Q,K_1|K,T}(m) := \delta_Q(m)^{\frac{1}{2}} \int_{K_S} \int_{U_Q(F_S)} \bomega(k) f(k^{-1}muk) u'_Q(k; K_1|K, T) \dd u \dd k . $$

\begin{proposition}[cf. {\cite[3.4]{Ar98}}]\label{prop:dependance-K-ponderee}
  Soit $T \in \mathfrak{a}_M$. On a
  $$ J_M^\bomega(\dot{\gamma},f;K_1) = \sum_{Q \in \mathcal{F}(M)} J_M^{M_Q,\bomega}(\dot{\gamma}, f^\bomega_{Q,K_1|K,T}). $$
\end{proposition}
\begin{proof}
  Il suffit de comparer les fonctions de poids: nous avons remarqué que celle associée à $K_1$ se déduit de la $(G,M)$-famille $v_P(x,T;K_1)$. Or
  $$ -H_P(x;K_1)+T = -H_P(x)-H_P(k_P(x);K_1)+T, \quad P \in \mathcal{P}(M). $$
  D'où
  $$ v_P(\lambda,x,T;K_1) = v_P(\lambda,x,0) u_P(\lambda,x;K_1|K, T), \quad P \in \mathcal{P}(M). $$
  On peut reprendre la preuve de \ref{prop:non-invariance-omega-pondere} à partir de maintenant.
\end{proof}

\subsection{Développement fin du terme unipotent}\label{sec:dev-fin-unip}
Fixons $S$ un sous-ensemble fini de places de $F$. Supposons que
\begin{itemize}
  \item $S$ contient toutes les places archimédiennes;
  \item $K_v$ est hyperspécial pour tout $v \notin S$;
  \item $\bomega$ est trivial sur $K^S$.
\end{itemize}
Désignons par $f_{K^S}$ la fonction caractéristique de $K^S$. On définit un homomorphisme continu injectif
\begin{align*}
  C_c^\infty(G(F_S)) & \to C_c^\infty(G(\A)) \\
  f & \mapsto f \cdot f_{K^S},
\end{align*}
par lequel on identifie $C_c^\infty(G(F_S))$ à un sous-espace de $C_c^\infty(G(\A))$.

Soit $M$ un sous-groupe de Lévi de $G$, notons
\begin{align*}
  \dot{\Gamma}_\text{unip}(M(F_S))^\bomega &:= \{ \dot{u} \in \dot{\Gamma}(M(F_S))^\bomega : \Supp(\dot{u}) \subset M_\text{unip}(F_S) \}, \\
  \dot{\Gamma}_\text{unip}(M(F),S)^\bomega & := \{ \dot{u} \in \dot{\Gamma}_\text{unip}(M(F_S))^\bomega : \Supp(\dot{u}) \cap M(F) \neq \emptyset \}.
\end{align*}
En oubliant les mesures, on définit $\Gamma_\text{unip}(M(F_S))^\bomega$ et $\Gamma_\text{unip}(M(F),S)^\bomega$ de la même manière.

\begin{theorem}[cf. \cite{Ar85}]\label{prop:int-unip-developpement}\index[iFT1]{$a^{M,\bomega}(S,\cdot)$}
  Il existe une unique application $a^{M,\bomega}(S,\cdot): \dot{\Gamma}_{\mathrm{unip}}(M(F),S))^\bomega \to \C$ pour tout $M \in \mathcal{L}(M_0)$, satisfaisant à l'équivariance
  $$ a^{M,\bomega}(S,y\dot{u}y^{-1}) = \bomega(y)^{-1} a^{M,\bomega}(S,\dot{u}), \quad y \in M(F_S), $$
  telle que, pour tout $f \in C_c^\infty(G(F_S) \cap G(\A)^1)$, on a
  \begin{gather*}
    J_{\mathrm{unip}}^\bomega(f) = \sum_{M \in \mathcal{L}(M_0)} |W_0^M| |W_0^G|^{-1} \sum_{u \in \Gamma_\mathrm{unip}(M(F),S)^\bomega}  a^{M,\bomega}(S,\dot{u}) J_M^\bomega(\dot{u},f)
  \end{gather*}
  où $\dot{u} \in \dot{\Gamma}_\mathrm{unip}(M(F),S)$ est une image réciproque de $u$ quelconque.  De plus, $a^{M,\bomega}(S,\cdot)$ ne dépend pas de $K_S$.
\end{theorem}

Les coefficients $a^{M,\bomega}(S,\cdot)$ dépendent encore de $M$, $M_0$ et $K^S$. La dépendance de $M_0$ sera enlevée plus tard  par \ref{prop:coef-indep-M_0}. L'équivariance des coefficients affirme que le produit $a^{M,\bomega}(S,\dot{u}) J_M^\bomega(\dot{u},f)$ ne dépend que de $u$.

\begin{proof}
  C'est le résultat principal de \cite{Ar85}. Il n'y a rien à prouver si $G$ est anisotrope modulo son centre. Supposons donc par récurrence que l'assertion est vérifiée pour tout sous-groupe de Lévi propre. Posons
  $$ T^\bomega(f) := J_\text{unip}^\bomega(f) - \sum_{\substack{M \in \mathcal{L}(M_0) \\ M \neq G}} |W^M_0| |W^G_0|^{-1} \sum_{u \in \Gamma_\mathrm{unip}(M(F),S)^\bomega} a^{M,\bomega}(S,\dot{u}) J_M^\bomega(\dot{u},f). $$

  Soit $y \in G(\A)$. La formule de non-invariance \ref{prop:non-variance-Jomega} entraîne que
  \begin{gather*}
    J_\text{unip}^\bomega(f^y) - \bomega(y)J_\text{unip}^\bomega(f) = \bomega(y) \sum_{\substack{Q \in \mathcal{F}(M_0) \\ Q \neq G}} |W_0^{M_Q}| |W_0^G|^{-1} J_\text{unip}^{M_Q, \bomega}(f_{Q,y}^\bomega),
  \end{gather*}
  tandis que \ref{prop:non-invariance-omega-pondere} entraîne que
  \begin{multline*}
    \sum_{\substack{M \in \mathcal{L}(M_0) \\ M \neq G}} |W_0^M| |W_0^G|^{-1} \sum_u a^{M,\bomega}(S,\dot{u}) (J^\bomega_M(\dot{u},f^y)-\bomega(y)J^\bomega_M(\dot{u},y)) \\
    = \bomega(y) \sum_{\substack{M \in \mathcal{L}(M_0) \\ M \neq G}} |W_0^M| |W_0^G|^{-1} \sum_{\substack{Q \in \mathcal{F}(M) \\ Q \neq G}} \sum_u a^{M,\bomega}(S,\dot{u}) J_M^{M_Q,\bomega}(\dot{u},f^{\bomega}_{Q,y}) \\
    = \bomega(y) \sum_{\substack{Q \in \mathcal{F}(M_0) \\ Q \neq G}} |W_0^{M_Q}| |W_0^G|^{-1} \sum_{M \in \mathcal{L}^{M_Q}(M_0)} |W_0^M| |W_0^{M_Q}|^{-1} \sum_u a^{M,\bomega}(S,\dot{u}) J_M^{M_Q,\bomega}(\dot{u}, f_{Q,y}^\bomega).
  \end{multline*}

  D'après l'hypothèse de récurrence, on en déduit $T^\bomega(f^y)=\bomega(y)T^\bomega(f)$. Par conséquent $T^\bomega$ est $\bomega$-équivariant. D'autre part, on montre (cf. \cite[4.2]{Ar85}) que si $f$ s'annule sur $G_\text{unip}(\A)$, alors $J^\bomega_\text{unip}(f)=0$. La même propriété est satisfaite par les distributions $J^\bomega_M(\dot{u},\cdot)$ d'après \ref{prop:int-ponderee-abscont}, donc par $T^\bomega(\cdot)$. Par conséquent, il existe des coefficients $a^{G,\bomega}(S,\dot{u})$ satisfaisant à la condition d'équivariance telle que
  $$ T^\bomega(f) = \sum_{u \in \Gamma_\text{unip}(M(F_S))^\bomega} a^{G,\bomega}(S,\dot{u}) J^\bomega_G(\dot{u},f). $$
  pour tout $f$. Si l'on sélectionne une image réciproque $\dot{u}$ pour chaque $u$, alors la famille de distributions $\{J^\bomega_G(\dot{u}, f): u \in \Gamma_\text{unip}(M(F_S))^\bomega \}$ est libre. L'unicité de $a^{G,\bomega}(S,\cdot)$ en découle.

  Il reste à montrer que les classes qui contribuent sont celles rencontrant $G(F)$. C'est l'ingrédient technique de \cite[\S 2 - \S 7]{Ar85}. On vérifie que les troncatures et estimations d'Arthur dans \cite{Ar85} demeurent valables si l'on utilise la mesure complexe $\bomega(x)\dd x$ sur $G(\A)$ et les autres groupes en question.

  Montrons l'indépendance de $K_S$. Soit $K_{1,S}$ un autre sous-groupe compact de $G(F_S)$ en bonne position relativement à $M_0$. Ajoutons l'affixe $K_1$ aux objets associés au sous-groupe compact maximal $K_1 = K_{1,S} \times K^S$ de $G(\A)$. On reprend les arguments ci-dessus en utilisant \ref{prop:dependance-K} et \ref{prop:dependance-K-ponderee} avec $T=T_1-T_0$ où $T_1$ (resp. $T_0$) est le paramètre de troncature canonique pour $K_1$ (resp. $K$), pour obtenir
  $$ T^\bomega(f;K_1) = T^\bomega(f). $$
  Comme la distribution $J_G^\bomega(\dot{u},\cdot)$ ne dépend que de $\dot{u}$, on tire du développement de $T^\bomega$ que $a^{G,\bomega}(S,\dot{u};K_1) = a^{G,\bomega}(S,\dot{u})$. Ceci est aussi valable pour tout $M \in \mathcal{L}(M_0)$ au lieu de $G$, ce qu'il fallait démontrer.
\end{proof}

Maintenant, prenons un sous-ensemble fini $S_+$ de places tel que $S_+ \supset S$. Fixons $M \in \mathcal{L}(M_0)$ et posons $K^\# := K \cap M(F_{S_+}^S)$. Si $\dot{u} \in \dot{\Gamma}(L(F),S_+)^\bomega$, on choisit une décomposition $\dot{u}_S := \prod_{v \in S} \dot{u}_v$ (resp. $\dot{u}_{S_+}^S := \prod_{v \in S_+ \setminus S} \dot{u}_v$); c'est bien déterminé à multiplication près par $\{(\lambda_v)_{v \in S} \in (\C^\times)^S : \prod_v \lambda_v = 1 \}$ (resp. $(\lambda_v)_{v \in S_+ \setminus S}$, etc). Soit $D = \sum_{i \in I} \dot{u}_i$ une somme finie d'éléments dans $\dot{\Gamma}(G(F_S))^\bomega$; pour $\dot{u} \in \dot{\Gamma}(G(F_S))^\bomega$, écrivons
$$ \frac{D}{\dot{u}} := \sum_{\substack{i \in I \\ \dot{u}_i \sim \dot{u}}} \frac{\dot{u}_i}{\dot{u}} \in \C . $$

\begin{proposition}\label{prop:S_+-S}
  Soit $\dot{v} \in \dot{\Gamma}_{\mathrm{unip}}(M(F),S)^\bomega$, alors
  $$ a^{M,\bomega}(S,\dot{v}) = \sum_{L \in \mathcal{L}^M(M_0)} |W_0^L| |W_0^M|^{-1} \sum_{u \in \Gamma_{\mathrm{unip}}(L(F),S_+)^\bomega} \frac{(\dot{u}_S)^M}{\dot{v}} \cdot  a^{L,\bomega}(S_+,\dot{u}) \cdot r^{M,\bomega}_{L, K^\#}(\dot{u}_{S_+}^S). $$
\end{proposition}
On vérifie sans peine que l'expression dans la somme ne dépend pas du choix de $\dot{u}$.

\begin{proof}
  Prenons $\xi \in \mathfrak{a}^M_\mathcal{M} := \{(H,-H): H \in \mathfrak{a}_M \}$ en position générale comme dans \eqref{eqn:xi-2} tel que sa projection dans $\mathfrak{a}_M^G$ vérifie aussi les conditions pour \eqref{eqn:xi-1}. On dispose alors d'une application $(L_1,L_2) \mapsto (Q_1, Q_2)$ pour $d^G_M(L_1,L_2) \neq 0$.

  Notons $\mathbbm{1}_{S_+}^S \in C_c^\infty(G(F_{S_+}^S))$ la fonction caractéristique de $K_{S_+}^S$. On déduit de \ref{prop:int-unip-developpement} (appliqué à $S_+$) et de \ref{prop:J-deploiement} que pour tout $f \in C_c^\infty(G(F_S) \cap G(\A)^1)$,
  \begin{align*}
    J^\bomega_\text{unip}(f) &= \sum_{L \in \mathcal{L}(M_0)} |W_0^L| |W_0^G|^{-1} \sum_{u \in \Gamma_\text{unip}(L(F),S_+)^\bomega} a^{L,\bomega}(S_+,\dot{u}) J_L^\bomega(\dot{u}, f \cdot \mathbbm{1}_{S_+}^S) \\
    & = \sum_L |W_0^L| |W_0^G|^{-1} \sum_{u \in \Gamma_\text{unip}(L(F),S_+)^\bomega} a^{L,\bomega}(S_+,\dot{u}) \cdot \\
    & \cdot \sum_{M,M' \in \mathcal{L}(L)} d^G_L(M',M) J^{M',\bomega}_L(\dot{u}_S, f^\bomega_{Q'}) J^{M,\bomega}_L(\dot{u}_{S_+}^S, (\mathbbm{1}_{S_+}^S)^\bomega_Q).
  \end{align*}

  Comme $\bomega$ est supposé trivial sur $K^S$, on a
  \begin{align*}
    (\mathbbm{1}_{S_+}^S)^\bomega_Q(m) &= \delta_Q(m)^{\frac{1}{2}} \int_{K_{S_+}^S} \int_{U_Q(F_{S_+}^S)} \mathbbm{1}_{S_+}^S(k^{-1}muk) \dd k \dd u \\
    & = \delta_Q(m)^{\frac{1}{2}} \int_{U_Q(F_{S_+}^S)} \mathbbm{1}_{S_+}^S(mu) \dd u = \mathbbm{1}_{K^\#}(m), \quad m \in M(F_{S_+}^S).
  \end{align*}
  Donc \ref{prop:J-descente} entraîne que
  \begin{align*}
    J_\text{unip}^\bomega(f) &= \sum_{L,u} |W_0^L| |W_0^G|^{-1} a^{L,\bomega}(S_+, \dot{u}) \sum_{M \in \mathcal{L}(L)} r^{M,\bomega}_{L,K^\#}(\dot{u}_{S_+}^S) \left( \sum_{M' \in \mathcal{L}(M)} d^G_L(M',M) J^{M',\bomega}_L(\dot{u}_S, f_{Q'}^\bomega) \right) \\
    & = \sum_{L,u} |W_0^L| |W_0^G|^{-1} a^{L,\bomega}(S_+, \dot{u})  \sum_{M \in \mathcal{L}(L)} r^{M,\bomega}_{L,K^\#}(\dot{u}_{S_+}^S) J^\bomega_M((\dot{u}_S)^M, f).
  \end{align*}
  On l'écrit comme
  \begin{multline*}
    \sum_{M \in \mathcal{L}(M_0)} |W_0^M| |W_0^G|^{-1} \sum_{v \in \Gamma_\text{unip}(M(F),S)^\bomega} J^\bomega_M(\dot{v},f) \cdot \\
    \cdot \left( \sum_{L \in \mathcal{L}^M(M_0)} |W_0^L| |W_0^M|^{-1} \sum_{u \in \Gamma_\text{unip}(L(F),S_+)^\bomega} \frac{(\dot{u}_S)^M}{\dot{v}} \cdot a^{L,\bomega}(S_+, \dot{u}) r^{M,\bomega}_{L, K^\#}(\dot{u}_{S_+}^S) \right).
  \end{multline*}
  Vu le développement \ref{prop:int-unip-developpement} appliqué à $S$ et l'unicité des coefficients, l'assertion en découle.
\end{proof}

Ci-dessous un interlude élémentaire de la théorie de Bruhat-Tits. Soit $E$ un corps local non archimédien. 

\begin{lemma}\label{prop:conjugaison-K-dans-M_0}
  Soient $H$ un $E$-groupe réductif connexe et $L$ un sous-groupe de Lévi. Soit $K_H$ un sous-groupe compact maximal de $H(E)$ en bonne position relativement à $L$. Si $K_H'$ est un sous-groupe compact maximal en bonne position relativement à $L$ et conjugué à $K_H$ par $H(E)$, alors $K_H'$ est conjugué par $K_H$ par $L(E)$.
\end{lemma}
\begin{proof}
  Supposons d'abord que $L$ est minimal, alors c'est le centralisateur d'un $E$-tore déployé maximal $T_0$. D'après la définition de l'immeuble de Bruhat-Tits \cite[7.4.1]{BT72}, $K_H$ et $K_H'$ sont conjugués par $N_H(T_0)(E)$. Comme $K_H$ est spécial, il contient des représentants du groupe de Weyl de $T_0$. Par conséquent $K_H$ et $K_H'$ sont conjugués par $L(E)$.

  En général, $K_H$ et $K'_H$ sont associés à des points dans l'immeuble élargi de $L$, qui se plonge dans celui de $H$ de façon équivariante. Quitte à les conjuguer par $L(E)$, on peut supposer qu'ils appartiennent au même appartement, ce qui nous ramène au cas précédent.
\end{proof}

\begin{proposition}\label{prop:coef-indep-M_0}
  Soient $M_0'$ un autre sous-groupe de Lévi minimal en bonne position relativement à $K$. Si $M \supset M_0$ et $M \supset M'_0$, alors les coefficients $a^{M,\bomega}(S,\cdot)$ associés à $M_0'$ coïncident avec ceux associés à $M_0$.
\end{proposition}
\begin{proof}
  Afin de souligner la dépendance en question, notons $a^{M,\bomega}(S,\cdot; K^S, M_0)$ (resp. $a^{M,\bomega}(S,\cdot; K^S, M_0')$) les coefficients associés à $K^S$ et $M_0$ (resp. $M_0'$). Notons $\pi: M_\text{SC} \to M$ le revêtement simplement connexe de $M_\text{der}$; si $L \in \mathcal{L}^M(M_0)$, notons $L_\text{sc}$ son image réciproque par $\pi$. Prenons $y \in \pi(M_\text{SC}(F))$ tel que $y M_0' y^{-1} = M_0$, alors $y K^S y^{-1}$ est en bonne position relativement à $M_0$ et $\bomega(y_v)=1$ pour tout $v$.

  Le transport de structure induit par $x \mapsto yxy^{-1}$ donne
  $$ a^{M,\bomega}(S,\dot{v}; K^S, M_0') = a^{M,\bomega}(S, y\dot{v}y^{-1}; (yKy^{-1})^S, M_0) = a^{M,\bomega}(S, \dot{v}; (yKy^{-1})^S, M_0) $$
  pour tout $\dot{v} \in \dot{\Gamma}_\text{unip}(M(F),S)^\bomega$. On peut oublier $M_0'$ dès maintenant et se ramener à montrer que $a^{M,\bomega}(S, \cdot; (yKy^{-1})^S) = a^{M,\bomega}(S, \cdot; (yKy^{-1})^S)$.

  Prenons $S_+ \supset S$ de sorte que $y_v \in K_v$ pour $v \notin S_+$. Vu \ref{prop:S_+-S}, il suffit de montrer que pour tout $L \in \mathcal{L}^M(M_0)$,
  $$ r^{M,\bomega}_{L,K^\#}(\dot{u}_{S_+}^S) = r^{M,\bomega}_{L,y K^\# y^{-1}}(\dot{u}_{S_+}^S) $$
  pour tout $\dot{u} \in \dot{\Gamma}_\text{unip}(M(F),S_+)^\bomega$, où $K^\# := K \cap M(F_{S_+}^S)$ comme dans \ref{prop:S_+-S} et la notation $y$ est quelque peu abusive pour signifier aussi $y_{S_+}^S$. D'après \ref{prop:conjugaison-K-dans-M_0}, il existe $z \in \pi(L_\text{sc}(F_{S_+}^S))$ tel que $y K^\# y^{-1} = z K^\# z^{-1}$. Le transport de structure via $x \mapsto zxz^{-1}$ donne
  $$ r^{M,\bomega}_{L,K^\#}(\dot{u}_{S_+}^S) = r^{M,\bomega}_{L,z K^\# z^{-1}}(z\dot{u}_{S_+}^S z^{-1}) = \bomega(z) r^{M,\bomega}_{L,z K^\# z^{-1}}(\dot{u}) = \bomega(z) r^{M,\bomega}_{L,y K^\# y^{-1}}(\dot{u}). $$
  Or $\bomega(z)=1$, ce qui achève la démonstration.
\end{proof}

\begin{remark}\label{prop:coef-dependance}
  Le bilan est que les coefficients $a^{M,\bomega}(S,\cdot)$ sont déterminés par les données $M$, $\bomega$, $S$, et le sous-groupe compact maximal $K^S$ de $M(F^S)$ tels que
  \begin{itemize}
    \item il existe un sous-groupe de Lévi minimal $M_0$ de $M$, défini sur $F$, qui est en bonne position relativement à $K^S$;
    \item $\bomega$ est trivial sur $K^S$.
  \end{itemize}
\end{remark}

\subsection{Interlude: $S$-admissibilité}
Pour l'instant, soit $M$ un $F$-groupe réductif connexe quelconque, et soit $S$ un ensemble fini de places tel que $M$ est non ramifié en dehors de $S$. La définition suivante fournit une façon explicite de dire que $S$ est suffisamment grand.

\begin{definition}[cf. {\cite[\S 1]{Ar02}}]\label{def:S-adm}\index[iFT1]{$S$-admissible}
  Définissons un morphisme invariant $\mathcal{D}=(D_0,\ldots, D_d): M \to \Ga^{d+1}$ avec $d := \dim M$, par
  $$ \det(1+t-\Ad(x)|\mathfrak{m}) = \sum_{k=0}^d D_k(x) t^k \in F[t]. $$

  Observons que $D_d=1$. Pour $X=(X_0,\ldots, X_d) \in F^{d+1}$, notons $X_\text{min}$ sa première coordonnée non nulle. Alors le discriminant de Weyl est $D^M(x) = \mathcal{D}(x)_\text{min}$.

  \begin{itemize}
    \item Un sous-ensemble $C_S \subset F_S^{d+1} \setminus \{0\}$ est dit admissible si pour tout $X \in F^{d+1} \cap (C_S \times (\mathfrak{o}^S)^{d+1})$, on a $|X_\text{min}|_v = 1$ pour toute place $v \notin S$.
    \item Un sous-ensemble $\Delta_S \in M(F_S)$ est dit admissible si $\mathcal{D}(\Delta_S) \subset F_S^{d+1}$ l'est.
    \item Un sous-ensemble $\Delta \in M(\A)$ est dit $S$-admissible s'il existe $C_S \subset F_S^{d+1}$ admissible tel que $\mathcal{D}(\Delta) \subset C_S \times (\mathfrak{o}^S)^{d+1}$.
  \end{itemize}
\end{definition}

En particulier, on peut parler de la $S$-admissibilité d'un élément ou d'une classe de conjugaison dans $M(F)$ ou dans $M(\A)$. Étant donné un sous-ensemble compact $\Delta$ de $M(\A)$, on peut toujours agrandir $S$ de sorte que $\Delta$ est $S$-admissible.

Nous utiliserons souvent le lemme suivant dû à Kottwitz.
\begin{proposition}\label{prop:Kottwitz}
  Fixons $K^S = \prod_{v \notin S} K_v$ où $K_v$ est un sous-groupe hyperspécial de $M(F_v)$ pour tout $v \notin S$. Soit $\sigma \in M(F)$ semi-simple. Si $\sigma$ est $S$-admissible et $\sigma^S \in K^S$, alors pour tout $v \notin S$, on a
  \begin{itemize}
    \item $K_v \cap M_\sigma(F_v)$ est un sous-groupe hyperspécial de $M_\sigma(F_v)$;
    \item soit $y \in M(\overline{F_v})$ tel que $y^{-1} \sigma y \in K_v$, alors il existe $y_1 \in M_\sigma(\overline{F_v})$ et $k \in K_v$ tels que $y = y_1 k$.
  \end{itemize}
\end{proposition}
\begin{proof}
  La première assertion résulte de \cite[7.1]{Ko86}. Quant à la deuxième assertion, ledit lemme de Kottwitz fournit une décomposition $y=y_1 k$ avec $y_1 \in M^\sigma(\overline{F_v})$ et $k \in K_v$. Si $M_\text{der}$ est simplement connexe alors $M^\sigma = M_\sigma$ et cela achève la démonstration. Le cas général en résulte à l'aide d'une $z$-extension non ramifiée de $M \times_F F_v$ (cf. \cite[p.386]{Ko86}).
\end{proof}

\subsection{Transport de structure}\index[iFT1]{transport de structure}
On se donne les objets suivants
\begin{itemize}
  \item $S$ un ensemble fini de places de $F$ tel que $S \supset V_\infty$ et $K_v$ est hyperspécial pour tout $v \notin S$;
  \item $M, M' \in \mathcal{L}^G(M_0)$;
  \item $\sigma \in M(F)$ semi-simple tel que $\sigma^S \in K^S$ et que $\sigma$ est $S$-admissible;
  \item idem pour $\sigma' \in M'(F)$;
  \item $\bomega$: caractère automorphe de $M_\sigma(\A)$, trivial sur $K_{\sigma}^S := K^S \cap M_\sigma(\A)$;
  \item $\bomega'$: caractère automorphe de $M'_{\sigma'}(\A)$, trivial sur $K_{\sigma'}^S := K^S \cap M'_{\sigma'}(\A)$.
\end{itemize}

Quitte à conjuguer $\sigma, \sigma'$ et à agrandir $S$, on peut supposer de plus que:
\begin{itemize}
  \item il existe un sous-groupe de Lévi standard $M_1$ tel que $\sigma \in M_1(F)$ mais $\sigma$ n'appartient à aucun sous-groupe de Lévi propre de $M_1$, cela entraîne que $M_{1,\sigma}$ est un sous-groupe de Lévi minimal de $M_\sigma$;
  \item idem, il existe un sous-groupe de Lévi standard $M'_1$ vérifiant ladite condition avec $\sigma', M'$ au lieu de $\sigma, M$.
\end{itemize}

Écrivons $K_\sigma^S = \prod_{v \notin S} K_{\sigma,v}$ et $K_{\sigma'}^S = \prod_{v \notin S} K_{\sigma',v}$. Le lemme de Kottwitz \ref{prop:Kottwitz} affirme que pour tout $v \notin S$, $K_{\sigma,v}$ est un sous-groupe hyperspécial de $M_\sigma(F_v)$; c'est aussi clair que $K_{\sigma,v}$ est en bonne position relativement à $M_{1,\sigma}$. Idem pour $K_{\sigma',v} \subset M'_\sigma(F_v)$ et  $M'_{1,\sigma'}$.

Vu \ref{prop:coef-dependance}, à ces données sont associés les coefficients du développement géométrique fin
$$ a^{M_\sigma,\bomega}(S,\cdot), a^{M'_{\sigma'}, \bomega'}(S,\cdot). $$

On se propose de comparer ces coefficients. Définissons d'abord le transporteur
$$ \mathcal{T}(\sigma,\sigma') := \{y \in G : y\sigma y^{-1} = \sigma', \; y M y^{-1} = M' \}. $$
C'est une sous-variété de $G$ définie sur $F$ sur laquelle $M_\sigma$ (resp. $M'_{\sigma'}$) opère à droite (resp. à gauche) par multiplication.

\begin{hypothesis}\label{hyp:transporteur}\index[iFT1]{$\Omega, \Omega_v$}
  Supposons que
  \begin{itemize}
    \item $\mathcal{T}(\sigma,\sigma')(F) \neq \emptyset$;
    \item pour toute place $v$ de $F$, on fixe une application $\Omega_v: \mathcal{T}(\sigma,\sigma')(F_v) \to \C^\times$, telle que 
    \begin{itemize}
      \item $\Omega_v(x'yx) = \bomega'(x') \Omega_v(y) \bomega(x)$ pour tout $x' \in M'_{\sigma'}(F_v)$, $x \in M_\sigma(F_v)$ et $y \in \mathcal{T}(\sigma,\sigma')(F_v)$;
      \item si $v \notin S$, alors $\Omega_v(x_v)=1$ pour tout $x_v \in K_v \cap \mathcal{T}(\sigma,\sigma')(F_v)$.
    \end{itemize}
  \end{itemize}

  Soit $V \subset V_F$ (éventuellement infini). Par abus de notation, on note $\Omega: \mathcal{T}(\sigma,\sigma')(F_V) \to \C^\times$ l'application donnée par $\prod_{v \in V} \Omega_v$, ce qui est bien définie d'après la dernière condition. En particulier, on sait définir $\Omega: \mathcal{T}(\sigma,\sigma')(\A) \to \C^\times$. On demande de plus que
  \begin{itemize}
    \item $\Omega|_{\mathcal{T}(\sigma,\sigma')F)}=1$.
  \end{itemize}
\end{hypothesis}

\begin{example}\label{ex:caractere}
  Supposons que $\mathcal{T}(\sigma,\sigma')(F) \neq \emptyset$ et qu'il existe un caractère automorphe $\bar{\bomega}: G(\A) \to \C^\times$ qui est trivial sur $K^S$, tel que $\bomega = \bar{\bomega}|_{M_\sigma(\A)}$, $\bomega' = \bar{\bomega}|_{M'_{\sigma'}(\A)}$. Prenons $\Omega_v := \bar{\bomega}|_{\mathcal{T}(\sigma,\sigma')(F_v)}$ pour tout $v$. Alors les conditions sur $\Omega$ sont satisfaites.
\end{example}

Le résultat suivant dit qu'un élément dans $\mathcal{T}(\sigma,\sigma')$ transporte les caractères automorphes.

\begin{lemma}\label{prop:transport-omega}
  Si \ref{hyp:transporteur} est vérifiée, alors pour tout $y \in \mathcal{T}(\sigma,\sigma')(\A)$, on a
  $$ \bomega'(yxy^{-1}) = \bomega(x), \quad x \in M_\sigma(\A). $$
\end{lemma}
\begin{proof}
  Pour $x \in M_\sigma(\A)$, on a $yxy^{-1} \in M'_{\sigma'}(\A)$. Donc
  $$ \Omega(yxy^{-1}y) = \bomega'(yxy^{-1}) \Omega(y). $$
  Or c'est aussi égal à $\Omega(yx) = \Omega(y) \bomega(x)$, d'où l'assertion.
\end{proof}

\begin{proposition}\label{prop:transport-structure}
  Supposons que \ref{hyp:transporteur} est vérifiée. Soit $y \in \mathcal{T}(\sigma,\sigma')(F)$, alors pour tout $\dot{u} \in \dot{\Gamma}_{\mathrm{unip}}(M_\sigma(F),S)^\bomega$, on a
  $$ a^{M_\sigma,\bomega}(S,\dot{u}) = \Omega(y^S)^{-1} a^{M'_{\sigma'}, \bomega'}(S, y\dot{u}y^{-1}). $$
\end{proposition}
\begin{proof}
  On peut translater $y$ par $M'_{\sigma'}(F)$ à gauche, donc on se ramène au cas où $y M_{1,\sigma} y^{-1} = M'_{1,\sigma}$. Par  transport de structure, on a
  $$ a^{M'_{\sigma'},\bomega'}(S,y\dot{u}y^{-1}) = a^{M_\sigma,\bomega}(S,\dot{u}; y^{-1} (K^S \cap M'_{\sigma'}(F^S))y). $$

  Posons
  \begin{align*}
    K_1 & := K^S \cap M_\sigma(F^S),\\
    K_2 & := y^{-1}(K^S \cap M'_{\sigma'}(F^S))y.
  \end{align*}
  Alors $K_1, K_2$ sont des sous-groupes compacts maximaux de $M_\sigma(F^S)$ en bonne position relativement à $M_{1,\sigma}$. On doit prouver que
  \begin{gather}\label{eqn:transport-but-1}
    a^{M_\sigma,\bomega}(S,\dot{u};K_2) = \Omega(y^S) a^{M_\sigma,\bomega}(S,\dot{u};K_1)
  \end{gather}

  On prend $S_+ \supset S$ assez grand de sorte que $K_{2,v} = K_{1,v}$ et $y \in K_v$ pour $v \notin S_+$. On applique \ref{prop:S_+-S} avec le Lévi minimal $M_{1,\sigma}$ et $S_+ \supset S$. Ainsi, on est ramené à prouver
  \begin{gather}\label{eqn:transport-but-2}
    r^{M_\sigma, \bomega}_{L,K_2}(\dot{t}) = \Omega(y^S) r^{M_\sigma, \bomega}_{L,K_1}(\dot{t}), \qquad \forall L \in \mathcal{L}^{M_\sigma}(M_{1,\sigma}), \dot{t} \in \dot{\Gamma}(M_\sigma(F^S_{S_+}))^\bomega .
  \end{gather}

  Pour tout $v$, $W_0^G$ est représenté par des éléments de $K_v$. Donc il existe $k_1 \in K^S$ et $m \in M(F^S)$ tels que $y^S=k_1 m$. Alors $m \sigma^S m^{-1} = k_1^{-1} (\sigma')^S k_1 \in K^S \cap M(F^S)$. D'après \ref{prop:Kottwitz} appliqué à $M$ et $K^S \cap M(F^S)$, il existe $k_2 \in K^S \cap M(F^S)$ et $m_\sigma \in M_\sigma(F^S)$ tels que $m=k_1 m_\sigma$. Posons $k := k_1 k_2$, on a
  \begin{gather*}
    y^S = k m_\sigma, \\
    K_2 = y^{-1} K^S y \cap M_\sigma(F^S) = m_\sigma^{-1} K_1 m_\sigma.
  \end{gather*}

  On applique \ref{prop:conjugaison-K-dans-M_0} avec $H=M_\sigma$, $L=M_{1,\sigma}$ et les sous-groupes ouverts compacts maximaux $K_1$, $K_2$. Pour $v \in S_+ \setminus S$, il existe donc $x_v \in M_{1,\sigma}(F_v)$ tel que $K_{2,v} = x_v^{-1} K_{1,v} x_v$; on prend $x_v=1$ pour $v \notin S_+$. Posons $x := (x_v)_{v \notin S} \in M_{1,\sigma}(F^S)$, alors $K_2 = x^{-1} K_1 x$. Or on a aussi $K_2 = m_\sigma^{-1} K_1 m_\sigma$. Parce que $K_1$ est hyperspécial, on a $m_\sigma \in K_1 x$ quitte à multiplier $x$ par un élément de $Z_{M_{1,\sigma}}(F^S)$. Alors $y^S \in K^S x$, d'où $\Omega(y^S) = \bomega(x)$.

  La relation cherchée \eqref{eqn:transport-but-2} résulte du transport du structure:
  $$ r^{M_\sigma, \bomega}_{L,K_2}(\dot{t}) =  r^{M_\sigma, \bomega}_{L, x^{-1} K_1 x}(\dot{t}) = r^{M_\sigma, \bomega}_{L,K_1}(x\dot{t}x^{-1}) = \bomega(x) r^{M_\sigma, \bomega}_{L,K_1}(\dot{t}). $$
\end{proof}

\section{La formule des traces pour les revêtements}\label{sec:formule-traces-revetement}
La plupart de cette section est une paraphrase des travaux fondamentaux d'Arthur \cite{Ar78, Ar80}.

Soit $F$ un corps de nombres, $G$ un $F$-groupe réductif connexe et un revêtement à $m$ feuillets
$$\xymatrix{
  1 \ar[r] & \bmu_m \ar[r] & \tilde{G} \ar[r]^{\rev} & G(\A) \ar[r] & 1 \\
  & & & G(F) \ar@{_{(}->}[lu] \ar@{_{(}->}[u] &
}$$
où nous supprimons le nom de l'immersion $G(F) \hookrightarrow \tilde{G}$; par abus de notation, nous regardons $G(F)$ comme un sous-groupe discret de $\tilde{G}^1$.

On considéra deux cas.
\begin{enumerate}
  \item \textit{Le cas global}: fixons les objets suivants
    \begin{itemize}
      \item $M_0$: un sous-groupe de Lévi minimal de $G$;
      \item $P_0 \in \mathcal{P}(M_0)$;
      \item $K = \prod_v K_v$: un sous-groupe compact maximal de $G(\A)$ en bonne position relativement à $M_0$, tel que pour tout $v \notin V_\text{ram}$, $K_v$ est le sous-groupe compact maximal fixé dans (G\ref{enu:G1}), qui s'identifie à un sous-groupe de $\tilde{G}_v$.
    \end{itemize}
    Si $P$ est un sous-groupe parabolique semi-standard, on note $P = M_P U_P$ sa décomposition de Lévi canonique. Comme d'habitude, si $H$ est un sous-groupe de $G(\A)$, on note $\tilde{H}$ son image réciproque par $\rev$.
  \item \textit{Le cas local}: fixons un ensemble fini $S$ de places de $F$. Considérons le revêtement à $m$ feuillets $\rev_S : \tilde{G}_S \to G(F_S)$. Fixons
    \begin{itemize}
      \item $M_0$ comme précédemment;
      \item $K_S = \prod_{v \in S} K_v$ un sous-groupe compact maximal de $G(F_S)$ en bonne position relativement à $M_0$.
    \end{itemize}
    Si $H$ est un sous-groupe de $G(F_S)$, notons $\tilde{H}_S$ son image réciproque par $\rev_S$.
\end{enumerate}

Soit $M \in \mathcal{L}(M_0)$, on note $K^M$ (resp. $K_S^M$) son intersection avec $M(\A)$ (resp. avec $M(F_S)$). La convention de mesures \S\ref{sec:mesure} s'applique ici. Rappelons en particulier que, soit $\phi$ est une fonction intégrable sur $G(F) \backslash G(\A)^1$, alors $\int_{G(F) \backslash \tilde{G}^1} (\phi \circ \rev) \dd \tilde{x}$ est égale à $\int_{G(F) \backslash G(\A)^1} \phi \dd x$ selon notre convention.

\subsection{La formule des traces grossière}\label{sec:formule-grossiere-revetement}
Plaçons-nous dans le cadre global.

\paragraph{Le noyau tronqué}
Définissons la représentation régulière $(R, L^2(G(F)\backslash \tilde{G}^1))$ par
$$ (R(\tilde{y})\phi)(\tilde{x}) = \phi(\tilde{x}\tilde{y}) $$
pour tout $\phi \in L^2(G(F)\backslash \tilde{G}^1)$, $\tilde{x},\tilde{y} \in G(F) \backslash\tilde{G}^1$. On peut d'ailleurs la regarder comme une représentation sur $L^2(G(F)A_{G,\infty}\backslash \tilde{G})$. Fixons $f \in C_c^\infty(\tilde{G}^1)$. Alors l'opérateur $R(f)$ a pour noyau

$$ K_f(\tilde{x},\tilde{y}) = \sum_{\gamma \in G(F)} f(\tilde{x}^{-1} \gamma \tilde{y}). $$

L'indice $f$ sera supprimé dans la suite. À cause du procédé de troncature, la formule des traces grossière dépendra des choix de $M_0$, $P_0$ et $K$.

Pour tout $P \in \mathcal{F}(M_0)$, posons
$$ K_P(\tilde{x},\tilde{y}) := \int_{U_P(\A)} \sum_{\gamma \in M_P(F)} f(\tilde{x}^{-1}\gamma u \tilde{y}) \dd u $$
avec $\tilde{x}, \tilde{y} \in U_P(\A) M_P(F) \backslash \tilde{G}$. On constate que pour tout $\tilde{x} \in \tilde{G}^1$, $K_P(\tilde{x},\tilde{x})$ ne dépend que de $x := \rev(\tilde{x}) \in G(\A)^1$; on l'écrit aussi $K_P(x,x)$.

Pour $T \in \mathfrak{a}_0^+$ qui est suffisamment régulier, définissons le noyau tronqué
$$ k^T(x) := \sum_{P \supset P_0} (-1)^{\dim A_P/A_G} \sum_{\delta \in P(F)\backslash G(F)} K_P(\delta x,\delta x) \hat{\tau}_P(H_P(\delta x) - T), $$
où $P$ décrit les sous-groupes paraboliques standards.

Il sera démontré que $k^T$ est intégrable sur $G(F) \backslash G(\A)^1$, et on en obtiendra les développements géométrique et spectral. Or c'est déjà clair que l'aspect ``géométrique'' de la troncature, c'est-à-dire celui qui s'agit des classes de conjugaison rationnelles et des objets dans \S\ref{sec:combinatoire}, est identique à celui de la formule des traces grossière de $G$: le revêtement n'y intervient pas.

\paragraph{Le $\mathfrak{o}$-développement}
Rappelons que dans \S\ref{sec:trace-omega-geometrique} l'on a défini la notion de $\mathcal{O}$-équivalence des éléments dans $G(F)$. Adoptons les mêmes notations ici. Pour tout $M \in \mathcal{L}(M_0)$, on a défini une application canonique $\mathcal{O}^M \to \mathcal{O}^G$ à fibres finies.

Pour tout $\mathfrak{o} \in \mathcal{O}^G$ et $P \in \mathcal{F}(M_0)$, posons comme précédemment
\begin{align*}
  K_{P,\mathfrak{o}}(\tilde{x},\tilde{y}) & := \sum_{\gamma \in M_P(F) \cap \mathfrak{o}}\; \int_{U_P(\A)} f(\tilde{x}^{-1} \gamma u \tilde{y}) \dd u,\\
  k^T_\mathfrak{o}(x) & := \sum_{P \subset P_0} (-1)^{\dim A_P/A_G} \sum_{\gamma \in P(F) \backslash G(F)} K_{P,\mathfrak{o}}(\delta x,\delta x) \hat{\tau}_P(H_P(\delta x)-T),
\end{align*}
alors $K_P = \sum_{\mathfrak{o}} K_{P,\mathfrak{o}}$ et $k^T = \sum_{\mathfrak{o}} k_\mathfrak{o}$.

\begin{theorem}[cf. {\cite[7.1]{Ar78}}]
  Si $T \in \mathfrak{a}_0^+$ est suffisamment régulier, alors
  $$ J^T(f) := \sum_{\mathfrak{o} \in \mathcal{O}^G} \int_{G(F) \backslash G(\A)^1} k^T_\mathfrak{o}(x) \dd x $$
  converge absolument.
\end{theorem}
\begin{proof}
  Comme nous avons remarqué, le revêtement n'intervient pas dans la troncature, et les arguments  de \cite{Ar78} marchent de la même manière.
\end{proof}

C'est donc loisible de poser $J^T_\mathfrak{o}(f) := \int_{G(F) \backslash G(\A)^1} k^T_\mathfrak{o}(x) \dd x$ pourvu que $T \in \mathfrak{a}_0^+$ est suffisamment régulier.

\paragraph{Le $\chi$-développement}
Soit $f \in L^2(G(F) \backslash \tilde{G}^1)$, le terme constant de $f$ le long d'un sous-groupe parabolique $P$ est défini par
$$ f_P: \tilde{x} \mapsto \int_{U_P(F) \backslash U_P(\A)} f(u \tilde{x}) \dd u $$
pour presque tout $\tilde{x} \in U_P(\A)G(F) \backslash \widetilde{G}^1$, à l'aide du scindage unipotent. On dit que $f$ est cuspidale si $f_P=0$ pour tout $P \subsetneq G$. Cela permet de définir l'espace des fonctions $L^2$ cuspidales $L^2_{\text{cusp}}(G(F) \backslash \tilde{G}^1)$. D'après le théorème de Gelfand et Piatetski-Shapiro, $L^2_\text{cusp}(G(F) \backslash \tilde{G}^1)$ est inclus dans $L^2_\text{disc}(G(F) \backslash \tilde{G}^1)$, la sous-représentation maximale de $L^2(G(F) \backslash \tilde{G}^1)$ qui se décompose discrètement. Posons $\Pi_\text{cusp}(\tilde{G}^1)$ l'ensemble des classes d'équivalence de représentations irréductibles intervenant dans $L^2_\text{cusp}(G(F) \backslash \tilde{G}^1)$. La théorie de décomposition spectrale est parallèle à celle pour les groupes réductifs et est bien établie dans \cite{MW94}.

Le groupe central $\bmu_m$ agit par translation sur de telles représentations. Notons la partie $\xi$-équivariante par $\Pi_{\text{cusp},\xi}(\tilde{G}^1)$ pour tout $\xi \in \widehat{\bmu_m}$. En particulier, on peut parler de la partie spécifique $\Pi_{\text{cusp},-}(\tilde{G}^1)$.

Le groupe de Weyl $W_0^G$ agit sur les paires $(M,\sigma)$ où $M \in \mathcal{L}(M_0)$ et $\sigma$ est une représentation automorphe cuspidale de $M(\A)^1$. L'ensemble des orbites $[M,\sigma]$ ainsi obtenues est noté $\mathfrak{X}$, ou $\mathfrak{X}^{\tilde{G}}$ si une confusion sur $\tilde{G}$ est à craindre.


Soit $M$ un sous-groupe de Lévi de $G$, on a une application canonique $\mathfrak{X}^{\tilde{M}} \to \mathfrak{X}^{\tilde{G}}$ donnée par $[M_1, \sigma_1]^M \mapsto [M_1, \sigma_1]^G$ (où $M_1 \in \mathcal{L}^M(M_0)$). Elle est à fibres finies.

Soient $P=M_P U_P \in \mathcal{F}(M_0)$, $\lambda \in \mathfrak{a}_{M,\C}^*$. Notons $R_{\widetilde{M_P},\text{disc}}$ la représentation régulière sur $L^2_\text{disc}(M_P(F) \backslash \widetilde{M_P}^1)$. Puisque $\widetilde{M_P} = \widetilde{M_P}^1 \times A_{M_P,\infty}$, on peut la regarder comme une représentation de $\widetilde{M_P}$. Notons $R_{\widetilde{M_P},\text{disc},\lambda}$ la représentation de $\widetilde{M_P}$ définie par
$$ R_{\widetilde{M_P},\text{disc},\lambda}(\tilde{m}) = R_{\widetilde{M_P},\text{disc}}(\tilde{m}) e^{\angles{\lambda,H_M(m)}}. $$

Notons $\mathcal{I}_{\tilde{P}}(\lambda)$ l'induite parabolique normalisée de $R_{\widetilde{M_P},\text{disc},\lambda}$. Définissons l'espace hilbertien $\mathcal{H}_{\tilde{P}}$ des fonctions mesurables $\phi: U_P(\A) M_P(F) A_{M,\infty} \backslash \tilde{G} \to \C$ telles que
\begin{itemize}
  \item pour presque tout $\tilde{x} \in \tilde{G}$, $\phi_{\tilde{x}}: \tilde{m} \mapsto \phi(\tilde{m}\tilde{x})$ appartient à $L^2_\text{disc}(M_P(F)\backslash \widetilde{M_P}^1)$,
  \item $\displaystyle \|\phi\|_2^2 = \iint_{\tilde{K} \times (M_P(F) \backslash \widetilde{M_P}^1)} |\phi(\tilde{m}\tilde{k})|^2 \dd\tilde{m}\dd\tilde{k} < +\infty.$
\end{itemize}
Alors $\mathcal{H}_{\tilde{P}}$ peut être vu comme l'espace sous-jacent de $\mathcal{I}_{\tilde{P}}(\lambda)$ muni de l'action
$$ (\mathcal{I}_{\tilde{P}}(\lambda, \tilde{y})\phi)(\tilde{x}) = \phi(\tilde{x}\tilde{y}) e^{\angles{\lambda+\rho_P, H_P(xy)-H_P(x)}}. $$

Remarquons que $\mathcal{H}_{\tilde{P}}$ et la restriction de $\mathcal{I}_{\tilde{P}}(\lambda)$ à $\tilde{K}$ ne dépendent pas de $\lambda$. Notons $\mathcal{H}_{\tilde{P},\text{cusp}}$ le sous-espace des fonctions $\phi \in \mathcal{H}_{\tilde{P}}$ telle que $\phi_{\tilde{x}} \in L^2_\text{cusp}(M_P(F) \backslash \widetilde{M_P}^1)$ pour presque tout $\tilde{x}$. Notons $\mathcal{H}_{\tilde{P}}^0$ le sous-espace dense de vecteurs $\tilde{K}$-finis et $\mathcal{H}_{\tilde{P},\text{cusp}}^0$ son intersection avec $\mathcal{H}_{\tilde{P},\text{cusp}}$.

Passons en revue la théorie de la décomposition spectrale.
\begin{enumerate}
  \item On sait construire un sous-espace fermé invariant $L^2_\chi(G(F) \backslash \tilde{G})$, engendré par les séries d'Eisenstein attachées à $\chi$ (\cite[II.2.4]{MW94}). Il existe une décomposition orthogonale
  $$ L^2(G(F) \backslash \tilde{G}^1) = \bigoplus_{\chi \in \mathfrak{X}^{\tilde{G}}} L^2_\chi(G(F) \backslash \tilde{G}^1). $$
  \item Soient $Q \in \mathcal{P}(M_P)$, $w \in W(\mathfrak{a}_P, \mathfrak{a}_{Q})$. On définit l'opérateur d'entrelacement suivant la recette de Mackey (cf. \cite[II.1.6]{MW94}):
  \begin{gather*}
    M_{Q|P}(w,\lambda): \mathcal{H}_{\tilde{P}} \to \mathcal{H}_{\tilde{Q}} \quad (\lambda \in \mathfrak{a}_{M,\C}^*) \\
    (M_{Q|P}(w,\lambda)\phi)(\tilde{x}) = \int_{U_{Q|P}(\hat{w},\A)} \phi(\hat{w}^{-1}u\tilde{x}) e^{\angles{\lambda+\rho_P, H_P(\hat{w}^{-1}ux)}} e^{\angles{-w\lambda-\rho_Q, H_Q(x)}} \dd u,
  \end{gather*}
  où $\hat{w} \in G(F)$ est un représentant quelconque de $w$ et
  $$ U_{Q|P}(\hat{w},\A) := (U_Q(\A) \cap \hat{w}U_P(\A)\hat{w}^{-1}) \backslash U_Q(\A). $$
  Notons $(\mathfrak{a}_P^*)_+$ le cône dual à ${}^+ \mathfrak{a}_P$, alors cette intégrale est convergente et holomorphe en $\lambda$ si $\mathrm{Re}(\lambda) \in \rho_P + (\mathfrak{a}_P^*)^+$. Elle admet une prolongement méromorphe sur $\mathfrak{a}_{M,\C}^*$.
\end{enumerate}

À l'aide de la décomposition spectrale, on définit les espaces $\mathcal{H}_{\tilde{P},\chi}$, $\mathcal{H}_{\tilde{P},\chi}^0$ et l'opérateur $\mathcal{I}_{\tilde{P},\chi}(\lambda)$ en prenant l'intersection avec $L^2_{\chi}$. On en déduit des décompositions $\mathcal{H}_{\tilde{P}} = \bigoplus_\chi \mathcal{H}_{\tilde{P},\chi}$ et $\mathcal{H}_{\tilde{P}}^0 = \bigoplus_\chi \mathcal{H}_{\tilde{P},\chi}^0$.

Soit $Q$ un sous-groupe parabolique standard de $G$. Posons $n_Q = n^G_Q := \sum_{Q'} |W(\mathfrak{a}_Q, \mathfrak{a}_{Q'})|$, la somme portant sur les sous-groupes paraboliques standards $Q'$. La même définition s'applique aux sous-groupes de Lévi de $G$, d'où les versions relatives $n^P_{P_1} := n_{P_1 \cap M_P}^{M_P}$ pour $P_1 \subset P$ sous-groupes paraboliques standards. Choisissons une base orthonormée $\mathcal{B}_{\tilde{P}_1,\chi}$ pour $\mathcal{H}_{\tilde{P}_1,\chi}$ formée de vecteurs $\tilde{K}$-finis pour tout $\chi$, alors $\mathcal{B}_{\tilde{P}_1} := \bigsqcup_\chi \mathcal{B}_{\tilde{P}_1,\chi}$ est une base orthonormée de $\mathcal{H}_{\tilde{P}_1}$ contenue dans $\mathcal{H}_{\tilde{P}_1}^0$. Posons

\begin{align*}
  K_{P,\chi}(\tilde{x},\tilde{y}) & := \sum_{P_1 \subset P} (n^P_{P_1})^{-1} \int_{i\mathfrak{a}_{P_1}^*} \sum_{\phi \in \mathcal{B}_{\widetilde{P_1},\chi}} E^{\widetilde{P}}_{\widetilde{P_1}}(\tilde{x}, \mathcal{I}_{\widetilde{P_1},\chi}(\lambda, f)\phi, \lambda) \overline{E^{\widetilde{P}}_{\widetilde{P_1}}(\tilde{y},\phi,\lambda)} \dd \lambda, \\
  k_\chi^T(x) & := \sum_{P \supset P_0} (-1)^{\dim A_P/A_G} \sum_{\delta \in P(F)\backslash G(F)} K_{P,\chi}(\delta x,\delta x) \hat{\tau}_P(H_P(\delta x)-T),
\end{align*}
où
$$ E^{\widetilde{P}}_{\widetilde{P_1}}(\tilde{x},\phi, \lambda) = \sum_{\gamma \in P_1(F) \backslash P(F)} \phi(\gamma\tilde{x}) e^{\angles{\lambda+\rho_{P_1}, H_{P_1}(x)}} $$
est la ``série d'Eisenstein''; elle est convergente si $\mathrm{Re}(\lambda) \in \rho_{P_1} + (\mathfrak{a}_{P_1}^*)^+$, et elle admet un prolongement méromorphe à tout $\lambda \in \mathfrak{a}_{M_1,\C}^*$. La définition de $K_{P,\chi}$ ne dépend évidemment pas du choix des bases $\mathcal{B}_{\tilde{P}_1,\chi}$. La décomposition spectrale affirme que $K_P = \sum_\chi K_{P,\chi}$ et $k^T = \sum_\chi k_\chi^T$.

\begin{theorem}[cf. {\cite[2.1]{Ar80}}]
  Si $T \in \mathfrak{a}_0^+$ est suffisamment régulier, alors
  $$ \sum_{\chi \in \mathfrak{X}^{\tilde{G}}} \; \int_{G(F) \backslash G(\A)^1} k^T_\chi(x) \dd x $$
  converge absolument.
\end{theorem}
\begin{proof}
  On peut reprendre \cite{Ar80}. On a déjà remarqué que la troncature est pareille que dans le cas des groupes réductifs. L'aspect spectral des arguments d'Arthur repose sur la théorie de décomposition spectrale, dont la généralisation aux revêtements est faite dans \cite{MW94}.
\end{proof}
C'est donc loisible de poser $J^T_\chi(f) := \int_{G(F) \backslash G(\A)^1} k^T_\chi(x) \dd x$, pour tout $\chi \in \mathfrak{X}^{\tilde{G}}$.

Selon l'action de $\bmu_m$, on a une décomposition $\mathfrak{X} = \bigsqcup_{\xi \in \widehat{\bmu_m}} \mathfrak{X}_\xi$. Fixons $\xi \in \widehat{\bmu_m}$, alors la $\xi$-équivariance de représentations est préservée par induction parabolique. Le résultat suivant en découle.

\begin{proposition}
  Soient $\xi, \eta \in \widehat{\bmu_m}$. Si $\chi \in \mathfrak{X}_\xi$, $f \in C^\infty_{c,\eta}(\tilde{G}^1)$, alors $K_{P,\chi,f}(\tilde{x},\tilde{y}) = 0$ pour tout $\tilde{x},\tilde{y} \in \tilde{G}^1$ sauf si $\xi = \bar{\eta}$. Idem pour $k^T_\chi(f)$ et $J^T_\chi(f)$.
\end{proposition}

Par conséquent, on peut se ramener à l'étude de la partie spécifique du côté spectral de la formule des traces grossière, ie. les termes associés à $\chi \in \mathfrak{X}_-$, appliqués à fonctions tests anti-spécifiques.

\paragraph{Conclusions}
Vu les résultats de convergence géométrique et spectral, on est arrivé à l'identité
$$ J^T(f) = \sum_{\mathfrak{o} \in \mathcal{O}^G} J^T_{\mathfrak{o}}(f) = \sum_{\chi \in \mathfrak{X}^{\tilde{G}}} J^T_\chi(f) $$
pour $T \in \mathfrak{a}_0^+$ suffisamment régulier, où $J^T(f) := \int_{G(F) \backslash \tilde{G}^1} k^T(\tilde{x}) \dd\tilde{x}$. L'intégrabilité résulte des théorèmes que nous venons d'obtenir. Dans la suite, $\mathfrak{o}$ (resp. $\chi$) désigne un élément dans $\mathcal{O}^G$ (resp. $\mathfrak{X}^{\tilde{G}})$ quelconque.

\begin{theorem}
  Supposons que le paramètre de troncature $T \in \mathfrak{a}_0^+$ est suffisamment régulier, alors
  \begin{enumerate}
    \item $J^T(f)$ est un polynôme en $T$ de degré $\leq \dim \mathfrak{a}_0^G$. Idem pour les termes $J^T_\mathfrak{o}(f)$, $J^T_\chi(f)$. Ces distributions se prolongent ainsi en polynômes en tout $T \in \mathfrak{a}_0$ de façon unique;
    \item notons $T_0 \in \mathfrak{a}_0$ le paramètre de troncature canonique défini dans \eqref{eqn:T_0}, alors les distributions $J := J^{T_0}$, $J_\chi := J_\chi^{T_0}$ et $J_\mathfrak{o} := J_\mathfrak{o}^{T_0}$  ne dépendent pas du choix de $P_0 \in \mathcal{P}(M_0)$.
  \end{enumerate}
\end{theorem}
\begin{proof}
  Les démonstrations sont pareilles que celles du cas des groupes réductifs, cf  \ref{prop:dependance-T}, \ref{prop:independance-P_0}.
\end{proof}

La formule des traces grossière s'écrit ainsi\index[iFT1]{$J_{\mathfrak{o}}, J_\chi$}
\begin{gather}\label{eqn:formule-grossiere}
  J(f) = \sum_{\mathfrak{o} \in \mathcal{O}^G} J_{\mathfrak{o}}(f) = \sum_{\chi \in \mathfrak{X}^{\tilde{G}}} J_\chi(f).
\end{gather}

Étudions la non-invariance des distributions. Soient $f \in C_c^\infty(\tilde{G}^1)$ et $y \in G(\A)$. Notons $f^y$ la fonction $\tilde{x} \mapsto f(y\tilde{x}y^{-1})$. Soit $Q \in \mathcal{F}(M_0)$, posons
\begin{align*}
  f_{Q,y}(\tilde{m}) := \delta_Q(x)^{\frac{1}{2}} \iint_{K \times U_Q(\A)} f(k^{-1} \tilde{m}u k) u'_Q(k,y) \dd u \dd k, \quad \tilde{m} \in \widetilde{M_Q}^1 ,
\end{align*}
où $u'_Q(k,y)$ est la fonction définie dans \eqref{eqn:u'_Q(k,y)}. Cela définit une application linéaire $C_c^\infty(\tilde{G}^1) \to C_c^\infty(\widetilde{M_Q}^1)$ qui préserve l'équivariance par $\bmu_m$.

Nous indiquons le groupe en question en exposant dans les notations, eg. $J^{G,T}(f) = J^T(f)$. Si $M \in \mathcal{L}(M_0)$, $\mathfrak{o} \in \mathcal{O}^G$ et $f \in C_c^\infty(\widetilde{M}^1)$, posons
$$ J^{M,T}_{\mathfrak{o}}(f) = \sum_{\substack{\mathfrak{o}' \in \mathcal{O}^M \\ \mathfrak{o}' \mapsto \mathfrak{o}}} J^{M,T}_{\mathfrak{o}'}(f). $$
Idem pour $J^{M,T}_\chi$ et $J^{M,T}$.

\begin{theorem}[cf. {\cite[\S 3]{Ar81}} et \ref{prop:non-invariance-Jomega-Weyl}]\label{prop:non-invariance-grossiere}
  Soit $T \in \mathfrak{a}_0$, alors
  \begin{align*}
    J^T(f^y) & = \sum_{Q \in \mathcal{F}(M_0)} |W_0^{M_Q}| |W_0^G|^{-1} J^{\widetilde{M_Q},T}(f_{Q,y}), \\
    J^T_\mathfrak{o}(f^y) & = \sum_{Q \in \mathcal{F}(M_0)} |W_0^{M_Q}| |W_0^G|^{-1} J^{\widetilde{M_Q},T}_{\mathfrak{o}}(f_{Q,y}),\\
    J^T_\chi(f^y) & = \sum_{Q \in \mathcal{F}(M_0)} |W_0^{M_Q}| |W_0^G|^{-1} J^{\widetilde{M_Q},T}_\chi(f_{Q,y}).
  \end{align*}
  En particulier, on peut mettre $T=T_0$ dans les formules ci-dessus et supprimer les symboles $T$.
\end{theorem}

\subsection{Réduction au cas unipotent}\label{sec:reduction-unip}
Plaçons-nous toujours dans le cas global. Fixons les objets suivants
\begin{itemize}
  \item $\mathfrak{o} \in \mathcal{O}^G$;
  \item $\sigma \in \mathfrak{o}$: semi-simple;
  \item $P_1 \in \mathcal{F}(M_0)$: un sous-groupe parabolique standard tel que $\sigma \in M_{P_1}(F)$, mais $\sigma$ n'appartient à aucun sous-groupe parabolique propre de $M_{P_1}$;
  \item $M_1 := M_{P_1}$;
  \item $K_\sigma = \prod_v K_{\sigma,v}$: un sous-groupe compact maximal de $G_\sigma(\A)$ en bonne position relativement à $M_{1,\sigma}$.
\end{itemize}
Afin d'éviter toute confusion, indiquons l'image de $\sigma$ via $G(F) \hookrightarrow \tilde{G}^1$ par $\tilde{\sigma}$. Posons $\mathfrak{a}^{G_\sigma}_0 := \mathfrak{a}^{G_\sigma}_{M_{1,\sigma}}$, on dispose d'une application linéaire canonique $\mathfrak{a}^G_0 \to \mathfrak{a}^{G_\sigma}_0$. À ces données est associé un paramètre de troncature canonique $T_{0,\sigma} \in \mathfrak{a}_0^{G_\sigma}$ pour $G_\sigma$.

Fixons aussi un ensemble de places $S \supset V_\text{ram}$ et supposons que
\begin{itemize}
  \item $\sigma^S \in K^S$;
  \item $\sigma$ est $S$-admissible;
  \item pour tout $v \notin S$, on a $K_{\sigma,v} = K \cap G_\sigma(F_v)$.
\end{itemize}
Le lemme de Kottwitz \ref{prop:Kottwitz} assure que $K_{\sigma,v}$ est un sous-groupe hyperspécial de $G_\sigma(F_v)$ pour tout $v \notin S$.

\begin{lemma}
  Le caractère $[\cdot,\sigma]: G_\sigma(\A) \to \bmu_m$ est un caractère automorphe. Il est trivial sur $K_\sigma^S$.
\end{lemma}
\begin{proof}
  La continuité de $[\cdot,\sigma]$ est claire. Le reste résulte des scindages de $\rev$ au-dessus de $K^S$ et de $G(F)$.
\end{proof}

Posons
\begin{gather*}
  \iota^G(\sigma) := G_\sigma(F) \backslash G^\sigma(F),\\
  T_1 := T_0 - T_{0,\sigma} \in \mathfrak{a}_0^{G_\sigma},
\end{gather*}
où, avec abus de notation, on a projeté le paramètre de troncature canonique $T_0$ de $G$ via $\mathfrak{a}_0^G \to \mathfrak{a}_0^{G_\sigma}$.

Rappelons une construction d'Arthur \cite[\S 6]{Ar86}. Soit $R \in \mathcal{F}^{G_\sigma}(M_{1,\sigma})$, on note
\begin{align*}
  \mathcal{F}^0_R(M_1) & :=  \{ P \in \mathcal{F}(M_1) : P_\sigma = R, \mathfrak{a}_P = \mathfrak{a}_R \}, \\
  \mathcal{F}_R(M_1) & := \{ P \in \mathcal{F}(M_1) : P_\sigma = R \}.
\end{align*}
Pour tout $z \in G(\A)$ et tout $Q \in \mathcal{F}(M_1)$, posons
$$ v_Q(\lambda,z,T) := e^{\angles{\lambda, -H_Q(z)+T}}, \quad \lambda \in i\mathfrak{a}_Q^*, T \in \mathfrak{a}_Q . $$
On lui associe la fonction $v'_Q(z,T)$ via \eqref{eqn:c'_P}. Si $R \in \mathcal{F}^{G_\sigma}(M_{1,\sigma})$, posons
$$ v'_R(z, T) := \sum_{Q \in \mathcal{F}^0_R(M_1)} v'_Q(z,T). $$

Soit $f \in C_{c,\asp}^\infty(\tilde{G}^1)$. Pour $R$ et $y \in G(\A)$ fixés, on obtient une application linéaire $f \mapsto \Phi_{R,y,T_1} \in C_{c,\asp}^\infty(\widetilde{M_R}^1)$ définie par
$$ \Phi_{R,y,T_1}(\tilde{m}) := \delta_R(m)^{\frac{1}{2}} \iint_{K_\sigma \times U_R(\A)} [k,\sigma] f(y^{-1} \tilde{\sigma} k^{-1} \tilde{m}uky) v'_R(ky,T_1) \dd u \dd k . $$
Alors $f \mapsto \Phi_{R,y,T_1}$ est aussi continue en $y$. Notre définition et celle dans \cite[p.201]{Ar86} ne diffèrent que par le commutateur $[k,\sigma]$, dont la raison d'être sera expliquée dans la procédure de descente.

\begin{theorem}[cf. {\cite[6.2]{Ar86}}]\label{prop:Jo-descente}
  Pour $f \in C_{c,\asp}^\infty(\tilde{G}^1)$, on a
  $$ J_{\mathfrak{o}}(f) = |\iota^G(\sigma)|^{-1} \int_{G_\sigma(\A) \backslash G(\A)} \sum_{R \in \mathcal{F}^{G_\sigma}(M_{1,\sigma})} |W_0^{M_R}| |W_0^{G_\sigma}|^{-1} J^{M_R,[\cdot,\sigma]}_{\mathrm{unip}} (\Phi_{R,y,T_1}) \dd y . $$
\end{theorem}
L'expression est loisible car la distribution $J^{M_R,[\cdot,\sigma]}_{\mathrm{unip}}$ est supportée sur $(M_R)_{\text{unip}}(\A)$, et le scindage unipotent adélique permet de restreindre $\Phi_{R,y,T_1}$ sur ce sous-ensemble fermé de façon canonique.

\begin{proof}
  La preuve est presque identique à celle dans \cite{Ar86} à l'exception du caractère $[\cdot,\sigma]$ qui apparaît dans $J^{M_R,[\cdot,\sigma]}_{\mathrm{unip}}$ et dans notre définition de $\Phi_{R,y,T_1}$. Expliquons-le. Prenons $T \in \mathfrak{a}_0^+$ suffisamment régulier. On part de la formule \cite[3.1]{Ar86}
  \begin{multline}\label{eqn:descente-unipotente-Jo}
    J^T_{\mathfrak{o}}(f) =  |\iota^G(\sigma)|^{-1} \int_{G(F) \backslash G(\A)^1} \sum_{\substack{R \subset G_\sigma \\ \text{parabolique}\\ \text{standard}}} \sum_{\xi \in R(F) \backslash G(F)} \\
    \left( \sum_{u \in (M_R)_{\mathrm{unip}}(F)} \; \int_{U_R(\A)} f(x^{-1} \xi^{-1} \tilde{\sigma}un \xi x) \dd n \right) \cdot \\
    \left( \sum_{P \in \mathcal{F}_R(M_1)} (-1)^{\dim A_P/A_G} \hat{\tau}_P(H_P(\xi x)-Z_P(T-T_0)-T_0)\right) \dd x,
  \end{multline}
  où $Z_P(T-T_0) \in \mathfrak{a}_0$ est défini dans \cite[(3.3)]{Ar86}. La démonstration ne fait intervenir que l'action adjointe de $G(F)$ sur les sous-groupes de Lévi et paraboliques sur $F$, et sur $G(F)$ lui-même, donc est valable sur un revêtement.

  Changeons l'intégrale ci-dessus prise sur
  $$ (\xi,x) \in (R(F) \backslash G(F)) \times (G(F) \backslash G(\A)^1)$$
  à une intégrale prise sur
  $$ (\delta, x, y) \in (R(F) \backslash G_\sigma(F)) \times (G_\sigma(F) \backslash G_\sigma(\A) \cap G(\A)^1) \times (G_\sigma(\A) \backslash G(\A)) $$
  à l'aide du fait $G_\sigma(\A) \backslash G(\A) = G_\sigma(\A) \cap G(\A)^1 \backslash G(\A)^1$. L'intégrale sur $U_R(\A)$ dans \eqref{eqn:descente-unipotente-Jo} devient
  $$ \int_{U_R(\A)} f(y^{-1} x^{-1}\delta^{-1} \tilde{\sigma}un \delta x y) \dd n $$
  en remplaçant l'expression $\xi x$ par $\delta x y$. Or
  $$ y^{-1} x^{-1}\delta^{-1} \tilde{\sigma}un \delta x y = [\sigma, x] y^{-1}\tilde{\sigma}x^{-1}\delta^{-1}un\delta x y , $$
  donc l'intégrale en question vaut
  $$ \int_{U_R(\A)} [x,\sigma] f(y^{-1}\tilde{\sigma} x^{-1} \delta^{-1} un \delta x y) \dd n $$
  car $f$ est anti-spécifique.

  À partir de maintenant, on peut reprendre la démonstration de \cite{Ar86}. L'expression $[x,\sigma]$ se décompose en le caractère $[\cdot,\sigma]$ dans le terme $J_\text{unip}^{M_R, [\cdot,\sigma]}$ et le caractère $[k,\sigma]$ dans la définition de $\Phi_{R,y,T_1}$ via une décomposition d'Iwasawa. Le bilan est l'analogue de \cite[(6.8)]{Ar86}:
  $$ J_{\mathfrak{o}}(f) = |\iota^G(\sigma)|^{-1} \int_{G_\sigma(\A) \backslash G(\A)} \sum_{\substack{Q \in \mathcal{F}^{G_\sigma}(M_{1,\sigma}) \\ Q \supset P_{1,\sigma}}} J^{M_Q,[\cdot,\sigma]}_{\mathrm{unip}} (\Phi_{Q,y,T_1}) \dd y . $$

  La dernière étape est d'enlever la dépendance de $P_{1,\sigma}$. Soit $R \in \mathcal{F}^{G_\sigma}(M_{1,\sigma})$, alors il existe $w \in W_0^{G_\sigma}$ et $Q \supset P_{1,\sigma}$ tel que $R = w^{-1} Q w$. Notons comme d'habitude $\pi: G_{\sigma,\text{SC}} \to G_\sigma$ le revêtement simplement connexe de $G_{\sigma,\text{der}}$ et $K_{\sigma,\text{sc}} := \pi^{-1}(K_\sigma)$. Prenons un représentant $\tilde{w} \in \pi(K_{\sigma,\text{sc}})$ de $w$. Comme $v'_Q(\tilde{w}ky, T_1)=v'_R(ky,T_1)$ (voir \cite[p.201]{Ar86}), on voit que $\Phi_{Q,y,T_1}(\tilde{m})=\Phi_{R,y,T_1}(\tilde{w}^{-1}\tilde{m}\tilde{w})$ pour tout $\tilde{m} \in \widetilde{M_Q}^1$. Vu \ref{prop:transport-w}, on en déduit que
  $$ J^{M_Q, [\cdot,\sigma]}_\text{unip}(\Phi_{Q,y,T_1}) = J^{M_R, [\cdot,\sigma]}_\text{unip}(\Phi_{R,y,T_1}). $$
  Alors un argument similaire à celui de \ref{prop:independance-P_0} permet de conclure.
\end{proof}

\subsection{Intégrales orbitales pondérées anti-spécifiques}
Plaçons-nous dans le cas local. Les conventions de \S\ref{sec:commutateurs} (qui correspond au cas $|S|=1$) se généralisent de façon évidente à ce cadre. En particulier, on peut parler des bons éléments dans $G(F_S)$.

\paragraph{Conventions sur la mesure}
Soit $M \in \mathcal{L}^G(M_0)$. La situation est similaire à celle de \S\ref{sec:int-orb-omega}: on considère les paires $(\mathcal{O},\mu)$, où
\begin{itemize}
  \item $\mathcal{O}$ est une bonne classe de conjugaison dans $\widetilde{M}_S$,
  \item $\mu$ est une mesure de Radon invariante non triviale sur $\mathcal{O}$.
\end{itemize}

Nous préférons une construction directe comme suit. Prenons $\tilde{\gamma} \in \mathcal{O}$, alors la mesure invariante $\mu$ est déterminée par le choix d'une mesure de Haar sur $M_\gamma(F_S)$, et réciproquement.

Le groupe $M(F_S)$ opère sur ces paires par conjugaison. On écrit $(\mathcal{O}, \mu) \sim (\mathcal{O}', \mu')$ si $\mathcal{O}=\mathcal{O}'$. Notons
\begin{align*}
  \dot{\Gamma}(\widetilde{M}_S) & := \{(\mathcal{O},\mu) \},\\
  \Gamma(\widetilde{M}_S) & := \{(\mathcal{O}, \mu)\}/\sim .
\end{align*}
Alors $\dot{\Gamma}(\widetilde{M}_S) \to \Gamma(\widetilde{M}_S)$ est un $\R_{>0}$-torseur.\index[iFT1]{$\dot{\Gamma}(\widetilde{M}_S), \Gamma(\widetilde{M}_S)$}

Nous utilisons les symboles pointés pour désigner un élément dans $\dot{\Gamma}(\widetilde{M}_S)$, eg. $\dot{\tilde{\gamma}}$; la classe de conjugaison sous-jacente est notée $\Supp(\dot{\tilde{\gamma}})$.

Une paire $\dot{\tilde{\gamma}} = (\mathcal{O},\mu)$ donne naissance à l'intégrale orbitale
\begin{gather}
  J_{\tilde{M}}(\dot{\tilde{\gamma}}, f) := |D^M(\gamma)|^{\frac{1}{2}} \int_{\mathcal{O}} f \dd \mu, \quad f \in C_{c,\asp}^\infty(\widetilde{M}_S)
\end{gather}
avec $\gamma \in \rev(\mathcal{O})$ quelconque. Si l'on fixe $\tilde{\gamma} \in \mathcal{O}$ et une mesure de Haar sur $M_\gamma(F_S)$, alors on a tout simplement
$$ J_{\tilde{M}}(\dot{\tilde{\gamma}}, f) =  |D^M(\gamma)|^{\frac{1}{2}} \int_{M_\gamma(F_S) \backslash M(F_S)} f(x^{-1} \tilde{\gamma} x) \dd x. $$
Pour montrer qu'elle converge, il suffit de remplacer $f$ par $|f|$. On obtient ainsi une intégrale orbitale pour un groupe réductif, dont la convergence est connue d'après Rao \cite{Rao72}. Cela permet d'immerger $\dot{\Gamma}(\widetilde{M}_S)$ dans l'espace de distributions spécifiques. On a
$$  J_{\tilde{M}}(y\dot{\tilde{\gamma}}y^{-1}, f) = J_{\tilde{M}}(\dot{\tilde{\gamma}}, f), \quad \text{pour tout } y \in M(F_S).  $$

Indiquons quelques opérations élémentaires.
\begin{enumerate}
  \item Le sous-groupe central $\bmu_m$ opère de manière évidente sur $\dot{\Gamma}(\widetilde{M}_S)$.
  \item Si $\dot{\tilde{\gamma}}$ est tel que $M_\gamma = G_\gamma$, alors il s'identifie canoniquement à un élément de $\dot{\Gamma}(\widetilde{G_S}) \sqcup \{0\}$, cf. \eqref{eqn:de-M-a-G}. Plus précisément, $\dot{\tilde{\gamma}}$ s'envoie à $0$ si et seulement si $\gamma$ n'est pas bon dans $G(F_S)$ .
  \item Un élément $\dot{\tilde{\gamma}} \in \dot{\Gamma}(\widetilde{M}_S)$ admet une unique décomposition de Jordan
$$ \dot{\tilde{\gamma}} = \tilde{\sigma} \dot{u}, \quad \tilde{\sigma} \in \widetilde{M}_S, \; \dot{u} \in \dot{\Gamma}(M_\sigma(F_S))^{[\cdot,\sigma]} $$
    correspondant à la décomposition de Jordan $\tilde{\gamma} = \tilde{\sigma} u$.
  \item On a défini un sous-groupe ouvert fermé $A_M(F_S)^\dagger \subset A_M(F_S)$ dans \ref{prop:bon-perturbation}. Supposons fixé un voisinage ouvert $\mathcal{U}$ de $1$ dans  $A_M(F_S)^\dagger$ muni d'un scindage $\mathcal{U} \hookrightarrow \rev^{-1}(\mathcal{U})$ de $\rev$ qui envoie $1$ à $1$. Soit $a \in \mathcal{U}$, on peut définir l'application $\dot{\tilde{\gamma}} \mapsto a\dot{\tilde{\gamma}}$ de translation par $a$ à l'aide de ce scindage.
  \item On peut définir l'induction de classes unipotentes de $M$ à $G$ à la \ref{prop:induction-omega}. Vu le scindage unipotent, on peut supprimer le $\sim$ et le noter $\dot{\gamma} \mapsto \dot{\gamma}^G$.
\end{enumerate}
Nous laissons ces yogas de mesures invariantes au lecteur.

\paragraph{Intégrales orbitales pondérées}
Commençons par définir l'intégrale orbitale pondérée anti-spécifiques pour les éléments $\dot{\tilde{\gamma}} \in \dot{\Gamma}(\widetilde{M}_S)$ tels que $M_\gamma=G_\gamma$. Rappelons que cette donnée équivaut au choix d'une mesure de Haar sur $M_\gamma(F_S) = G_\gamma(F_S)$.

\begin{definition}\label{def:int-ponderee-nr}\index[iFT1]{$J_{\tilde{M}}(\dot{\tilde{\gamma}}, \cdot)$}
  Supposons que $M_\gamma=G_\gamma$. Pour tout $f \in C_{c,\asp}^\infty(\widetilde{G_S})$, posons
  $$ J_{\tilde{M}}(\dot{\tilde{\gamma}}, f) := |D^M(\gamma)|^{\frac{1}{2}} \int_{G_\gamma(F_S) \backslash G(F_S)} f(x^{-1} \tilde{\gamma} x) v_M(x) \dd x . $$
\end{definition}

Si $\gamma$ n'est pas bon dans $G(F_S)$, alors $J_{\tilde{M}}(\dot{\tilde{\gamma}}, f)=0$.

Revenons au cas général. L'intégrale orbitale pondérée est définie comme suit.
\begin{theorem}[cf. {\cite[5.2]{Ar88LB}}]\label{prop:integrale-ponderee-bien-definie}
  Pour tout $f \in C_{c,\asp}^\infty(\widetilde{G_S})$, la limite
  $$ J_{\tilde{M}}(\dot{\tilde{\gamma}}, f) := \lim_{\substack{a \to 1 \\ a \in A_{M,\text{reg}}(F_S)^\dagger}} \sum_{L \in \mathcal{L}(M)} r^L_M(\gamma,a) J_{\tilde{L}}(a\dot{\tilde{\gamma}}, f) $$
  existe, où les $a$ dans la limite sont supposés en position générale de sorte que $M_{a\gamma}=G_{a\gamma}$.

  Si $M_\gamma = G_\gamma$, elle coïncide avec la définition \ref{def:int-ponderee-nr}. On a
  $$ \forall y \in M(F_S), \quad J_{\tilde{M}}(y \dot{\tilde{\gamma}} y^{-1}, f) = J_{\tilde{M}}(\dot{\tilde{\gamma}}, f^y) =  J_{\tilde{M}}(\dot{\tilde{\gamma}}, f). $$
\end{theorem}

Pour que l'expression $a\dot{\tilde{\gamma}}$ dans l'énoncé soit loisible, il faut fixer un voisinage ouvert $\mathcal{U}$ de $1$ dans $A_M(F_S)^\dagger$, un scindage de $\rev$ au-dessus de $\mathcal{U}$ qui envoie $1$ à $1$, et supposer que $a \in \mathcal{U}$; comme on ne regarde que la limite $a \to 1$, ces choix n'importent pas.

Notons tout d'abord que $J_{\tilde{M}}(\dot{\tilde{\gamma}}, f)=0$ si $\gamma$ n'est pas bon dans $G(F_S)$. Les définitions entraînent aussi que $J_{\tilde{M}}(\noyau\dot{\tilde{\gamma}}, f) = \noyau^{-1} J_{\tilde{M}}(\dot{\tilde{\gamma}}, f)$ pour tout $\noyau \in \bmu_m$.

\begin{proof}
  On peut supposer $\gamma$ bon dans $M(F_S)$ d'après l'observation précédente, alors $a\gamma$ l'est aussi pour $a \in A_M(F_S)^\dagger$. On se ramène à l'étude de
  $$ |D^G(a\gamma)|^{\frac{1}{2}} \int_{M_\gamma(F_S) \backslash G(F_S)} f(x'^{-1} a\tilde{\gamma} x') \left( \sum_{L \in \mathcal{L}(M)} r^L_M(\gamma,a) v_L(x') \right) \dd x' $$
  lorsque $a \to 1$. Soit $\tilde{\gamma}=\tilde{\sigma} u$ la décomposition de Jordan. On décompose la variable $x'=mxy$ avec $m \in M_\gamma(F_S) \backslash M_\sigma(F_S)$, $x \in M_\sigma(F_S) \backslash G_\sigma(F_S)$, $y \in G_\sigma(F_S) \backslash G(F_S)$. Alors l'intégrale ci-dessus s'écrit
  $$
    |D^G(a\gamma)|^{\frac{1}{2}} \iiint f(y^{-1} x^{-1} m^{-1} a\tilde{\sigma} u mxy) \left( \sum_{L \in \mathcal{L}(M)} r^L_M(\gamma,a) v_L(xy) \right) \dd m \dd x \dd y.
  $$
  Comme $a$ est central dans $\tilde{M}$, on a
  \begin{align*}
    y^{-1} x^{-1} m^{-1} a\tilde{\sigma} u mxy & = [\sigma,m] [\sigma,x] y^{-1}\tilde{\sigma} x^{-1} a m^{-1}umxy, \\
    f(y^{-1} x^{-1} m^{-1} a\tilde{\sigma} u mxy) & = [m,\sigma] [x,\sigma] f(y^{-1}\tilde{\sigma} x^{-1} a m^{-1}umxy)
  \end{align*}
  par l'anti-spécificité de $f$. Posons $g_{\tilde{\sigma}, a,x,y}(\tilde{m}) := [x,\sigma] f(y^{-1}\tilde{\sigma} x^{-1}a\tilde{m} x y)$, alors l'intégrale sur $x$ devient
  $$ \int_{M_\gamma(F_S) \backslash M_\sigma(F_S)} [m,\sigma]g_{\tilde{\sigma}, a,x,y}(m^{-1} u m ) \dd m , $$
  une intégrale orbitale unipotente avec le caractère $[\cdot,\sigma]$. Elle est bien définie car $\gamma$ est bon.

  Ce que l'on a fait est la première étape de la démonstration dans \cite[\S 6]{Ar88LB}; en fait c'est la seule part où intervient le revêtement. Après l'argument de descente ci-dessus, le revêtement disparaît au prix de rajouter le caractère $[\cdot,\sigma]$. Le reste de la démonstration marche de la même façon qu'en \cite{Ar88LB} si l'on remplace les mesures $\dd x$ par $[x,\sigma]\dd x$ et $\dd m$ par $[m,\sigma] \dd m$. Cela n'affecte pas les estimations dans \cite{Ar88LB}; en particulier, la clef \cite[6.1]{Ar88LB} et sa démonstration, qui repose sur une technique géométrique de Langlands, restent les mêmes. Cela permet de reprendre les arguments d'Arthur.
\end{proof}

\paragraph{Le cas non ramifié}
Fixons $G$ et $M$ comme précédemment. Supposons que $S$ consiste en places non archimédiennes et supposons $K_v$ hyperspécial de pour chaque $v \in S$. Notons $f_{K_S} = \prod_{v \in S} f_{K_v}$, où $f_{K_v}$ est l'unité de l'algèbre de Hecke sphérique anti-spécifique en $v$ (voir \S\ref{sec:algebre-Hecke}).

\begin{definition}\index[iFT1]{$r_{\tilde{M},K_S}(\dot{\tilde{\gamma}})$}
  Les intégrales orbitales pondérées anti-spécifiques non ramifiées sont définies par
  $$ r_{\tilde{M},K_S}(\dot{\tilde{\gamma}}) = r^{\tilde{G}}_{\tilde{M},K_S}(\dot{\tilde{\gamma}}) := J_{\tilde{M}}(\dot{\tilde{\gamma}}, f_{K_S}), \quad \dot{\tilde{\gamma}} \in \dot{\Gamma}(\widetilde{M_S}) .$$
\end{definition}

Lorsqu'il n'y a pas de confusion sur $K_S$, on l'écrit aussi $r^{\tilde{G}}_{\tilde{M}}(\dot{\tilde{\gamma}})$.

\paragraph{Descente semi-simple}
Fixons $\dot{\tilde{\gamma}} \in \dot{\Gamma}(\widetilde{M}_S)$ comme précédemment avec la décomposition de Jordan $\dot{\tilde{\gamma}} = \tilde{\sigma} \dot{u}$. Supposons que
\begin{itemize}
  \item $\tilde{\sigma} \in M(F)$;
  \item $\sigma$ est $F$-elliptique dans $M$.
\end{itemize}

Prenons un sous-groupe compact maximal $K_\sigma$ de $G_\sigma(F_S)$ en bonne position relativement à $M_\sigma$. Il faut rappeler une construction parallèle à celle pour \ref{prop:Jo-descente} (cf. \cite[\S 8]{Ar88LB}). Soit $\dot{\tilde{\gamma}} = \tilde{\sigma} \dot{u}$ la décomposition de Jordan. Soit $R \in \mathcal{F}^{G_\sigma}(M_\sigma)$. Prenons aussi $T \in \mathfrak{a}_M$ et posons
$$ v_P(\lambda,z,T) := v_P(\lambda, z) e^{\angles{\lambda,T}}, \quad P \in \mathcal{P}(M), $$
où $v_P(\lambda,z)$ est la $(G,M)$-famille définissant le poids. Les fonctions $v_P(\lambda,z,T)$ forment encore une $(G,M)$-famille.

Avec le formalisme de \ref{prop:Jo-descente}, posons
\begin{align*}
  v'_R(z,T) & := \sum_{Q \in \mathcal{F}_R^0(M)} v'_Q(z,T), \quad z \in G(F_S); \\
  \Phi_{R,y,T}(\tilde{m}) & := \delta_R(m)^{\frac{1}{2}} \iint_{K_\sigma \times U_R(F_S)} [k,\sigma] f(y^{-1}
\tilde{\sigma} k^{-1} \tilde{m}uky) v'_R(ky,T) \dd u \dd k \\
  & \text{où} \quad  \tilde{m} \in (\widetilde{M_R})_S, \; y \in G(F_S).
\end{align*}

\begin{proposition}[cf. {\cite[8.6]{Ar88LB}}]\label{prop:int-ponderee-descente}
  On a
  $$ J_{\tilde{M}}(\dot{\tilde{\gamma}}, f) = |D^G(\sigma)|^{\frac{1}{2}} \int_{G_\sigma(F_S) \backslash G(F_S)} \; \sum_{R \in \mathcal{F}^{G_\sigma}(M_\sigma)} J^{M_R, [\cdot,\sigma]}_{M_\sigma}(\dot{u}, \Phi_{R,y,T}) \dd y, $$
\end{proposition}
\begin{proof}
  On reprend l'argument dans \cite{Ar88LB}. Ici le caractère $[\cdot,\sigma]$ intervient pour la même raison que dans la démonstration de \ref{prop:integrale-ponderee-bien-definie}.
\end{proof}

\subsection{Comportement des intégrales orbitales pondérées anti-spécifiques}
Les résultats ci-dessous sont parallèles à ceux de \S\ref{sec:comport-int-ponderee-omega}. Vu \ref{prop:integrale-ponderee-bien-definie}, leurs démonstrations sont aussi similaires et nous ne les répéterons pas.

\paragraph{$(\tilde{M},\sigma)$-équivalence}\index[iFT1]{$(\tilde{M},\sigma)$-équivalence}
Soient $\sigma \in M(F_S)_\text{ss}$ et $\Sigma \subset \sigma M_\sigma(F_S)$ un ouvert invariant par $M_\sigma(F_S)$. Notons
$$\dot{\Gamma}(\tilde{\Sigma}) := \{ \dot{\gamma} \in \dot{\Gamma}(\widetilde{M}_S) : \Supp(\dot{\tilde{\gamma}}) \subset \rev^{-1}(\Sigma) \}. $$
Supposons désormais que l'adhérence de $\Sigma$ dans $\sigma M_\sigma(F_S)$ contient un voisinage invariant de $\sigma$. On dit que deux fonctions $\phi_1$, $\phi_2$ sur $\dot{\Gamma}(\tilde{\Sigma})$ sont $(\tilde{M},\sigma)$-équivalentes s'il existe $f \in C_{c,\asp}^\infty(\widetilde{M}_S)$ et un voisinage $U$ de $\sigma$ dans $M(F_S)$ tels que
$$ (\phi_1-\phi_2)(\dot{\tilde{\gamma}}) = J^{\tilde{M}}_{\tilde{M}}(\dot{\tilde{\gamma}},f) $$
pour tout $\dot{\tilde{\gamma}} \in \dot{\Gamma}(\tilde{\Sigma})$ tel que $\Supp(\dot{\tilde{\gamma}}) \subset \rev^{-1}(U)$. Si cette condition est vérifiée, on écrit
$$ \phi_1 \stackrel{(\tilde{M},\sigma)}{\sim} \phi_2 .$$

\begin{proposition}
  Si $M_\sigma=G_\sigma$, alors pour tout $f \in C_{c,\asp}^\infty(\widetilde{M}_S)$ on a
  $$ J_{\tilde{M}}(\dot{\tilde{\gamma}},f) \stackrel{(\tilde{M},\sigma)}{\sim} 0 $$
  pour tout $\dot{\tilde{\gamma}} \in \dot{\Gamma}(\widetilde{M}_S)$ assez proche de $\tilde{\sigma}$ modulo conjugaison.
\end{proposition}

\paragraph{Formules de descente}\index[iFT1]{$f_Q$}
Fixons $\dot{\tilde{\gamma}} \in \dot{\Gamma}(\widetilde{M}_S)$. Soit $Q = L U_Q \in \mathcal{F}(M_0)$. Définissons
$$ f_Q(\tilde{m}) := \delta_Q(m)^{\frac{1}{2}} \int_{K_S} \int_{U_Q(F_S)} f(k^{-1}\tilde{m}uk) \dd u \dd k, \quad m \in L(F_S). $$
Ceci fournit une application linéaire $C_{c,\asp}^\infty(\widetilde{G_S}) \to C_{c,\asp}^\infty(\widetilde{L}_S)$.

\begin{proposition}\label{prop:J-asp-descente}
  Supposons que $\gamma \in M_\mathrm{unip}(F_S)$. Avec les choix effectués dans \ref{prop:J-descente}, on a
  $$ J_{\widetilde{L_1}}(\dot{\gamma}^{L_1},f) = \sum_{L \in \mathcal{L}(M)} d^G_M(L_1,L) J^{\tilde{L}}_{\tilde{M}}(\dot{\gamma}, f_{Q_L}). $$
\end{proposition}

\begin{proposition}\label{prop:J-asp-deploiement}
  Supposons $S=S_1 \sqcup S_2$. Soient $\dot{\tilde{\gamma}} = \dot{\tilde{\gamma}}_1 \dot{\tilde{\gamma}}_2$, et $f = f_1 f_2 \in C_{c,\asp}^\infty(\widetilde{G_S})$. En conservant le formalisme de \ref{prop:J-deploiement}, on a
  $$ J_{\tilde{M}}(\dot{\tilde{\gamma}},f) = \sum_{L_1, L_2 \in \mathcal{L}(M)} d^G_M(L_1, L_2) J^{\widetilde{L_1}}_{\tilde{M}}(\dot{\tilde{\gamma}}_1, f_{Q_1}) J^{\widetilde{L_2}}_{\tilde{M}}(\dot{\tilde{\gamma}}_2, f_{Q_2}). $$
\end{proposition}
Remarquons que la décomposition $\dot{\tilde{\gamma}} = \dot{\tilde{\gamma}}_1 \dot{\tilde{\gamma}}_2$ n'est unique qu'à l'action près du groupe $\{(\noyau,\noyau^{-1}) : \noyau \in \bmu_m\}$.

\paragraph{Non-invariance}
Soient $Q = L U_Q \in \mathcal{F}(M_0)$ et $y \in G(F_S)$. On définit
$$ f_{Q,y}(\tilde{m}) = \delta_Q(m)^{\frac{1}{2}} \iint_{K_S \times U_Q(F_S)} f(k^{-1}\tilde{m}uk) u'_Q(k,y) \dd u \dd k, \quad \tilde{m} \in \widetilde{L}, $$
où $u'_Q$ est la fonction définie en \ref{prop:non-invariance-omega-pondere}. Ceci fournit une application linéaire $C_{c,\asp}^\infty(\widetilde{G_S}) \to C_{c,\asp}^\infty(\widetilde{L}_S)$. Rappelons aussi que $f^y$ est la fonction $\tilde{x} \mapsto f(y\tilde{x}y^{-1})$.

\begin{proposition}\label{prop:non-invariance-int-ponderee}
  On a
  $$ J_{\tilde{M}}(\dot{\tilde{\gamma}}, f^y) = \sum_{Q \in \mathcal{F}(M)}  J_{\tilde{M}}^{\widetilde{M_Q}}(\dot{\tilde{\gamma}}, f_{Q,y}). $$
\end{proposition}

\begin{proposition}\label{prop:J-asp-transport}
  Soit $y \in G(F_S)$ tel que $yMy^{-1} \in \mathcal{L}(M_0)$ , alors
  $$ J_{y\tilde{M}y^{-1}}(y\dot{\tilde{\gamma}}y^{-1}, f) = J_{y\tilde{M}y^{-1}}(\dot{\tilde{\gamma}}, f). $$
\end{proposition}
\begin{proof}
  C'est le même argument qu'en \ref{prop:J-transport}.
\end{proof}

\begin{remark}
  L'intégrale orbitale pondérée satisfait à d'autres propriétés importantes, par exemple le développements en germes au voisinage d'un élément $\tilde{\sigma}$ d'image semi-simple (pas forcément bon). Les détails se trouvent dans \cite{Ar88LB}. Par ailleurs, le cas $G=M=\GL(n)$ est déjà étudié dans \cite{KF86}.
\end{remark}

\subsection{Développement géométrique fin}
\paragraph{Passage à une situation locale}
Revenons au cas global. Nous considérons $S$ un ensemble fini de places de $F$ tel que $S \supset V_\text{ram}$.

Rappelons une définition d'Arthur dans \cite[\S 8]{Ar86}.
\begin{definition}\label{def:(M,S)-equiv}
  On dit que deux éléments $\gamma_1, \gamma_2 \in M(F)$ avec décompositions de Jordan $\gamma_i = \sigma_i u_i$ (où $i=1,2$) sont $(M,S)$-équivalents s'il existe $x \in M(F)$ et $y \in M_{\sigma_2}(F_S)$ tels que
  \begin{itemize}
    \item $x^{-1}\sigma_1 x = \sigma_2$,
    \item $y^{-1}x^{-1} u_1 xy = u_2$.
  \end{itemize}
\end{definition}
On vérifie que c'est une relation d'équivalence. Une classe de $\mathcal{O}^M$-équivalence se découpe en un nombre fini de classes de $(M,S)$-équivalence.

\begin{definition}\label{def:K-bon}\index[iFT1]{$(M(F))_{\tilde{M},S}^{K,\text{bon}}$}
  Posons
  $$
  \begin{array}{ll}
  (M(F)) & := \{\text{classes de conjugaison dans $M(F)$}\}, \\
  (M(F))_{M,S} & := \{\text{classes de $(M,S)$-équivalence dans $M(F)$}\}, \\
  (M(F))_{M,S}^K & := \left\{
    \begin{array}{ll}
      c \in (M(F))_{M,S}: & \exists \gamma = \sigma u \in c \;\text{(décomposition de Jordan),  où} \\
      & \sigma \text{ est $S$-admissible},\\
      & \sigma^S \in K^S.
    \end{array}
  \right\},\\
  (M(F))_{\tilde{M},S}^{K,\text{bon}} & := \{c \in (M(F))_{M,S}^K : c \text{ est bon dans } M(F_S) \}.
  \end{array}
  $$
  Soient $c \in (M(F))_{M,S}^K$ et $\gamma = \sigma u \in c$ un représentant vérifiant les conditions dans cette définition. On dit que $\gamma$ est un représentant admissible de $c$. Par abus de notation, on désignera une classe dans $(M(F))_{M,S}^K$ par un représentant admissible.
\end{definition}

Notons
$$ K_M = \prod_v K_{M,v} := K \cap M(\A). $$

Afin de compléter le raffinement géométrique d'Arthur, il faut compléter les fonctions test locales en celles globales comme dans \S\ref{sec:dev-fin-unip}. Comme $S \supset V_\text{ram}$, on sait définir $f_{K_{M,v}} \in C_{c,\asp}^\infty(\tilde{M}_v)$ l'unité de $\mathcal{H}_{\asp}(\tilde{M}_v /\!/ K_{M,v})$, pour tout $v \notin S$. D'où un homomorphisme injectif
\begin{equation}\begin{split}\label{eqn:completion-fonction}
  C_{c,\asp}^\infty(\tilde{M}_S) & \to C_{c,\asp}^\infty(\tilde{M}) \\
  f_S & \mapsto f_S \cdot \prod_{v \notin S} f_{K_{M,v}}.
\end{split}\end{equation}

Il faut aussi extraire la part locale d'une classe dans $(M(F))_{M,S}^{K,\text{bon}}$. Précisons. Soit $c \in (M(F))_{\tilde{M},S}^K$ avec un représentant admissible $\gamma = \sigma u$. Le scindage au-dessus de $K_M^S$ fournit une identification
\begin{align}\label{eqn:scindage-KS}
  \rev^{-1}(M(F_S) \times K_M^S) = \tilde{M}_S \times K_M^S .
\end{align}
Notons $(\widetilde{\sigma_S}, \sigma^S) \in \tilde{M}_S \times K_M^S$ l'élément auquel $\sigma$ s'identifie. Posons $\widetilde{\gamma_S} = \widetilde{\sigma_S} u_S$, où $u_S$ est relevé à l'aide du scindage unipotent. On vérifie que $u_S$ est $[\cdot,\sigma]$-bon dans $M_\sigma(F_S)$ si et seulement si $c \in (M(F))_{\tilde{M},S}^{K,\text{bon}}$.

\begin{definition}\label{def:passage-local}\index[iFT1]{$\gamma \leadsto \widetilde{\gamma_S}$}
  Soient $\gamma \in (M(F))_{\tilde{M},S}^{K,\text{bon}}$, $\widetilde{\gamma_S} \in \tilde{M}_S$. On écrit
  $$ \gamma \leadsto \widetilde{\gamma_S} $$
  si $\gamma$ est un représentant admissible qui donne $\widetilde{\gamma_S}$ comme ci-dessus. Il faut prendre garde qu'en général, les représentants admissibles d'une classe de $(M,S)$-équivalence ne sont pas conjugués par $M(F) \cap (M(F_S) \times K_M^S)$, par conséquent $\leadsto$ n'a aucune raison d'être une application bien définie de $(M(F))_{\tilde{M},S}^{K,\text{bon}}$ dans $\Gamma(\tilde{M}_S)$!
\end{definition}

\paragraph{Exprimer $J_\mathfrak{o}$ par intégrales orbitales pondérées anti-spécifiques}
Nous nous proposons de remonter la formule descendue pour $J_\mathfrak{o}$ dans \ref{prop:Jo-descente} en termes des intégrales orbitales pondérées anti-spécifiques; nous n'en donnerons qu'une esquisse car les arguments complets se trouvent dans \cite[\S 8]{Ar86}.

Plaçons-nous dans le cas global. Fixons $\mathfrak{o} \in \mathcal{O}^G$ et prenons les objets $\sigma$, $M_1$, $K_\sigma$, $T_1$ et $S \subset V_F$ comme dans \S\ref{sec:reduction-unip}. Quitte à agrandir $S$, on peut aussi supposer que
\begin{itemize}
  \item pour tout $v \notin S$ et tout $y_v \in G(F_v)$, on a $(y_v^{-1} \sigma G_{\sigma,\text{unip}}(F_v) y_v) \cap \sigma K_v \neq \emptyset$ seulement si $y_v \in G_\sigma(F_v) K_v$ (voir \cite[6.1]{Ar86}).
\end{itemize}

Alors \ref{prop:Jo-descente} se lit
$$ J_\mathfrak{o}(f) = |\iota^G(\sigma)|^{-1} \int_{G_\sigma(\A) \backslash G(\A)} \sum_{R \in \mathcal{F}^{G_\sigma}(M_{1,\sigma})} |W_0^{M_R}| |W_0^{G_\sigma}|^{-1} J_{\text{unip}}^{M_R, [\cdot,\sigma]} (\Phi_{R,y,T_1}) \dd y. $$

L'expression dans l'intégrale est nulle sauf si
$$ y = y_S y', \qquad y_S \in G_\sigma(F_S) \backslash G(F_S), \; y' \in \prod_{v \notin S} G_\sigma(F_v) \backslash G_\sigma(F_v)K_v .$$

Pour un tel $y$, on a l'identification $\Phi_{R,y,T_1} = \Phi _{R,y_S,T_1} \in C_c^\infty(M_R(F_S)^1)$ via $C_c^\infty(M_R(F_S)^1) \hookrightarrow C_c^\infty(M_R(\A)^1)$, où $\Phi _{R,y_S,T_1}$ est la fonction associée à $f_S$ par \ref{prop:int-ponderee-descente},  car $v'_R(ky, T_1)=v'_R(k_S y_S, T_1)$ et $[K^S,\sigma]=1$; toutes ces assertions sont démontrées dans \cite[\S 7]{Ar86} sauf la dernière, qui résulte du fait que $\sigma^S \in K^S$ et $S \supset V_\text{ram}$.

Donc $J_{\mathfrak{o}}(f)$ est égal à
$$ |\iota^G(\sigma)|^{-1} \int_{G_\sigma(F_S) \backslash G(F_S)} \sum_{R \in \mathcal{F}^{G_\sigma}(M_{1,\sigma})} |W_0^{M_R}| |W_0^{G_\sigma}|^{-1} J_{\text{unip}}^{M_R, [\cdot,\sigma]} (\Phi_{R,y_S,T_1}) \dd y_S. $$

Posons $\mathcal{L}^\sigma := \mathcal{L}^{G_\sigma}(M_{1,\sigma})$. D'après \ref{prop:int-unip-developpement},
\begin{multline*}
  |W_0^{M_R}| |W_0^{G_\sigma}|^{-1} J_{\text{unip}}^{M_R, [\cdot,\sigma]} (\Phi_{R,y_S,T_1}) \\
  = \sum_{\substack{L \in \mathcal{L}^\sigma \\ L \subset M_R}} |W_0^L| |W_0^{G_\sigma}|^{-1} \sum_{u \in \Gamma_\text{unip}(L(F),S)^{[\cdot,\sigma]}} a^{L,[\cdot,\sigma]}(S,\dot{u}) J_L^{M_R, [\cdot,\sigma]}(\dot{u}, \Phi_{R,y_S,T_1}).
\end{multline*}

Posons $\mathcal{L}^0_\sigma(M_1) := \{ M \in \mathcal{L}(M_1) : A_M = A_{M_\sigma}\}$, on a une bijection
\begin{align*}
  \mathcal{L}^0_\sigma(M_1) & \to \mathcal{L}^\sigma \\
  M & \mapsto M_\sigma.
\end{align*}

Il en résulte que $J_\mathfrak{o}(f)$ est égale à
\begin{multline*}
  |\iota^G(\sigma)|^{-1} \sum_{M \in \mathcal{L}^0_\sigma(M_1)} \sum_{u \in \Gamma_\text{unip}(M_\sigma(F),S)^{[\cdot,\sigma]} } |W_0^{M_\sigma}| |W_0^{G_\sigma}|^{-1} a^{M_\sigma,[\cdot,\sigma]}(S,\dot{u}) \\
  \int_{G_\sigma(F_S) \backslash G(F_S)} \left( \sum_{R \in \mathcal{F}^{G_\sigma}(M_\sigma)} J^{M_R,[\cdot,\sigma]}_{M_\sigma}(\dot{u}, \Phi_{R,y_S,T_1}) \right) \dd y.
\end{multline*}

Rappelons qu'un relèvement $\widetilde{\sigma_S} \in \tilde{M}_S$ de $\sigma$ est défini dans \ref{def:passage-local}. Soit $u \in \Gamma_\text{unip}(M_\sigma(F),S)^{[\cdot,\sigma]}$, alors $\widetilde{\sigma_S} u_S$ est bon dans $\tilde{M}_S$. D'autre part, $\prod_{v \in S} |D^G(\sigma)|_v = |D^G(\sigma)| = 1$ d'après la $S$-admissibilité de $\sigma$. En multipliant la formule ci-dessus par $\prod_{v \in S} |D^G(\sigma)|_v$, on peut appliquer \ref{prop:int-ponderee-descente} et obtient ainsi

\begin{lemma}[cf. {\cite[7.1]{Ar86}}]\label{prop:developpement-geometrique-preraffine}
  Soit $f \in C_{c,\asp}^\infty(\tilde{G}_S)$, alors
  \begin{align*}
    J_\mathfrak{o}(f) & = |\iota^G(\sigma)|^{-1} \sum_{M \in \mathcal{L}^0_\sigma(M_1)} |W_0^{M_\sigma}| |W_0^{G_\sigma}|^{-1} \\
    & \sum_{u \in \Gamma_\mathrm{unip}(M_\sigma(F),S)^{[\cdot,\sigma]}} a^{M_\sigma,[\cdot,\sigma]}(S,\dot{u}) J_{\tilde{M}}(\tilde{\sigma}\dot{u}, f).
  \end{align*}
\end{lemma}

\paragraph{Les coefficients}

\begin{definition}\label{def:coef-gamma-cond}
  Soit $\gamma \in M(F)$ avec décomposition de Jordan $\gamma=\sigma u$. Supposons que $\gamma$ est un représentant admissible d'une classe dans $(M(F))_{\tilde{M},S}^{K,\text{bon}}$ (voir \ref{def:K-bon}), alors un élément $\widetilde{\gamma_S} = \widetilde{\sigma_S} u \in \tilde{M}_S$ lui est associé selon la construction de \ref{def:passage-local}. Prenons un élément $\dot{\widetilde{\gamma_S}} \in \dot{\Gamma}(\tilde{M}_S)$ supporté sur la classe de conjugaison contenant $\widetilde{\gamma_S}$ avec la décomposition de Jordan $\dot{\widetilde{\gamma_S}}=\widetilde{\sigma_S} \dot{u}$.

  Posons\index[iFT1]{$\epsilon^M(\sigma)$}\index[iFT1]{$a^{\tilde{M}}(S, \dot{\widetilde{\gamma_S}})$}
  \begin{gather*}
    \epsilon^M(\sigma) := \begin{cases} 1, & \text{si $\sigma$ est $F$-elliptique dans $M$},\\ 0,& \text{sinon}; \end{cases},  \\
    \text{Stab}(\sigma,u) := \{ t \in \iota^M(\sigma) : t u t^{-1} \stackrel{\text{conj}}{\sim} u \text{ dans } M_\sigma(F_S) \}, \\
    a^{\tilde{M}}(S, \dot{\widetilde{\gamma_S}}) := \epsilon^M(\sigma) |\text{Stab}(\sigma,u)|^{-1} a^{M_\sigma, [\cdot,\sigma]}(S,\dot{u}).
  \end{gather*}
\end{definition}
C'est la définition d'Arthur (voir \cite[(2.4)]{Ar02}) lorsque le revêtement est trivial. On vérifie que $a^{\tilde{M}}(S, \dot{\widetilde{\gamma_S}}) J_{\tilde{M}}(\dot{\widetilde{\gamma_S}},f)$ ne dépend pas des choix des mesures, et il est invariant par conjugaison par $M_\sigma(F_S)$ d'après \ref{prop:int-unip-developpement}.

\begin{remark}
  D'après \ref{prop:coef-dependance}, les coefficients $a^{\tilde{M}}(S,\cdot)$ sont déterminés par les données $\rev: \tilde{M} \to M(\A)$, $S$, et le sous-groupe compact maximal $K^S$ de $M(F^S)$ tels que
  \begin{itemize}
    \item il existe un sous-groupe de Lévi minimal $M_0$ de $M$, défini sur $F$, qui est en bonne position relativement à $K_M^S$;
    \item $S \supset V_\text{ram}$.
  \end{itemize}
\end{remark}

\begin{lemma}
  Soient $M, M' \in \mathcal{L}(M_0)$ et $\gamma \in M(F)$ (resp. $\gamma' \in M'(F)$ ) avec la décomposition de Jordan $\gamma=\sigma u$ (resp. $\gamma' = \sigma' u' \in M'(F)$) satisfaisant aux conditions dans \ref{def:coef-gamma-cond}. S'il existe $y \in G(F)$ tel que $yMy^{-1} = M'$, $y \gamma y^{-1}= \gamma'$, alors
  \begin{align*}
    a^{\tilde{M}}(S, \dot{\widetilde{\gamma_S}}) J_{\tilde{M}}(\dot{\widetilde{\gamma_S}}, f) &= a^{\tilde{M}'}(S, \dot{\widetilde{\gamma_S}}') J_{\tilde{M}'}(\dot{\widetilde{\gamma_S}}', f), \\
    a^{M_\sigma, [\cdot,\sigma]}(S, \dot{u}) J_{\tilde{M}}(\dot{\widetilde{\gamma_S}}, f) &= a^{M'_{\sigma'}, [\cdot,\sigma']}(S, \dot{u}') J_{\tilde{M}'}(\dot{\widetilde{\gamma_S}}', f)
  \end{align*}
  pour tout $f$.

  En particulier, le produit $a^{\tilde{M}}(S, \dot{\widetilde{\gamma_S}}) J_{\tilde{M}}(\dot{\widetilde{\gamma_S}},f)$ ne dépend que de $f$ et de la classe de $\gamma$ dans $(M(F))_{M,S}^{K,\mathrm{bon}}$.
\end{lemma}
\begin{proof}
  Sélectionnons les mesures de sorte que
  \begin{gather}\label{eqn:transport-tilde-Jordan}
    \dot{\widetilde{\gamma_S}}' = \widetilde{\sigma'_S} \cdot (y \dot{u} y^{-1}).
  \end{gather}

  On a les caractères automorphes $\bomega = [\cdot,\sigma]$ sur $M_\sigma(\A)$ et $\bomega' = [\cdot,\sigma']$ sur $M'_{\sigma'}(\A)$. On va appliquer \ref{prop:transport-structure}. Pour satisfaire aux hypothèses dans \ref{hyp:transporteur}, il reste à construire les fonctions $\Omega_v$ sur $\mathcal{T}(\sigma,\sigma')(F_v)$ pour toute place $v$.

  Sélectionnons $\tilde{\sigma}_v \in \rev^{-1}(\sigma_v)$ pour toute place $v$ de sorte que $\tilde{\sigma}_v \in K_v$ si $v \notin S$ et $[\tilde{\sigma}_v]_v = \tilde{\sigma}$; alors on a aussi $\prod_{v \in S} \tilde{\sigma}_v = \tilde{\sigma}_S$. Idem pour $\tilde{\sigma}'_v$.

  Fixons une place $v$. Définissons $\Omega_v: \mathcal{T}(\sigma,\sigma')(F_v) \to \bmu_m$ par la formule
  $$ z\tilde{\sigma}_v z^{-1} = \Omega_v(z) \tilde{\sigma}'_v. $$
  Alors $\Omega_v(x'zx) = \bomega'(x')\Omega_v(z)\bomega(x)$ pour tout $x \in M_\sigma(F_v)$ et tout $x' \in M'_{\sigma'}(F_v)$. Si $v \notin S$ et $z \in \mathcal{T}(\sigma,\sigma')(F_v) \cap K_v$, alors $\Omega_v(z)=1$ grâce au fait que $\sigma^S,{\sigma'}^S \in K^S$ et au scindage de $\rev$ au-dessus de $K_v$. Comme $\rev$ se scinde au-dessus de $G(F)$, on a aussi $\Omega|_{\mathcal{T}(\sigma,\sigma')(F)}=1$.

  Alors \ref{prop:transport-structure} implique
  \begin{align*}
    a^{M_\sigma,[\cdot,\sigma]}(S, \dot{u}) &= \Omega(y^S)^{-1} a^{M'_{\sigma'},[\cdot,\sigma']}(S, \dot{u}'), \\
    \text{d'où} \quad a^{\tilde{M}}(S, \dot{\widetilde{\gamma_S}}) &= \Omega(y^S)^{-1} a^{\tilde{M}'}(S, \dot{\widetilde{\gamma_S}}').
  \end{align*}

  D'autre part, \ref{prop:J-asp-transport} implique
  $$ J_{\tilde{M}'}(y \dot{\widetilde{\gamma_S}} y^{-1}, f) = J_{\tilde{M}}(\dot{\widetilde{\gamma_S}},f). $$
  Comme $y \widetilde{\sigma_S} y^{-1} = \Omega(y_S) \widetilde{\sigma'_S}$, on déduit de \eqref{eqn:transport-tilde-Jordan} que
  $$ J_{\tilde{M}}(\dot{\widetilde{\gamma_S}}, f) = \Omega(y_S)^{-1} J_{\tilde{M}'}(\dot{\widetilde{\gamma_S}}',f). $$

  Or $\Omega(y_S)\Omega(y^S)=\Omega(y)=1$ car $y \in G(F)$, cela conclut la démonstration pour le premier énoncé. On a déjà remarqué que $\widetilde{\gamma_S} \mapsto a^{\tilde{M}}(S, \dot{\widetilde{\gamma_S}}) J_{\tilde{M}}(\dot{\widetilde{\gamma_S}},f)$ est invariant par conjugaison par $M_\sigma(F_S)$, donc le dernier énoncé résulte de la définition de $(M,S)$-équivalence.
\end{proof}

Notons $(M(F) \cap \mathfrak{o})_{M,S}^{K,\mathrm{bon}}$ le sous-ensemble des classes dans $(M(F))_{M,S}^{K,\mathrm{bon}}$ qui rencontrent $\mathfrak{o}$.

\begin{theorem}[cf. {\cite[8.1]{Ar86}}]\label{prop:developpement-geometrique-o}
  Avec les mêmes notations, on a
  $$ J_\mathfrak{o}(f) = \sum_{M \in \mathcal{L}(M_0)} |W_0^M| |W_0^G|^{-1} \sum_{\substack{\gamma \in (M(F) \cap \mathfrak{o})_{M,S}^{K,\mathrm{bon}} \\ \gamma \leadsto \widetilde{\gamma_S} }} a^{\tilde{M}}(S, \dot{\widetilde{\gamma_S}}) J_{\tilde{M}}(\dot{\widetilde{\gamma_S}}, f). $$
  La somme ne porte que sur un nombre fini de classes.
\end{theorem}

Ici l'expression signifie que, pour chaque classe dans $(M(F) \cap \mathfrak{o})_{M,S}^{K,\mathrm{bon}}$, on en prend un représentant admissible $\gamma$ quelconque, puis un $\widetilde{\gamma_S}$ quelconque tel que $\gamma \leadsto \widetilde{\gamma_S}$ via la correspondance définie dans \ref{def:passage-local}. Le produit $a^{\tilde{M}}(S, \dot{\widetilde{\gamma_S}}) J_{\tilde{M}}(\dot{\widetilde{\gamma_S}}, f)$ est bien défini grâce au lemme précédent.

\begin{proof}
  Reprenons les notations de \ref{prop:developpement-geometrique-preraffine}. Le groupe $\iota^M(\sigma)$ opère sur $\Gamma_\text{unip}(M_\sigma(F),S)^{[\cdot,\sigma]}$, et le groupe d'isotropie d'une classe $u$ est $\text{Stab}(\sigma,u)$. Vu le lemme précédent et \ref{prop:developpement-geometrique-preraffine}, $J_\mathfrak{o}(f)$ est égal à
  \begin{gather}\label{eqn:dev-fin-1}
    \sum_{M \in \mathcal{L}(M_1)} |W_0^{M_\sigma}| |W_0^{G_\sigma}|^{-1} |\iota^M(\sigma)| |\iota^G(\sigma)|^{-1} \sum_{\substack{\gamma \in (M(F) \cap \mathfrak{o})_{M,S}^{K,\text{bon}} \\ \gamma_s = \sigma}} a^{\tilde{M}}(S, \dot{\widetilde{\gamma_S}}) J_{\tilde{M}}(\dot{\widetilde{\gamma_S}},f).
  \end{gather}
  où $\gamma_s = \sigma$ signifie que l'on prend les représentants admissibles ayant partie semi-simple $\sigma$.

  Traitons maintenant le côté à droite de l'assertion. Tous les regroupements ci-dessous sont justifiés par le lemme précédent. La somme sur $\gamma$ se décompose en une somme double sur les classes semi-simples et des classes unipotentes. On peut combiner la somme sur les classes semi-simples avec la somme sur $M$ et on obtient une somme sur
  $$ \Pi := \{(M,\sigma_M) : M \in \mathcal{L}(M_0), \; \sigma_M \in (M(F)), \; \sigma_M \stackrel{\text{conj}}{\sim} \sigma \text{ dans } G(F) \} $$
  suivie par une somme sur des classes unipotentes. Cette somme double est évidemment finie. De plus, $W_0^G$ opère sur $\Pi$ et on peut sommer sur le quotient $\Pi/W_0^G$ pourvu que l'on multiplie les coefficients par l'ordre du groupe d'isotropie.

  Toute classe dans $\Pi/W_0^G$ contient une paire de la forme $(M,\sigma)$ avec $M \supset M_1$ (cf. \cite[p.186]{Ar86}). Arthur en a calculé l'ordre du groupe d'isotropie (cf. \cite[pp.206-207]{Ar86}): c'est $|W_0^{M^\sigma(F)}| |W_0^{G^\sigma(F)}|^{-1}$, où
  $$ W_0^{G^\sigma(F)} := M_1^\sigma(F) \backslash N_{G^\sigma}(A_{M_1})(F). $$
  Idem pour $M$ au lieu de $G$. Donc le terme à droite de l'assertion est égal à
  \begin{gather}\label{eqn:dev-fin-2}
    \sum_{M \in \mathcal{L}(M_1)} |W_0^{M^\sigma(F)}| |W_0^{G^\sigma(F)}|^{-1} \sum_{\substack{\gamma \in (M(F) \cap \mathfrak{o})_{M,S}^{K,\text{bon}} \\ \gamma_s = \sigma}} a^{\tilde{M}}(S, \dot{\widetilde{\gamma_S}}) J_{\tilde{M}}(\dot{\widetilde{\gamma_S}},f).
  \end{gather}

  En comparant \eqref{eqn:dev-fin-1} et \eqref{eqn:dev-fin-2}, on se ramène à prouver que
  $$ |W_0^{M_\sigma}| |W_0^{G_\sigma}|^{-1} |\iota^M(\sigma)| |\iota^G(\sigma)|^{-1} = |W_0^{M^\sigma(F)}| |W_0^{G^\sigma(F)}|^{-1}, \quad M \in \mathcal{L}(M_1), $$
  ce qui est exactement (8.5) de \cite{Ar86}.
\end{proof}

Étant donné $\Delta$ un voisinage compact de $1$ dans $G(\A)^1$, posons $\tilde{\Delta} := \rev^{-1}(\Delta)$. Notons $C_{\Delta,\asp}^\infty(\tilde{G}^1)$ l'espace des fonctions dans $C_{c,\asp}^\infty(\tilde{G}^1)$ à support dans $\tilde{\Delta}$. Posons
$$ C_{\Delta,\asp}^\infty(\tilde{G}_S^1) := C_{\Delta,\asp}^\infty(\tilde{G}^1) \cap C_{c,\asp}^\infty(\tilde{G}_S^1)$$
via \eqref{eqn:completion-fonction}. On arrive ainsi au développement géométrique fin.

\begin{theorem}[cf. {\cite[9.2]{Ar86}}]\label{prop:developpement-raffine-geometrique}
  Il existe un sous-ensemble fini de places $S_{\Delta} \supset V_\mathrm{ram}$ tel que
  \begin{itemize}
    \item il existe $\Delta_S \subset G(F_S)$ tel que $\Delta = \Delta_{S_\Delta} \times K^{S_\Delta}$;
    \item pour tout $S \supset S_\Delta$ et tout $f \in C_{\Delta,\asp}^\infty(\tilde{G}_S^1)$, on a
      $$ J(f) = \sum_{M \in \mathcal{L}(M_0)} |W_0^M| |W_0^G|^{-1} \sum_{\substack{\gamma \in (M(F))_{M,S}^{K,\mathrm{bon}} \\ \gamma \leadsto \widetilde{\gamma_S} }} a^{\tilde{M}}(S, \dot{\widetilde{\gamma_S}}) J_{\tilde{M}}(\dot{\widetilde{\gamma_S}}, f) ; $$
  \end{itemize}
  les termes dans la somme ci-dessus sont nuls pour presque tout $\gamma$.
\end{theorem}
\begin{proof}
  Pour $\mathfrak{o}$ fixé, on peut toujours prendre $S$ de sorte que la condition pour \ref{prop:developpement-geometrique-o} soit satisfaite. Puisque $J = \sum_\mathfrak{o} J_\mathfrak{o}$, il suffit de montrer qu'il n'y qu'un nombre fini de $\mathfrak{o}$ qui rencontre $\Delta$, ce qu'assure \cite[9.1]{Ar86}.
\end{proof}

\bibliographystyle{abbrv-fr}
\bibliography{metaplectic}

\bigskip
\begin{flushleft}
  Wen-Wei Li \\
  Institut de Mathématiques de Jussieu \\
  175 rue du Chevaleret, 75013 Paris  \\
  France \\
  Adresse électronique: \texttt{wenweili@math.jussieu.fr}
\end{flushleft}

\end{document}